\documentclass[reqno,10.5 pt]{amsart}
\usepackage{amsmath,amssymb,latexsym}
\usepackage[colorlinks=true, linkcolor={black}]{hyperref}
\usepackage{amsthm}
\usepackage{comment}
\usepackage{color}  
\usepackage[all]{xy}
\usepackage{tikz-cd}
\usepackage{multirow}
\usepackage{multicol}
\usepackage{longtable}
\usepackage{caption}
\usetikzlibrary{cd}
\hypersetup{
    colorlinks=true, 
    linktoc=all,     
    linkcolor=blue,  
    urlcolor=black,
}

\newcommand{\F}{{\mathbb{F}}}
\newcommand{\Z}{{\mathbb{Z}}}
\newcommand{\Q}{{\mathbb{Q}}} 
\newcommand{\QQ}{\overline{{\mathbb{Q}}}}

\newcommand{\q}{{\mathfrak{q}}}
\newcommand{\p}{{\mathfrak{p}}}

\newcommand{\OO}{{\mathcal{O}}}

\newcommand{\Gal}{\mathrm{Gal}}
\newcommand{\Ker}{\mathrm{Ker}}

\makeatletter
\newcommand*{\rom}[1]{\expandafter\@slowromancap\romannumeral #1@}
\makeatother
\newcommand\isomto{\stackrel{\sim}{\smash{\longrightarrow}\rule{0pt}{0.4ex}}}
\DeclareFontFamily{U}{wncy}{}
    \DeclareFontShape{U}{wncy}{m}{n}{<->wncyr10}{}
    \DeclareSymbolFont{mcy}{U}{wncy}{m}{n}
    \DeclareMathSymbol{\Sh}{\mathord}{mcy}{"58}
		
\theoremstyle{plain}

\newtheorem{theorem}{Theorem}[section]
\newtheorem*{theorem*}{Theorem}

\newtheorem{proposition}[theorem]{Proposition}
\newtheorem{rem}[theorem]{Remark}
\newtheorem{lemma}[theorem]{Lemma}
\newtheorem{corollary}[theorem]{Corollary}
\newtheorem{defn}[theorem]{Definition}

\begin{document}
\title{$3$-Selmer groups, ideal class groups and the cube sum problem} 

\author{Somnath Jha, Dipramit Majumdar and Pratiksha Shingavekar}
\address{\begin{scriptsize}Somnath Jha, \href{mailto:jhasom@iitk.ac.in}{jhasom@iitk.ac.in}, Indian Institute of Technology Kanpur, India\end{scriptsize}}
\address{\begin{scriptsize}Dipramit Majumdar, \href{mailto:dipramit@iitm.ac.in}{dipramit@iitm.ac.in}, Indian Institute of Technology Madras, India\end{scriptsize}}
\address{\begin{scriptsize}Pratiksha Shingavekar, \href{mailto:pshingavekar@gmail.com}{pshingavekar@gmail.com}, Chennai Mathematical Institute, India\end{scriptsize}}

\begin{abstract} 
  Consider a Mordell curve \begin{small}$E_a:y^2=x^3+a$\end{small} with \begin{small}$a \in \Z$\end{small}. These curves have a rational $3$-isogeny, say \begin{small}$\varphi$.\end{small} 
We give an upper and a lower bound on the rank of the \begin{small}$\varphi$\end{small}-Selmer group of \begin{small}$E_a$\end{small} over \begin{small}$\Q(\zeta_3)$\end{small}   in terms of the  $3$-part of the ideal class group of certain quadratic extension of \begin{small}$\Q(\zeta_3)$.\end{small}  Using our bounds on the Selmer groups, we prove some cases of the rational cube sum problem. Further, using these bounds, we give  explicit families of the Mordell curves to show that for a positive proportion of \begin{small}$E_a$, ${\rm Sel}^3(E_{a}/\Q)=0$\end{small} (respectively \begin{small}${\rm Sel}^3(E_{a}/\Q)$\end{small} has \begin{small}$\mathbb F_3$\end{small}-rank $1$).
\end{abstract}

\maketitle

\section*{Introduction}\label{intro}
Given an elliptic curve \begin{small}$E$\end{small} over a number field \begin{small}$F$\end{small} and an isogeny \begin{small}$\varphi:E \rightarrow E$,\end{small} the study of \begin{small}$\varphi$\end{small}-Selmer group of \begin{small}$E$\end{small} over \begin{small}$F$\end{small} (Definition \ref{mainsel}) goes back to Cassels \cite{cass} and has many important contributions. The recent works of Bhargava et al., Mazur-Rubin and others on the \begin{small}$n$\end{small}-Selmer group of an elliptic curve have made a significant impact on the arithmetic of elliptic curves. In \cite{bes}, they determine the average size of the \begin{small}$\phi_a$\end{small}-Selmer group of \begin{small}$E_a: Y^2=X^3+a$\end{small} over \begin{small}$\Q$,\end{small} where \begin{small}$a \in \Z$\end{small} varies and \begin{small}$\phi_a$\end{small} is a rational $3$-isogeny (see \S \ref{iso}). On the other hand, giving an upper bound to an isogeny-induced Selmer group by the ideal class group has been investigated by Cassels, Brumer-Kramer  and others. More generally, using the upper bound, \cite{ss} has outlined a strategy to compute, at least in principle,  the isogeny-induced Selmer group of elliptic curves. The relation between Selmer groups and class groups has also been explored by \cite{ces, ps}. Recently, given an elliptic curve \begin{small}$E: Y^2=F(X)$\end{small} over \begin{small}$\Q$,\end{small} such that the cubic \begin{small}$F(X)$\end{small} is irreducible and has negative, square-free discriminant, Li \cite[Theorem 2.18]{li} has given an  upper and a lower bound for the $2$-Selmer group of \begin{small}$E/\Q$,\end{small} in terms of the $2$-torsion of the ideal class group of \begin{small}$\Q[X]/(F(X))$.\end{small}

Let \begin{small}$E/\Q$\end{small} be an elliptic curve and \begin{small}$\varphi=\varphi_{E,\widehat{E}}: E \rightarrow \widehat{E}$\end{small} be a rational isogeny of degree $3$. Then the equation of \begin{small}$E$\end{small} can be expressed either in the form  \begin{small}$E_a: Y^2=X^3+a$,\end{small} with \begin{small}$a \in \Z$\end{small} or \begin{small}$E_{a,b}: Y^2=X^3+a(X-b)^2$,\end{small} with \begin{small}$a,b \in \Z$\end{small} and \begin{small}$ab(4a+27b)\neq 0$\end{small} and the isogeny \begin{small}$\varphi$\end{small} between two curves of the form \begin{small}$E_a$\end{small} (respectively, \begin{small}$E_{a,b}$\end{small}) can also be described explicitly (see \S\ref{iso}). In this article, we restrict ourselves to the study of the curves \begin{small}$E_a: Y^2=X^3+a$\end{small} and their associated $3$-isogenies, denoted by \begin{small}$\phi_a$\end{small}.  Let us assume for now that \begin{small}$a$\end{small} is not a square in \begin{small}$K=\Q(\zeta_3)$.\end{small} (In \S\ref{type1curves}, we also consider  \begin{small}$a \in K^{*2}$\end{small}.) Our main goal here is to give an upper bound and a lower bound of the \begin{small}$\mathbb F_3$\end{small}-dimension of \begin{small}$\text{Sel}^{\phi_a}(E_a/K)$,\end{small}  the \begin{small}$\phi_a$\end{small}-Selmer group of \begin{small}$E_a$\end{small}  over \begin{small}$K$\end{small}  in terms of the $3$-part of the \begin{small}$S$\end{small}-class group of \begin{small}$\Q(\zeta_3, \sqrt a)$.\end{small} Here \begin{small}$S$\end{small} is a subset of the set of primes of bad reduction of \begin{small}$E_a$\end{small} over \begin{small}$K$\end{small} (Definitions \ref{defofSa}). Using such bounds for \begin{small}$\phi_a$-\end{small} and \begin{small}$\widehat{\phi}_a$-\end{small}Selmer groups, we also get the corresponding upper and lower bounds for the $3$-Selmer group \begin{small}$\text{Sel}^3(E_a/K)$.\end{small} We also indicate the  upper bounds for \begin{small}$\dim_{\F_3}{\rm Sel}^{\phi_a}(E_a/\Q)$\end{small} and \begin{small}$\dim_{\F_3}{\rm Sel}^3(E_a/\Q)$,\end{small} in terms of the $3$-class group of suitable quadratic fields, via similar methods. The curves \begin{small}$E_{a,b}$\end{small} and their $3$-isogeny Selmer groups are a topic of study for a companion article. 

The first major result of this article, giving such upper and lower bounds, is Theorem \ref{type1selmer}. We then go on to sharpen these bounds using various arithmetic properties of elliptic curves.  We present a special case of such refined bounds for the curves $E_a$ below (see Theorem \ref{type1bounds} for a detailed statement). 
\begin{theorem}[Theorem \ref{type1bounds}]
Suppose that \begin{small}$a \not\in K^{*2}$\end{small} and the set \begin{small}$S_a$\end{small} defined in \ref{defofSa} is empty. Let \begin{small}$L_a:=K[X]/(X^2-a)$\end{small} and \begin{small}$h^3_{L_a}$\end{small} be the $\F_3$-rank of the $3$-part of the class group of $L_a$. Then \begin{small}$\dim_{\F_3} {\rm Sel}^{\phi_a}(E_a/K) \in \{h^3_{L_a}, h^3_{L_a}+1\}$\end{small} and it is uniquely determined by the root number of \begin{small}$E_a/\Q$.\end{small} \qed
\end{theorem}
 Since there is a rational isogeny \begin{small}$\phi_a: E_a \rightarrow \widehat{E}_a$\end{small} of degree $3$, the residual representation \begin{small}$E_a[3](\QQ)$\end{small} is a reducible \begin{small}$G_\Q$\end{small}-module and giving a $3$-isogeny is equivalent to having  a \begin{small}$G_\Q$\end{small}-stable subgroup of order $3$ in \begin{small}$E_a(\QQ)$\end{small} and we may assume that this subgroup of order $3$ is generated by \begin{small}$(0, \pm \sqrt a)$\end{small}. Thus, corresponding to \begin{small}$\text{Sel}^{\phi_a}(E_a/K)$,\end{small} one naturally associates the quadratic \'etale algebra \begin{small}$L_a:=K[X]/(X^2-a)$.\end{small} An important  observation is  that  \begin{small}$H^1(G_K, E_a[\phi_a]) \cong  \text{Ker}(L^*_a/{L^{*3}_a} \stackrel{\text{Norm}}{\longrightarrow} K^*/{K^{*3}})$\end{small} (Prop. \ref{cohom}) and there is a similar description of the local Galois cohomology groups appearing in the definition of \begin{small}$\text{Sel}^{\phi_a}(E_a/K)$.\end{small}  It then reduces to do explicit  computations of the images of the local Kummer maps and apply some  algebraic number theory  (see \S \ref{ant},  Props. \ref{boundsgen} and  \ref{newM1M2sel}) to arrive  at the bounds. 
 
There are several novel aspects of the bounds given here (see \S\ref{type1curves} for details). For example, we introduce sharp lower (and upper) bounds for \begin{small}$\text{Sel}^{\phi_a}(E_a/K)$\end{small}  in terms of $3$-class groups via more explicit and possibly easier computations and thus we get a significant refinement and improvement on the existing results. Further, the computation of the Kummer map at \begin{small}$\p \mid 3$\end{small} for the \begin{small}$\phi_a$\end{small}-Selmer group is always more delicate and here, to avoid calculating the explicit generators for the image of the Kummer map at \begin{small}$\p \mid 3$\end{small} of \begin{small}$K$,\end{small} we do the following: For the lower bound, we  introduce two new sets \begin{small}$V_{K_\p}$\end{small} and \begin{small}$V_3$\end{small} defined using the unit group \begin{small}$\OO^*_{K_\p}$\end{small} (Definitions \ref{defnofVF}, \ref{defnofV3}). As for the upper bound, we consider the norm-$1$ subgroup of  \begin{small}$A_\p^*/{A_\p^{*3}}$\end{small} (Definition \ref{defnofAq}). Note that to give a (possibly) non-trivial lower bound for \begin{small}$\text{Sel}^{\phi_a}(E_a/K)$\end{small} using class groups, it is essential that \begin{small}$L_a$\end{small} contains \begin{small}$\zeta_3$\end{small} (see Props. \ref{boundsgen}) and hence we primarily base our results for Selmer groups  defined over \begin{small}$K$\end{small} rather than \begin{small}$\Q$.\end{small} 
 
 In order to demonstrate that the bounds obtained using $3$-class groups are sharp, we give some examples (see Examples 1-4 in \S \ref{type1boundsec}) where \begin{small}$\dim_{\F_3}{\rm Sel}^3(E_a/K)$\end{small} and \begin{small}$\dim_{\F_3}{\rm Sel}^{\phi_a}(E_a/K)$\end{small} attain their respective upper  and lower bounds as given in Theorem \ref{type1bounds}.  Also, we  have  examples in Table \ref{tab:type1examples} where our lower bound  \begin{small}$h^3_{S_a(L)}$\end{small}  for   \begin{small}$\dim_{\F_3}{\rm Sel}^3(E_a/K)$\end{small} satisfies \begin{small}$h^3_{S_a(L)} > \text{ rk }_{\Z} E_a(K) $\end{small}. Thus we get examples of non-trivial \begin{small}$\Sh(E_a/K)[3]$\end{small} (see Remark \ref{rmkforsha}). 

These bounds on Selmer groups have applications to the rational cube sum problem  (see \cite{syl}, \cite{sel}), which asks the question: which  integers \begin{small}$D$\end{small} can be expressed as a sum of cubes of two rational numbers? The elliptic curve \begin{small}$X^3+Y^3=DZ^3$\end{small} can be written in the   Weierstrass form as \begin{small}$E_{-432D^2}:y^2=x^3-432D^2$\end{small} and for a cube-free integer \begin{small}$D>2$,\end{small} it is well-known that \begin{small}${E_{-432D^2}(\Q)}_{\text{torsion}}=\{O\}$.\end{small}  Thus, a cube-free integer \begin{small}$D > 2$\end{small} satisfies \begin{small}$x^3+y^3=D$\end{small} with \begin{small}$x, y \in \Q \Leftrightarrow  \text{rk }_{\Z} E_{-432D^2}(\Q) > 0$.\end{small}  

Let \begin{small}$\ell$\end{small} be a rational prime. The special cases of cube sum problem  for \begin{small}$\ell$, $\ell^2$\end{small} are typically called the `Sylvester's conjecture' (cf. \cite{dv}). Other than  classical results  by Sylvester, Selmer, Satg{\'e} \cite{sa}, Lieman \cite{lie} etc., there are recent works including \cite{rz,dv,cst,sy,yi,kl,ms,jms,abs}. In some of these works, they consider cube sum problem for one of the pairs \begin{small}$\ell, \ell^2 $\end{small} or \begin{small}$2\ell, 2\ell^2$\end{small} or \begin{small}$3\ell, 3\ell^2$\end{small} with the prime \begin{small}$\ell $\end{small} chosen in  some specific congruence class modulo $9$. 
Note that the cube sum problem for \begin{small}$2\ell$\end{small} or \begin{small}$2\ell^2$\end{small} with prime \begin{small}$\ell \equiv 1 \pmod 3$\end{small}  is not considered in these references. 

In this article (\S \ref{cubesumsubsection}), we consider the cube sum problem  for \begin{small}$2\ell,  2\ell^2$\end{small} for all primes \begin{small}$\ell \geq 5$.\end{small} Our method and proofs are  different than the above-mentioned works. We first compute  \begin{small}$\dim_{\F_3}{\rm Sel}^{\phi_a}(E_{-432D^2}/K)$\end{small} (Theorems \ref{sylvester2}-\ref{sylvester2psqr2}). To do this, we compute the bound for \begin{small}$\dim_{\F_3}{\rm Sel}^{\phi_a}(E_{-432D^2}/K)$\end{small} using  results in \S\ref{type1curves} and then we  obtain some potential candidates as generators of the Selmer group, which we verify by hand with explicit computation. We then go on to exploit the relation between the \begin{small}$\phi_a$\end{small}-Selmer group, the $3$-Selmer group,  \begin{small}$E_{-432D^2}(K)$\end{small} and \begin{small}$\Sh(E_{-432D^2}/K)$\end{small} and their counterparts over \begin{small}$\Q$\end{small} to compute \begin{small}$\text{rk } E_{-432D^2}(\Q)$\end{small} in Corollaries \ref{cortosylvester23}-\ref{cortosylvester2psqr}. We have the following new unconditional result on cube sum problem:  
\begin{theorem}\label{sylvesterintro}
 Let \begin{small}$\ell$\end{small} be a prime such that $2$ is not a cube in \begin{small}$\F_\ell$.\end{small}
\begin{enumerate}
	\item  If \begin{small}$\ell \equiv 1,7 \pmod 9$,\end{small} then \begin{small}$2\ell$\end{small} is not a rational cube sum. 
 \item If \begin{small}$\ell \equiv 1,4 \pmod 9$,\end{small} then \begin{small}$2\ell^2$\end{small} is not a rational cube sum.
\end{enumerate}
 \noindent In both the cases, \begin{small}$\Sh(E_{-432D^2}/\Q)[3]=0$\end{small} for \begin{small}$D \in \{2\ell, \ 2\ell^2\}$.\end{small}  \qed
\end{theorem}
\noindent We further have some conditional results for \begin{small}$D=2\ell$\end{small} and \begin{small}$2\ell^2$\end{small} in Corollaries \ref{cortosylvester23} and \ref{cortosylvester2psqr}. 

We discuss further applications of our bounds on Selmer groups in this article.
In \S\ref{positiveprop}, we consider explicit families of Mordell curves \begin{small}$E_a$\end{small}, defined using certain congruence relations, and show that a positive proportion of these curves has $3$-Selmer rank over \begin{small}$\Q$\end{small} equal to $0$ and $1$, respectively.
In \S\ref{cubictwistsec}, for a suitable  \begin{small}$a$\end{small}, we study the family  \begin{small}$E_{a\ell^2}$\end{small} of cubic twists of \begin{small}$E_a$\end{small}, as the primes \begin{small}$\ell$\end{small} vary and show that for a positive density of primes \begin{small}$\ell$,\end{small}  \begin{small}$\dim_{\F_3} \text{Sel}^{\phi_a}(E_{a\ell^2}/K)$\end{small} is bounded.

The  article is structured as follows: In \S \ref{iso}, we discuss the basic setup of $3$-isogeny and give an explicit description of Selmer groups in Prop. \ref{cohom}. In \S \ref{ant}, we collect  purely algebraic results over number fields that are needed later. 
For any finite set \begin{small}$S$\end{small} of finite primes of \begin{small}$K$,\end{small}  we define two \begin{small}$\F_3$\end{small}-vector spaces  \begin{small}$M(S,a), N(S,a)$\end{small} which will be subsequently used to give the lower and upper bound respectively on \begin{small}$\dim_{\F_3} {\rm Sel}^{\phi_a}(E_{a}/K)$\end{small}. 
The local theory for the curves \begin{small}$E_a$\end{small} is contained in \S\ref{type1theory}. At first, we discuss the image of the Kummer map at primes outside $3$ and then at the prime dividing $3$. Using these and the results from \S \ref{ant}, we obtain the bounds for the Selmer group of \begin{small}$E_a$\end{small} in \S \ref{type1global} and further sharpen them in \S \ref{type1boundsec}. We discuss applications of our results including the cube sum problem  in \S \ref{applications}.  Table \ref{tab:type1examples} contains numerical examples of bounds on Selmer groups.

\noindent {\bf Acknowledgements:}\small{  We thank Ashay Burungale for discussions. The numerical examples are computed using SageMath 9.4 and Magma.}
 
\section{The Basic Set-up}\label{iso}
Let \begin{small}$E$\end{small} be an elliptic curve over \begin{small}$\Q$\end{small} and we may assume \begin{small}$E:y^2=f(x)$,\end{small} where \begin{small}$f(x) \in \Q[X]$\end{small} is monic of degree $3$. Let \begin{small}$C$\end{small} be a subgroup of order $3$ in \begin{small}$E(\QQ)$\end{small} stable under the action of \begin{small}$G_\Q:=\Gal(\QQ/\Q)$.\end{small} Then \begin{small}$C=\{O,(\alpha, \beta),(\alpha, -\beta)\}$\end{small} with \begin{small}$\alpha \in \Q$\end{small} and \begin{small}$\beta^2 \in \Q$\end{small}  \cite[pg.~3]{top}.
By a change of co-ordinate (if necessary), we may assume that \begin{small}$\alpha =0$.\end{small} If \begin{small}$E$\end{small} is now represented by \begin{small}$y^2=x^3+rx^2+sx+t$,\end{small} then \begin{small}$C=\{O,(0, \sqrt{t}),(0, -\sqrt{t})\}$\end{small} with \begin{small}$s^2=4rt$.\end{small}
If \begin{small}$s=0$,\end{small} we see that \begin{small}$r=0$,\end{small} hence \begin{small}$E$\end{small} is of the form \begin{small}$y^2=x^3+t$,\end{small}  with \begin{small}$t \in \Z$.\end{small} On the other hand, if \begin{small}$s \neq 0$,\end{small} then by change of variables, if necessary, the equation of \begin{small}$E$\end{small} takes the form \begin{small}$y^2=x^3+a(x-b)^2$,\end{small} with \begin{small}$a, b \in \Z$\end{small} and \begin{small}$ab(4a+27b) \neq 0$.\end{small} In this paper, we will study the curves \begin{small}$y^2=x^3+t$\end{small} with \begin{small}$t \in \Z$.\end{small}

 It is well-known that \begin{small}$E_a: y^2=x^3+a$\end{small} has CM by \begin{small}$\Z[\sqrt{-3}]$\end{small} and \begin{small}$j(E_a)=0$.\end{small} 
 Using \cite{velu}, we get that \begin{small}$E_a/C$\end{small} is again a curve of the same form given by \begin{small}$E_{-27a}:y^2=x^3-27a$.\end{small} Thus, we obtain a rational $3$-isogeny \begin{small}$\rho_a: E_a \to E_{-27a}$\end{small} given by \begin{small}$\rho_a(x,y) = \Big( \frac{x^3+4a}{x^2}, \frac{y(x^3-8a)}{x^3} \Big)$\end{small} and the corresponding dual isogeny \begin{small}$\widehat{\rho}_a: E_{-27a} \to E_a$\end{small} is given by \begin{small}$\widehat{\rho}_a(x,y) = \Big( \frac{x^3-108a}{9x^2}, \frac{y(x^3+216a)}{27x^3} \Big).$\end{small} Set \begin{small}$K:=\Q(\zeta_3)$\end{small} and take \begin{small}$F \in \{ K, \Q\}$.\end{small} For any  \begin{small}$c \in F^*$,\end{small} note that there is an isomorphism \begin{small}$\theta_c: E_{c^6a} \isomto E_a$\end{small} given by \begin{small}$(x,y) \mapsto (c^{-2}x,c^{-3}y)$\end{small} over \begin{small}$F$.\end{small}  Now if \begin{small}$27 \mid a$,\end{small} then the Weierstrass equation for \begin{small}${E}_{-27a}$\end{small} is not minimal. Following this, we will use the rational $3$-isogeny \begin{small}$\phi_a: E_a \to \widehat{E}_a$,\end{small} where  
 \begin{small}$$\quad \phi_a = \rho_a \ \text{  and } \ \widehat{E}_a:y^2=x^3-27a, \ \text{  if } 27 \nmid a $$ $$\text{ whereas } \phi_a = \theta_{3} \circ \rho_a \ \text{  and } \ \widehat{E}_a:y^2=x^3-\frac{a}{27}, \ \text{  if } 27 \mid a.$$\end{small} 
To ease the notation, in either cases, we will always write \begin{small}$\widehat{E}_a:y^2=x^3+a\alpha^2$,\end{small} where \begin{small}$\alpha^2=-27$,\end{small} if \begin{small}$27 \nmid a$\end{small} and \begin{small}$\alpha^2=-\frac{1}{27}$,\end{small} if \begin{small}$27 \mid a$.\end{small} Observe that \begin{small}$E_a \stackrel{\theta_\p}{\cong} \widehat{E}_a$\end{small} over \begin{small}$K$\end{small}, where \begin{small}$\p=1-\zeta_3$.\end{small}  
Thus, we get a $3$-isogeny \begin{small}$\theta_\p \circ \phi_a: E_a \to E_a$\end{small} over \begin{small}$K$,\end{small} which we simply denote by \begin{small}$\phi$\end{small} to ease the notation. Explicitly, 
\begin{small}\begin{equation}\label{eq:defofphi}
\phi(x,y) = \Big( \frac{x^3+4a}{\p^2x^2}, \frac{y(x^3-8a)}{\p^3x^3} \Big).
\end{equation}\end{small}
In \S\ref{type1global} and \S\ref{type1boundsec}, we study the Selmer groups \begin{small}${\rm Sel}^\phi(E_a/K)$\end{small} and \begin{small}${\rm Sel}^{\phi_a}(E_a/\Q)$.\end{small}

\noindent {\bf{Notation}: } For any finite set \begin{small}$X$,  $|X|$\end{small} denotes its cardinality.
\begin{small}\begin{itemize}
\item  Given a multiplicative group $G$, define the subgroup $G^3:=\{x^3\mid x \in G\}$.
\item  For abelian groups $A$, $A'$ and a homomorphism $\eta: A \rightarrow A'$, set $A[\eta]:=  \{x \in A \mid \eta(x)=0\}$. 
Further, for a given prime $p$, put  $A[p^\infty]: =\underset{r \geq 1}{\cup}A[p^r]$. 
\item  Given a field  $F$  of characteristic $0$ and an  $F$-rational isogeny $E {\overset{\varphi}{\longrightarrow}}  \widehat{E}$  between two elliptic curves  $E$, $\widehat{E}$ over  $F$,  we denote $E(\bar{F})[\varphi]$  simply by $E[\varphi]$.
\item  $F$ denotes an arbitrary number field and let $\Sigma_F$ be the set of all its finite places. For a finite subset $S$ of $\Sigma_F$, $\OO_S$ denotes the ring of $S$-integers in $F$. For $\varpi \in \Sigma_F$, $F_\varpi$ will denote the completion of $F$ at $\varpi$ and its ring of integers will be denoted by $\OO_{F_\varpi}$.
\item  $\zeta=\zeta_3=\frac{-1+\sqrt{-3}}{2}$ is the primitive third root of unity, $K=\Q(\zeta)$ and $\OO_K=\Z[\zeta]$ the ring of integers of $K$.  We will denote a finite prime of $\OO_K$ by $\q$.
\item $(\p)=(1-\zeta)$ is the unique prime of $K$ above $3$ with uniformizer $\p=1-\zeta$ in $\OO_{K_\p}$. By abuse of notation, we simply write $\p$ for the ideal $(\p)$ of $\OO_K$.
\item We set $L=L_a=\frac{K[X]}{(X^2-a)}$. Thus $L/K$ is a quadratic field extension  if $ a \notin K^{*2}$. Otherwise, $L \cong K \times K$ if $ a \in K^{*2}.$
\item Define $A := \OO_L $, the ring of integers  of $L$   if $a \notin K^{*2}.$ If $ a \in K^{*2}$, then  set $A:= \OO_K \times \OO_K$. 
\item Next, we define the quadratic \'etale $K_\q$-algebras: $L_\q=\frac{K_\q[X]}{(X^2-a)}$. 
Thus  $L_\q/K_\q$ is a quadratic field extension, if $ a \notin K_\q^{*2}$; otherwise, $L_\q \cong K_\q \times K_\q, \text{ if } a \in K_\q^{*2}$. 
\item Define $A_\q:=\begin{cases} \text{the ring of integers $\OO_{L_\q}$ of the field } L_\q, & \text{ if } a \notin K_\q^{*2},\\  \OO_{K_\q} \times \OO_{K_\q}, & \text{ if } a \in K_\q^{*2}. \end{cases}$ 
\item $N_{L/K}:L \to K$ denotes the field norm if $L$ is field and the multiplication of co-ordinates if $L \cong K \times K$. A similar notation is followed for $L_\q$ and $K_\q$.
\end{itemize}\end{small} 
   
Let \begin{small}$E \underset{\widehat{\varphi}}{\overset{\varphi}{\rightleftarrows}}  \widehat{E}$\end{small} be rational $3$-isogenies of elliptic curves \begin{small}$E/\Q$\end{small} and \begin{small}$\widehat{E}/\Q$\end{small}. 
Write \begin{small}$E: y^2=x^3+a$\end{small} and \begin{small}$\widehat{E}: y^2=x^3+a'$.\end{small} Now, let \begin{small}$F$\end{small} be any field of characteristic $0$. To \begin{small}$\varphi$\end{small} and \begin{small}$\widehat{\varphi}$\end{small} we associate two mirror quadratic  {\'e}tale \begin{small}$F$\end{small}-algebras, 
  \begin{small}$ \widetilde{F}_{\widehat{\varphi}}:= \frac{F[X]}{(X^2-a')}, \ \widetilde{F}_{\varphi}:= \frac{F[X]}{(X^2-a)}.$\end{small} If \begin{small}$F$\end{small} contains \begin{small}$\zeta$,\end{small} then  \begin{small}$a/a'=-27 \beta^2$\end{small} for some \begin{small}$\beta \in F^*$\end{small} and hence, \begin{small}$\widetilde{F}:=\widetilde{F}_{\varphi} \cong \widetilde{F}_{\widehat{\varphi}}$.\end{small}
\begin{proposition}\label{cohom}
There is an isomorphism of group schemes 
\begin{small}$E[\varphi] \cong \Ker( \mathrm{Res}_F^{\widetilde{F}_{\widehat{\varphi}}} \mu_3 \to \mu_3)$\end{small}
and an induced isomorphism
\begin{small}$$H^1(G_F, E[\varphi]) \cong (\widetilde{F}_{\widehat{\varphi}}^*/\widetilde{F}_{\widehat{\varphi}}^{*3})_{N=1},$$\end{small}
where \begin{small}$(\widetilde{F}^*_{\widehat{\varphi}}/\widetilde{F}^{*3}_{\widehat{\varphi}})_{N=1}$\end{small} denotes the kernel of the norm map \begin{small}$\overline{N}_{\widetilde{F}_{\widehat{\varphi}}/F}:\widetilde{F}^*_{\widehat{\varphi}}/\widetilde{F}^{*3}_{\widehat{\varphi}} \to F^*/F^{*3}$.\end{small} 
The same statement holds if we replace \begin{small}$(E,\varphi)$\end{small} by \begin{small}$(\widehat{E},\widehat{\varphi})$.\end{small}
\end{proposition}

\begin{proof}
The proof of \cite[Prop.~24]{bes} extends easily and the details are omitted.
\end{proof}
\vspace{2mm}

Let \begin{small}$E$, $\widehat{E}$\end{small} be elliptic curves over \begin{small}$F$\end{small} and \begin{small}$\varphi: E \to \widehat{E}$\end{small} be an isogeny over \begin{small}$F$.\end{small} For \begin{small}$T \in \{F, F_\omega\}$,\end{small} let \begin{small}$\delta_{\varphi, T}:\widehat{E}(T) \longrightarrow \widehat{E}(T)/\varphi(E(T)) \lhook\joinrel\xrightarrow{\overline{\delta}_{\varphi, T}} H^1(G_T, E[\varphi])$\end{small} be the Kummer map. Then we have the following commutative diagram:
\begin{small}\begin{center}\begin{tikzcd}
 0 \arrow[r] & \widehat{E}(F)/\varphi(E(F)) \arrow[r, "\overline{\delta}_{\varphi, F}"] \arrow[d]
& H^1(G_F, E[\varphi]) \arrow[d, "\underset{\omega \in \Sigma_F}{\prod} {\rm res}_\omega"] \arrow[r] & H^1(G_F,E)[\varphi] \arrow[d] \arrow[r] & 0\\
 0 \arrow[r] & \underset{\omega \in \Sigma_F}{\prod} \widehat{E}(F_\omega)/\varphi(E(F_\omega)) \arrow[r, "\underset{\omega \in \Sigma_F}{\prod} \overline{\delta}_{\varphi,F_\omega}"] & \underset{\omega \in \Sigma_F}{\prod} H^1(G_{F_\omega},E[\varphi]) \arrow[r] & \underset{\omega \in \Sigma_F}{\prod} H^1(G_{F_\omega},E)[\varphi] \arrow[r] & 0.
\end{tikzcd}\end{center}\end{small}

\begin{defn}\label{mainsel}
The \begin{small}$\varphi$\end{small}-Selmer group of \begin{small}$E$\end{small} over \begin{small}$F$,\end{small} \begin{small}${\rm Sel}^\varphi(E/F)$\end{small} is defined as
\begin{small}$${\rm Sel}^\varphi(E/F)= \{c \in H^1(G_F,E[\varphi]) \mid {\rm res}_\omega(c) \in {\rm Im } \ \delta_{\varphi,F_\omega} \text{ for every } \omega \in \Sigma_F \}.$$\end{small}
\end{defn}

Setting \begin{small}$\Sh(E/F):=\text{Ker}\big( H^1(G_F,E) \to \underset{\omega \in \Sigma_F}{\prod} H^1(G_{F_\omega},E) \big)$\end{small}, the Tate-Shafarevich group of \begin{small}$E$\end{small} over \begin{small}$K$,\end{small} we get the fundamental exact sequence:
\begin{small}
\begin{equation}\label{eq:defofsha}
0 \longrightarrow {\widehat{E}(F)}/{\varphi(E(F))} \longrightarrow {\rm Sel}^\varphi(E/F) \longrightarrow \Sh(E/F)[\varphi] \longrightarrow 0
\end{equation}
\end{small}
In particular, for \begin{small}$\varphi=[n]: E(\overline{F}) \to E(\overline{F})$,\end{small} we have the \begin{small}$n$\end{small}-Selmer group \begin{small}${\rm Sel}^n(E/F)$.\end{small} \\

In view of Prop. \ref{cohom}, by slight abuse of notation, we continue to denote the composite map \begin{small}$\widehat{E}(K) \longrightarrow \widehat{E}(K)/\varphi(E(K))  \lhook\joinrel\xrightarrow{  } H^1(G_K, E[\varphi]) \isomto (L^*/L^{*3})_{N=1}$\end{small} again by \begin{small}$\delta_{\varphi,K}$;\end{small} similarly  we view \begin{small}$\delta_{\varphi,K_\q}(\widehat{E}(K_\q)) \subset (L_\q^*/L_\q^{*3}\big)_{N=1}$.\end{small} 
We have the commutative diagram,
\begin{small}\begin{equation}\label{eq:localglobal}
	\begin{tikzcd}
	  {{L^*}/{L^{*3}}} \arrow[r, "\overline{N}_{L/K}"] \arrow[d] & {{K^*}/{K^{*3}}} \arrow[d]\\
		{{L_\q^*}/{L_\q^{*3}}} \arrow[r, "\overline{N}_{L_\q/K_\q}"] & {{K_\q^*}/{K_\q^{*3}.}}
	\end{tikzcd}
 \end{equation}\end{small}
  To ease the notation, via the canonical embedding \begin{small}$\iota_\q:K \to K_\q$,\end{small} we identify \begin{small}$\iota_\q(x)$\end{small} with \begin{small}$x \in K$;\end{small} similarly for \begin{small}$L$.\end{small} Hence from \eqref{eq:localglobal}, if \begin{small}$\overline{x} \in \big(L^*/L^{*3}\big)_{N=1}$,\end{small} then \begin{small}$\overline{x} \in \big(L_\q^*/L_\q^{*3}\big)_{N=1}$.\end{small}

We now have the following alternative description of \begin{small}${\rm Sel}^\varphi(E/K)$:\end{small}
\begin{small}
\begin{equation}\label{eq:newseldef}
{\rm Sel}^\varphi(E/K)=\{ \overline{x} \in (L^*/L^{*3})_{N=1} \mid \overline{x} \in \text{Im } \delta_{\varphi,K_\q} \text{ for all } \q \in \Sigma_K \}.
\end{equation}
\end{small}
Similarly, if \begin{small}$a \in K^{*2}$,\end{small} then the definition in \eqref{eq:newseldef} can be written more explicitly:
\begin{small}
\begin{equation}\label{eq:newseldefsq}
{\rm Sel}^\varphi(E/K)=\{ (\overline{x}_1,\overline{x}_2) \in (K^*/K^{*3} \times K^*/K^{*3})_{N=1} \mid (\overline{x}_1,\overline{x}_2) \in \text{Im } \delta_{\varphi,K_\q} \text{ for all } \q \in \Sigma_K \}.
\end{equation}
\end{small}
These will be our working definitions of the Selmer groups for the rest of the article.

\section{Some algebraic number theory}\label{ant}
\noindent We obtain some  algebraic results in this section which will be used later.
\begin{defn}\label{defnofAq}
We define \begin{small}$(A^*/A^{*3})_{N=1}$\end{small} to be  the kernel of the norm map \begin{small}$\overline{N}_{L/K}: A^*/A^{*3} \to \OO_K^*/\OO_K^{*3}$\end{small} induced from the norm map \begin{small}$\overline{N}_{L/K}: L^*/L^{*3} \to K^*/K^{*3}$.\end{small} The group 
\begin{small}$(A_\q^*/A_\q^{*3})_{N=1}$\end{small} is defined similarly by replacing \begin{small}$A$\end{small} and \begin{small}$\OO_K$\end{small} with \begin{small}$A_\q$\end{small} and \begin{small}$\OO_{K_\q}$,\end{small} respectively, in this definition. 
\end{defn}
In the following proposition we compute the size of \begin{small}$(A_\q^*/A_\q^{*3})_{N=1}$.\end{small}
\begin{proposition}\label{N1subgrp}
For \begin{small}$\q \nmid 3$,\end{small} we have
\begin{small}$|(A_\q^*/A_\q^{*3})_{N=1}| =       \begin{cases}
1, & \text{ if } a \notin K_\q^{*2}, \\
3, & \text{ if } a \in K_\q^{*2}.   \end{cases}$\\\end{small}
On the other hand, for \begin{small}$\p \mid 3$,\end{small} we have 
\begin{small}$|(A_\p^*/A_\p^{*3})_{N=1}| =       \begin{cases}
9, & \text{ if } a \notin K_\p^{*2}, \\
27, & \text{ if } a \in K_\p^{*2}.  \end{cases}$\end{small}
\end{proposition}

\begin{proof}
First, assume that \begin{small}$a \notin  K_\q^{*2}$\end{small} i.e. \begin{small}$\OO_{L_\q}$\end{small} is the ring of integers of the field \begin{small}$L_\q$.\end{small} In this case, cokernel of \begin{small}$N_{L_\q/K_\q}: \OO_{L_\mathfrak{q}}^* \to \OO_{K_\q}^*$\end{small} is a $2$-primary group, 
hence \begin{small}$\overline{N}_{L_\mathfrak{q}/K_\q}: \OO_{L_\mathfrak{q}}^*/\OO_{L_\mathfrak{q}}^{*3} \to \OO_{K_\q}^*/\OO_{K_\q}^{*3}$\end{small} is surjective. Thus, we get 
\begin{small}$|(A_\q^*/A_\q^{*3})_{N=1}|= \frac{|\OO_{L_\mathfrak{q}}^*/\OO_{L_\mathfrak{q}}^{*3}|}{| \OO_{K_\q}^*/\OO_{K_\q}^{*3}|}.$\end{small}

On the other hand, if \begin{small}$a \in K_\q^{*2}$,\end{small} then \begin{small}$A_\q = \OO_{K_\q} \times \OO_{K_\q}$\end{small} (see \S\ref{iso}) and the norm map \begin{small}$N_{L_\q/K_\q}: \OO_{K_\q}^* \times \OO_{K_\q}^* \to \OO_{K_\q}^*$\end{small} is given by the multiplication of co-ordinates, hence surjective. Therefore,
\begin{small}$|(A_\q^*/A_\q^{*3})_{N=1}|= \frac{|\OO_{K_\q}^*/\OO_{K_\q}^{*3} \times \OO_{K_\q}^*/\OO_{K_\q}^{*3}|}{| \OO_{K_\q}^*/\OO_{K_\q}^{*3}|}= |\OO_{K_\q}^*/\OO_{K_\q}^{*3}|.$\end{small}

Thus, to complete the proof, we compute \begin{small}$|\OO_F^*/\OO_F^{*3}|$,\end{small} with \begin{small}$F \in \{K_\q,L_\mathfrak{q}\}$\end{small} a field. Let \begin{small}$q$\end{small} and \begin{small}$|\kappa_F|$\end{small} denote the characteristic and cardinality of the residue field of \begin{small}$F$,\end{small} respectively. Then
\begin{small}$\OO_F^* \cong \frac{\Z}{(|\kappa_F|-1)\Z} \times \frac{\Z}{q^{s}\Z} \times \Z_q^{r},$\end{small}
where \begin{small}$r= [F: \Q_q]$\end{small} and \begin{small}$s=\mathrm{max}\{ t \mid F \text{ contains } q^{t}-\text{th roots of unity} \}$.\end{small}

Consider the primes \begin{small}$q \neq 3$.\end{small} If \begin{small}$q \equiv 1 \pmod{3}$,\end{small} then \begin{small}$q$\end{small} splits as a product of two (distinct) primes in \begin{small}$K$,\end{small} so \begin{small}$|\kappa_{K_\q}| = q \equiv 1 \pmod{3}$.\end{small} On the other hand, if \begin{small}$q \equiv 2 \pmod{3}$,\end{small} then \begin{small}$q$\end{small} is inert in \begin{small}$K$.\end{small} Then \begin{small}$K_\q$\end{small} is an unramified quadratic extension of \begin{small}$\Q_q$,\end{small} hence \begin{small}$|\kappa_{K_\q}| = q^2 = 1 \pmod{3}$.\end{small} Thus for \begin{small}$q \neq 3$, $|\kappa_{K_\q}| \equiv 1 \pmod{3}$.\end{small} Since \begin{small}$|\kappa_{L_\mathfrak{q}}|$\end{small} is either \begin{small}$|\kappa_{K_\q}|$\end{small} or \begin{small}$|\kappa_{K_\q}|^2$,\end{small} we see that \begin{small}$|\kappa_{L_\mathfrak{q}}| \equiv 1 \pmod{3}$\end{small} for \begin{small}$q \neq 3$.\end{small} We conclude that
\begin{small}
$$|\OO_F^*/\OO_F^{*3}| = \Big[\frac{\Z}{(|\kappa_F|-1)\Z} : 3 \frac{\Z}{(|\kappa_F|-1)\Z}\Big] \Big[\frac{\Z}{q^s\Z} : 3 \frac{\Z}{q^s\Z}\Big] \Big[ \Z_q^r : 3 \Z_q^r\Big] =3.$$
\end{small}
Now if \begin{small}$q=3$,\end{small} then \begin{small}$\OO_{K_\p}^* \cong \frac{\Z}{2\Z} \times \frac{\Z}{3\Z} \times \Z_3^2$\end{small}  implies \begin{small}$|\OO_{K_\p}^*/\OO_{K_\p}^{*3}| = 3^3.$\end{small} 
Similarly, \begin{small}$\OO_{L_\mathfrak{p}}^* \cong \frac{\Z}{8\Z} \times \frac{\Z}{3\Z} \times \Z_3^4$\end{small} implies \begin{small}$|\OO_{L_\mathfrak{p}}^*/\OO_{L_\mathfrak{p}}^{*3}| = 3^5.$\end{small}
Combining all of these, the claim follows.
\end{proof}

Recall that for a local field \begin{small}$F$\end{small} with the ring of integers \begin{small}$\OO_F$\end{small} and uniformizer \begin{small}$\pi$, $U^n_F:=1+\pi^n\OO_F$.\end{small} Given an element \begin{small}$u \in \OO_F^*$,\end{small} we let \begin{small}$\overline{u}$\end{small} denote its image in \begin{small}${\OO_F^*}/{U^n_F} \cong (\OO_F/{\pi^n\OO_F})^*$.\end{small}

\begin{defn}\label{defnofVF}
For a local field \begin{small}$F$,\end{small} define
\begin{small}$$V_F = \{ u \in \OO_F^* \mid \overline{u}  \in  {\OO_F^*}/{U^3_F} \text{ satisfies } \overline{u} = \alpha^3 \text{ for some } \alpha \in {\OO_F^*}/{U^3_F} \}.$$\end{small}
\end{defn}
It is obvious from the definition that \begin{small}$\OO_F^{*3} \subseteq V_F$\end{small} and \begin{small}$U^3_F\subseteq V_F$.\end{small}

\begin{lemma}\label{cubesinVF}
For a field \begin{small}$F \in \{K_\p,L_\mathfrak{p}\}$,\end{small} we have \begin{small}$|V_F/\OO_F^{*3}| =3$.\end{small}
\end{lemma}

\begin{proof}
Note that \begin{small}$|V_F/\OO_F^{*3}| = \frac{|\OO_F^*/\OO_F^{*3}|}{|\OO_F^*/V_F|}.$\end{small} First we compute \begin{small}$|\OO_F^*/V_F| = \frac{|\OO_F^*/U^3_F|}{|V_F/U^3_F|}.$\end{small} It is easy to see that \begin{small}$V_F/{U^3_F} \cong \big(\OO_F^*/U^3_F\big)^3$\end{small} and so \begin{small}$|\OO_F^*/V_F| = \frac{|\OO_F^*/U^3_F|}{|(\OO_F^*/U^3_F)^3|}.$\end{small} 
We know that \begin{small}$\frac{\OO_F^*}{U^3_F} \cong \frac{\Z}{(|\kappa_F|-1)\Z} \times \frac{U^1_F}{U^3_F}$.\end{small} The structure of \begin{small}$U_F^1/U_F^3$\end{small} is well known; \begin{small}$\frac{U_{K_\p}^1}{U_{K_\p}^3} \cong \big( \frac{\mathbb{Z}}{3\mathbb{Z}} \big)^{\oplus 2}$\end{small} and \begin{small}$\frac{U_{L_\mathfrak{p}}^1}{U_{L_\mathfrak{p}}^3} \cong \big( \frac{\mathbb{Z}}{3\mathbb{Z}} \big)^{\oplus 4}$.\end{small} 
Therefore, \begin{small}$|\OO_F^*/V_F| = |U^1_F/U^3_F| = \begin{cases} 9, & \text{ if } F=K_\p,\\ 81, & \text{ if } F=L_\mathfrak{p}.\end{cases}$\end{small}
We have computed \begin{small}$|\OO_F^*/\OO_F^{*3}|=\begin{cases} 27, & \text{ if } F=K_\p,\\ 243, & \text{ if } F=L_\mathfrak{p},\end{cases}$\end{small} (in the proof of Prop. \ref{N1subgrp}).
The result is immediate from these observations.
\end{proof}

\begin{defn}\label{defnofV3}
We define a subset \begin{small}$V_3$\end{small} of \begin{small}$A_\p^*$\end{small}:
\begin{small}$$V_3=\begin{cases} \{u \in V_{L_\mathfrak{p}} \mid N_{L_\mathfrak{p}/K_\p}(u) \in \OO_{K_\p}^{*3} \}, & \text{  if  } a \notin K_\p^{*2},\\
\{ (u_1,u_2) \in V_{K_\p} \times V_{K_\p}  \mid u_1u_2 \in \OO_{K_\p}^{*3}\}, & \text{  if  } a \in K_\p^{*2}. \end{cases}$$\end{small}
\end{defn}
Note that \begin{small}$A_\p^{*3} \subseteq V_3$\end{small} and \begin{small}$V_3/A_\p^{*3}  \subseteq (A_\p^*/A_\p^{*3})_{N=1}$,\end{small} whether or not \begin{small}$a$\end{small} is a square in \begin{small}$K_\p^*$.\end{small} 

\begin{proposition}\label{U3cubes}
\begin{enumerate}
\item There exists a unit \begin{small}$u = 1+ \p^3 \beta \in U^3_{K_\p}$,\end{small} such that its image \begin{small}$\overline{u}$\end{small} in \begin{small}$\OO^*_{K_\p}/{U^4_{K_\p}}$\end{small} is not a cube i.e. \begin{small}$\overline{u} \notin (\OO^*_{K_\p}/{U^4_{K_\p}})^{3}$.\end{small} In particular, this implies \begin{small}$\beta \notin \p$.\end{small}
\item Let \begin{small}$a \notin K_\p^{*2}$\end{small} and consider the field \begin{small}$L_\mathfrak{p}$.\end{small} There exists a unit \begin{small}$u \in U^3_{L_\mathfrak{p}}$,\end{small} such that \begin{small}$N_{L_\mathfrak{p}/K_\p}(u)  \notin \OO_{K_\p}^{*3}$.\end{small} In particular, we deduce that \begin{small}$V_3 \subsetneqq V_{L_\mathfrak{p}}$.\end{small}
\end{enumerate}
\end{proposition}

\proof
\begin{enumerate}
    \item  Consider the following short exact sequence with canonical maps:
\begin{small}
$$1 \longrightarrow \frac{U_{K_\p}^3}{U_{K_\p}^4} \longrightarrow \frac{\OO_{K_\p}^*/U_{K_\p}^4}{\big(\OO_{K_\p}^*/U_{K_\p}^4\big)^3} \longrightarrow \frac{\OO_{K_\p}^*/U_{K_\p}^3}{\big(\OO_{K_\p}^*/U_{K_\p}^3\big)^3} \longrightarrow 1.$$
\end{small}
Choose any unit \begin{small}$u = 1+\p^3\beta \in U_{K_\p}^3 \setminus U_{K_\p}^4$.\end{small} This implies \begin{small}$\beta \notin \p$.\end{small} Further, \begin{small}$\overline{u} \in \frac{U_{K_\p}^3}{U_{K_\p}^4}$\end{small} has order \begin{small}$o(\overline{u})=3$\end{small} as \begin{small}$\frac{U_{K_\p}^3}{U_{K_\p}^4} \cong \frac{\OO_{K_\p}}{\p} \cong \frac{\Z}{3\Z}$.\end{small} Via \begin{small}$\frac{U_{K_\p}^3}{U_{K_\p}^4} \hookrightarrow \frac{\OO_{K_\p}^*}{U_{K_\p}^4}$,\end{small} consider \begin{small}$\overline{u} \in \OO_{K_\p}^*/U_{K_\p}^4$.\end{small} As \begin{small}$\big(\OO_{K_\p}^*/U_{K_\p}^4\big)^3 \cong \frac{\Z}{2\Z}$,\end{small} we see that \begin{small}$\overline{u} \notin \big(\OO_{K_\p}^*/U_{K_\p}^4\big)^3$.\end{small}
\item First note that the map \begin{small}$x \mapsto x^2$\end{small} on \begin{small}$U^3_{K_\p}$\end{small} is surjective, since \begin{small}$U^3_{K_\p} \cong \p^3\OO_{K_\p} \cong \OO_{K_\p}$\end{small} and the multiplication by $2$ map on \begin{small}$\OO_{K_\p}$\end{small} is surjective. Now, given any \begin{small}$u \in U^3_{K_\p} \setminus U^4_{K_\p}$,\end{small} there exists an element \begin{small}$u_0 \in U^3_{K_\p} \setminus U^4_{K_\p}$\end{small} such that \begin{small}$u_0^2=u$.\end{small} Note that \begin{small}$a \notin K_\p^{*2}$\end{small} implies \begin{small}$\p$\end{small} remains inert in the field \begin{small}$L_\mathfrak{p}$\end{small} and hence \begin{small}$N_{L_\mathfrak{p}/K_\p}(u_0)=u$.\end{small} We now choose \begin{small}$u$\end{small} as in part(1) and fix \begin{small}$u_0$\end{small} so that \begin{small}$u_0^2=u$.\end{small} Note that  \begin{small}$u_0 \in U^3_{K_\p} \subset U^3_{L_\mathfrak{p}} \subset V_{L_\mathfrak{p}}$.\end{small} Then we claim that \begin{small}$N_{L_\mathfrak{p}/K_\p}(u_0)  \notin \OO_{K_\p}^{*3}$.\end{small} Indeed, if \begin{small}$u \in \OO_{K_\p}^{*3}$,\end{small} then \begin{small}$\overline{u} \in \big(\OO_{K_\p}^*/U_{K_\p}^4\big)^3$,\end{small} which is a contradiction by part(1). Hence, \begin{small}$u_0$\end{small} is the required unit and in particular, \begin{small}$u_0 \in V_{L_\mathfrak{p}} \setminus V_3$.\end{small} \qed
\end{enumerate}

\begin{lemma}\label{V3U3U4}
Let \begin{small}$a \in K_\p^{*2}$.\end{small} Then \begin{small}$\Big(\frac{A_\p^*}{A_\p^{*3}}\Big)_{N=1} \cong \frac{U^1_{K_\p}}{U^4_{K_\p}}$\end{small} and \begin{small}$\frac{V_3}{A_\p^{*3}} \cong \frac{U^3_{K_\p}}{U^4_{K_\p}}$\end{small}.
\end{lemma}
\begin{proof}
First note that we have a canonical isomorphism  \begin{small}$\Big(\frac{A_{\p}^*}{A_{\p}^{*3}}\Big)_{N=1} \cong \frac{\OO^*_{K_\p}}{\OO^{*3}_{K_\p}}$.\end{small} We know that \begin{small}$\OO_{K_\p}^* \cong \frac{\Z}{2\Z} \times U^1_{K_\p}$\end{small} and hence, \begin{small}$\frac{\OO_{K_\p^*}}{\OO_{K_\p^{*3}}} \cong \frac{U^1_{K_\p}}{(U^1_{K_\p})^3}$.\end{small} Now if we take \begin{small}$\alpha=1+\p \beta \in U^1_{K_\p}$,\end{small} then by substituting \begin{small}$\zeta=1-\p$\end{small} and \begin{small}$\p^2\zeta^2=-3$,\end{small} it is easy to see that \begin{small}$\alpha^3 \equiv 1+\p^3(\beta^3-\beta) \pmod{\p^4}$.\end{small} As \begin{small}$x \mapsto x^3$\end{small} is the identity automorphism on \begin{small}$\frac{\OO_{K_\p}}{\p\OO_{K_\p}} \cong \frac{\Z}{3\Z}$,\end{small} we see that \begin{small}$\beta^3 \equiv \beta \pmod \p$.\end{small} So, \begin{small}$\alpha^3 \equiv 1 \pmod{\p^4}$\end{small} and \begin{small}$\big(U^1_{K_\p}\big)^3 \subset U^4_{K_\p}$.\end{small} Note that \begin{small}$\Big|\frac{\OO^*_{K_\p}}{\OO^{*3}_{K_\p}}\Big|= \Big|\frac{U^1_{K_\p}}{(U^1_{K_\p})^3}\Big|=27$\end{small} and \begin{small}$\frac{U^1_{K_\p}}{U^4_{K_\p}}\cong \big(\frac{\Z}{3\Z}\big)^{\oplus 3}$,\end{small} so \begin{small}$\Big|\frac{U^1_{K_\p}}{U^4_{K_\p}}\Big|=27$.\end{small} Thus, \begin{small}$\big(U^1_{K_\p}\big)^3 = U^4_{K_\p}$\end{small} and \begin{small}$\Big(\frac{A_\p^*}{A_\p^{*3}}\Big)_{N=1} \cong \frac{U^1_{K_\p}}{U^4_{K_\p}}$.\end{small}

Recall that \begin{small}$U^3_{K_\p} \subset V_{K_\p}$.\end{small} So by the definition of \begin{small}$V_3$,\end{small} there is an obvious map \begin{small}$f:U^3_{K_\p} \to \frac{V_3}{A_\p^{*3}}$\end{small} given by \begin{small}$u \mapsto (\overline{u},\overline{u}^2)$.\end{small}   
Now, \begin{small}$\big(U^1_{K_\p}\big)^3 = U^4_{K_\p} \subset{\rm Ker } f$\end{small} and there is an induced map \begin{small}$\overline{f}: \frac{U^3_{K_\p}}{U^4_{K_\p}} \to \frac{V_3}{A_\p^{*3}}$.\end{small}  Note that the order of \begin{small}$\frac{U^3_{K_\p}}{U^4_{K_\p}}$\end{small} is $3$. Further, by Prop. \ref{U3cubes}(1), \begin{small}$\exists$\end{small} a unit \begin{small}$u \in U^3_{K_\p} \setminus U^4_{K_\p}$\end{small} such that \begin{small}$u \notin \OO_{K_\p}^{*3}$,\end{small} so  \begin{small}$\overline{f}$\end{small} is injective.  We claim the order of \begin{small}$\frac{V_3}{A_\p^{*3}}$\end{small} is $3$ and hence \begin{small}$\overline{f}$\end{small} is an isomorphism. Indeed, we have a canonical isomorphism \begin{small}$\frac{V_3}{A_\p^{*3}} \cong \frac{V_{K_\p}}{{\OO_{K_\p}^{*3}}}$\end{small} and by Lemma \ref{cubesinVF}, \begin{small}$\big|\frac{V_{K_\p}}{{\OO_{K_\p}^{*3}}}\big|=3$.\end{small}
\end{proof}

\begin{proposition}\label{sizeofV3}
We have  \begin{small}$|V_3/A_\p^{*3}| =
1$\end{small}  if \begin{small}$a \notin K_\p^{*2}$\end{small}  and \begin{small}$|V_3/A_\p^{*3}| =3$\end{small}  if  \begin{small}$a \in K_\p^{*2}$.\end{small}
\end{proposition}

\begin{proof}
If \begin{small}$a \notin K_\p^{*2}$\end{small} i.e. \begin{small}$A_\p=\OO_{L_\mathfrak{p}}$\end{small} is the ring of integers of the field \begin{small}$L_\p$,\end{small} then \begin{small}$A_\p^{*3} \subset V_3 \subsetneqq V_{L_\mathfrak{p}}$\end{small} (Prop. \ref{U3cubes}(2)). Since \begin{small}$|V_3/ A_\p^{*3}|$\end{small} divides \begin{small}$|V_{L_\mathfrak{p}}/A_\p^{*3}|=3$,\end{small} the result follows.

In the case when \begin{small}$a \in K^{*2}$,\end{small} the proof is immediate from  Lemma \ref{V3U3U4}.
\end{proof}

We make a remark regarding some notational modification.
\begin{rem}\label{valu34}
For any non-archimedean local field \begin{small}$F$\end{small} with discrete valuation \begin{small}$v$\end{small} and \begin{small}$\overline{x} \in {F}^*/{{F}^{*3}}$,\end{small}  define \begin{small}$v(\overline{x}) \in \Z/{3\Z}$\end{small} by \begin{small}$v(\overline{x}):= v(x) \pmod 3$,\end{small} for  any lift \begin{small}$x \in F^*$.\end{small} 
Similarly, we define \begin{small}$v(\overline{y}, \overline{z}):=(v(y) \pmod 3, v(z) \pmod 3) \in \Z/{3\Z} \times \Z/3\Z$\end{small} for \begin{small}$(\overline{y},\overline{z}) \in F^*/{F^{*3}} \times F^*/{F^{*3}}$.\end{small} If both \begin{small}$v(y), v(z) \equiv 0 \pmod 3$,\end{small}  then by abuse of notation, we write \begin{small}$v(\overline{y},\overline{z}) \equiv 0 \pmod{3}$\end{small}. 
\end{rem}

For a prime \begin{small}$\q \in \Sigma_K$,\end{small} we denote the induced valuation on \begin{small}$L_\q$\end{small} simply by \begin{small}$\upsilon_\q$.\end{small}

\begin{lemma}\label{norm1mult3}
For an element \begin{small}$\overline{x}$\end{small} of \begin{small}$(L_\q^*/L_\q^{*3})_{N=1}$,\end{small} one has \begin{small}$\overline{x} \in (A_\q^*/A_\q^{*3})_{N=1}$\end{small} if and only if \begin{small}$\upsilon_\q(\overline{x}) \equiv 0 \pmod 3$.\end{small}
\end{lemma}

\begin{proof}
Let us first assume \begin{small}$L_\q$\end{small} to be a field. If \begin{small}$\overline{x} \in (A_\q^*/A_\q^{*3})_{N=1}$,\end{small} then choosing any lift \begin{small}$x \in A_\q^*$\end{small} of \begin{small}$\overline{x}$,\end{small} clearly we see that \begin{small}$\upsilon_\q({x}) =0$.\end{small} On the other hand, let \begin{small}$\upsilon_\q(\overline{x})  \equiv 0 \pmod 3$\end{small} for \begin{small}$\overline{x} \in (L_\q^*/L_\q^{*3})_{N=1}$.\end{small} Then \begin{small}$\overline{x}=\overline{\q^{3n}y}$,\end{small} where \begin{small}$y \in A_q^*$\end{small} and \begin{small}$n \in \Z$\end{small}, thus, \begin{small}$\overline{x}=\overline{y}$\end{small} in \begin{small}$(L_\q^*/L_\q^{*3})_{N=1}$\end{small} and \begin{small}$\overline{y} \in A_\q^*/A_\q^{*3}$.\end{small} As \begin{small}$N_{L_\q/K_\q}(y) \in \OO_{K_q}^{*3}$,\end{small} we deduce  \begin{small}$\overline{y}  \in (A_\q^*/A_\q^{*3})_{N=1}$.\end{small}

A similar proof works for \begin{small}$L_\q \cong K_\q \times K_\q$.\end{small}
\end{proof}
\begin{defn}\label{defofM1M2}
Let \begin{small}$S$\end{small} be a finite set of finite primes of \begin{small}$K$\end{small} and recall that \begin{small}$\OO_S$\end{small} denotes the ring of \begin{small}$S$\end{small}-integers in \begin{small}$K$.\end{small}
If \begin{small}$a \notin K^{*2}$,\end{small} then \begin{small}$L=K(\sqrt{a})$.\end{small} Let \begin{small}$S(L)=\{ \mathfrak{Q} \in \Sigma_L \mid \mathfrak{Q} \cap K \in S \}$.\end{small} Further, \begin{small}$\OO_{S(L)}$\end{small} will denote the ring of \begin{small}$S(L)$\end{small}-integers in \begin{small}$L$.\end{small} Let \begin{small}$x \in L^*$\end{small} be any lift of \begin{small}$\overline{x} \in L^*/L^{*3}$.\end{small} We define the following \begin{small}$\F_3$\end{small}-vector spaces:
\begin{small}
$$M(S,a)=\{ \overline{x} \in L^*/L^{*3} : L(\sqrt[3]{x})/L \text{ is unramified and } x \in L_\q^{*3} \text{ for all } \q  \in  S \},$$
$$N(S,a)= \{ \overline{x} \in L^*/L^{*3} : (x)=I^3 \text{ for some fractional ideal } I \text{ of } \OO_{S(L)}  \}. $$
\end{small}
\end{defn}

Note that if \begin{small}$\overline{x} \in M(S,a)$,\end{small} then any lift \begin{small}$x \in L^*$\end{small} of \begin{small}$\overline{x}$\end{small} satisfies \begin{small}$(x)=I^3$\end{small} for some fractional ideal \begin{small}$I$\end{small} of \begin{small}$\OO_{S(L)}$\end{small} (\cite[Ch.~9]{was}). This implies that \begin{small}$N_{L/K}(x) \in K^{*3}$.\end{small} So, \begin{small}$M(S,a) \subset N(S,a) \subset (L^*/L^{*3})_{N=1}$.\end{small}

\begin{defn}
Let \begin{small}$S$\end{small} be a finite set of finite primes of \begin{small}$K$.\end{small} When \begin{small}$a \in K^{*2}$,\end{small} we define 
\begin{small}$$N'(S,a):=\{ (\overline{x}_1,\overline{x}_2) \in \big(K^*/K^{*3} \times K^*/K^{*3}\big)_{N=1} \mid \upsilon_\q(\overline{x}_1,\overline{x}_2) \equiv 0 \pmod 3 \ \text{ for all } \ \q \notin S \}.$$\end{small}
\end{defn}
\begin{defn}\label{defofh3}
We assume that  \begin{small}$a \notin K^{*2}$\end{small} and denote by \begin{small}$Cl_{S(L)}(L)$\end{small} the \begin{small}$S(L)$\end{small}-ideal class group of \begin{small}$L$\end{small} i.e. the class group of \begin{small}$\OO_{S(L)}$.\end{small} Further, let \begin{small}$h^3_{S(L)}$\end{small} denote the $3$-rank of \begin{small}$Cl_{S(L)}(L)$\end{small} i.e. \begin{small}$h^3_{S(L)}:= \dim_{\F_3} Cl_{S(L)}(L) \otimes_\Z \F_3 = \dim_{\F_3}Cl_{S(L)}(L)[3]$.\end{small} In particular, if \begin{small}$S(L)$\end{small} is empty, then \begin{small}$Cl_\emptyset(L)=Cl(L)$\end{small} is the ideal class group of \begin{small}$L$\end{small} and \begin{small}$h^3_L=\dim_{\F_3} Cl(L) \otimes_\Z \F_3$.\end{small}
\end{defn}

\begin{lemma}\label{hergoltzresult}
Assume that \begin{small}$a \notin K^{*2}$\end{small} and put \begin{small}$\widetilde{\Q}_{\widehat{\phi}_a}:=\frac{\Q[x]}{(x^2+a\alpha^2)}$\end{small} and \begin{small}$\widetilde{\Q}_{{\phi}_a}:=\frac{\Q[x]}{(x^2-a)}$.\end{small}
Then the $3$-rank of the class group of \begin{small}$L$\end{small} is given by \begin{small}$h^3_L=h^3_{\widetilde{\Q}_{\widehat{\phi}_a}}+h^3_{\widetilde{\Q}_{{\phi}_a}}$\end{small} (see \cite{her}).
\qed
\end{lemma}

\begin{proposition}\label{boundsgen}
We compute the dimensions of the \begin{small}$\F_3$\end{small}-modules \begin{small}$M(S,a)$\end{small} and \begin{small}$N(S,a)$\end{small} as follows:\\
If \begin{small}$a \notin K^{*2}$,\end{small} then
\begin{small}
$$\dim_{\F_3} M(S,a) = h^3_{S(L)} \ \text{ and } \ \dim_{\F_3} N(S,a)= h^3_{S(L)}+ |S(L)|+2.$$
\end{small}
On the other hand if \begin{small}$a \in K^{*2}$,\end{small} the \begin{small}$\F_3$\end{small}-dimension of \begin{small}$N'(S,a)$\end{small} is given by \begin{small}$\dim_{\F_3} N'(S,a) = |S|+1.$\end{small}
\end{proposition}
\begin{proof}
First we consider the case \begin{small}$a \notin K^{*2}$.\end{small} Recall by class field theory, \begin{small}$Cl(L) \cong Gal(H_L/L)$, 
 \end{small} where \begin{small}$H_L$\end{small} is the Hilbert class field of \begin{small}$L$\end{small} i.e. maximal abelian everywhere unramified extension of \begin{small}$L$.\end{small} Note that as \begin{small}$H_L/L$\end{small} is unramified, the decomposition subgroup of a prime \begin{small}$\mathfrak{Q}$\end{small} of \begin{small}$L$\end{small} is generated by the corresponding Frobenius element. On the other hand, \begin{small}$Cl_{S(L)}(L)$\end{small} is the quotient of the ideal class group of \begin{small}$L$\end{small} i.e. \begin{small}$Cl(L)$\end{small} by the subgroup of ideal classes generated by primes in \begin{small}$S(L)$.\end{small} Thus, using class field theory, \begin{small}$Cl_{S(L)}(L)$\end{small} can be identified with the Galois group of the maximal abelian everywhere unramified extension of \begin{small}$L$\end{small} in which every prime of \begin{small}$S(L)$\end{small} splits completely.

Recall \begin{small}$\mu_3 \subset L$\end{small} and by Kummer theory, the maximal abelian extension of exponent $3$ of \begin{small}$L$\end{small} is given by \begin{small}$L(\{\sqrt[3]{x} \mid x \in L^*\})$.\end{small}
Also by definition, for each \begin{small}$\overline{x} \in M(S,a)$, $L(\sqrt[3]{x})/L$\end{small} is everywhere unramified. Define  \begin{small}$L_1:=L(\{ \sqrt[3]{x} \mid \overline{x} \in M(S,a) \})$\end{small} and let \begin{small}$\mathcal{Q}$\end{small} be a prime in \begin{small}$L_1$\end{small} lying above a prime \begin{small}$\mathfrak{Q} \in S(L)$.\end{small} 
Then \begin{small}$L_\mathfrak{Q}(\sqrt[3]{x})=L_\mathfrak{Q}$\end{small} for all \begin{small}$\overline{x} \in M(S,a)$.\end{small} Hence, \begin{small}$L_\mathfrak{Q}( \{ \sqrt[3]{x} \mid \overline{x} \in M(S,a) \}) = L_\mathfrak{Q}$\end{small} and  \begin{small}$(L_1)_\mathcal{Q}=L_\mathfrak{Q}$.\end{small}  This implies that \begin{small}$e(\mathcal Q/\mathfrak Q)=f(\mathcal Q/\mathfrak Q)=1$\end{small} for any \begin{small}$\mathcal{Q}$\end{small} in \begin{small}$L_1$\end{small} lying above \begin{small}$\mathfrak{Q}$.\end{small} Therefore, we conclude that \begin{small}$\mathfrak Q$\end{small} splits completely in \begin{small}$L_1$.\end{small} By using the Definition \ref{defofM1M2} of \begin{small}$M(S,a)$, $L_1$\end{small} is the maximal abelian everywhere unramified extension of \begin{small}$L$\end{small} of exponent $3$ in which every prime in \begin{small}$S(L)$\end{small} splits completely. By class field theory, \begin{small}$Gal(L_1/L) \cong Cl_{S(L)}(L) \otimes_\Z \F_3.$\end{small}
As \begin{small}$\mu_3 \subset L$,\end{small} again by Kummer theory,
\begin{small}$$M(S,a) \cong {\rm Hom}(Gal(L_1/L), \mu_3) \cong \mathrm{Hom}(Cl_{S(L)}(L)/3Cl_{S(L)}(L), \F_3).$$\end{small}
Hence, \begin{small}$\dim_{\F_3} M(S,a) = \dim_{\F_3} \mathrm{Hom}(Cl_{S(L)}(L)/3Cl_{S(L)}(L), \F_3) =h^3_{S(L)}.$\end{small}

Now suppose \begin{small}$\overline{x} \in N(S,a)$\end{small} and \begin{small}$I \subset L$\end{small} is a fractional ideal of \begin{small}$\OO_{S(L)}$\end{small} such that \begin{small}$(x)=I^3$.\end{small} Consider the map \begin{small}$\rho: N(S,a) \to Cl_{S(L)}(L)[3]$\end{small} given by \begin{small}$\overline{x} \mapsto [I]$.\end{small} Then \begin{small}$\rho$\end{small} is a well-defined surjective group homomorphism.

To see the well-definedness, let \begin{small}$\overline{y} \in N(S,a)$\end{small} be such that \begin{small}$\overline{x}= \overline{y} \in L^*/L^{*3}$.\end{small} Thus \begin{small}$y=r^3x$\end{small} for some \begin{small}$r \in L^*$.\end{small} Hence, we have \begin{small}$(y)=(r^3x)=(r^3)(x)=(r)^3I^3$.\end{small} Since \begin{small}$I$\end{small} and \begin{small}$(r)I$\end{small} lie in the same ideal class in \begin{small}$Cl_{S(L)}(L)$,\end{small} we get \begin{small}$\rho(\overline{x})=\rho(\overline{y})$.\end{small} Finally,  if there are two fractional ideals \begin{small}$I$\end{small} and \begin{small}$J$\end{small} such that \begin{small}$I^3=(x)=J^3$,\end{small} then the uniqueness of prime factorization of fractional ideals in \begin{small}$L$\end{small} gives \begin{small}$I=J$.\end{small} Thus, \begin{small}$\rho$\end{small} is well-defined. Now, \begin{small}$\rho$\end{small} is a group homomorphism is clear.

Let \begin{small}$[J]$\end{small} be an ideal class in \begin{small}$Cl_{S(L)}(L)[3]$\end{small} and \begin{small}$J$\end{small} be a fractional ideal in the class. Then \begin{small}$J^3=(y)$\end{small} for some \begin{small}$y \in L^*$,\end{small} hence \begin{small}$\overline{y} \in L^*/L^{*3}$\end{small} is an element of \begin{small}$N(S,a)$\end{small} such that \begin{small}$\rho(\overline{y}) = [J]$.\end{small} This shows that the map \begin{small}$\rho$\end{small} is surjective.

Now, we compute the kernel. Note that \begin{small}$\OO_{S(L)}^*/\OO_{S(L)}^{*3} \subset N(S,a)$.\end{small} If \begin{small}$\overline{x} \in \OO_{S(L)}^*/\OO_{S(L)}^{*3}$,\end{small} then any lift \begin{small}$x$\end{small} of \begin{small}$\bar{x}$\end{small} is a unit in \begin{small}$\OO_{S(L)}$,\end{small} hence \begin{small}$(x)=\OO_{S(L)} = [ 1 ] \in Cl_{S(L)}(L)$\end{small} and so \begin{small}$\overline{x} \in \mbox{Ker}(\rho)$,\end{small} thus \begin{small}$\mbox{Ker}(\rho) \supset \OO_{S(L)}^*/\OO_{S(L)}^{*3}$.\end{small}
Conversely, if \begin{small}$\overline{x} \in \mbox{Ker}(\rho)$,\end{small} then we have the equality of \begin{small}$\OO_{S(L)}$\end{small}-fractional ideals, \begin{small}$(x)=(y)^3$,\end{small} for some \begin{small}$y \in L^*$.\end{small} Thus, there exists some \begin{small}$s \in \OO_{S(L)}^*$\end{small} such that \begin{small}$x=sy^3$,\end{small} hence, \begin{small}$\overline{x}= \overline{s}$\end{small} in \begin{small}$L^*/L^{*3}$\end{small} and as \begin{small}$s \in \OO_{S(L)}^*$,\end{small} we can view \begin{small}$\overline{s}$\end{small} in  \begin{small}$\OO_{S(L)}^*/\OO_{S(L)}^{*3}$.\end{small} This shows that \begin{small}$\mbox{Ker}(\rho) = \OO_{S(L)}^*/\OO_{S(L)}^{*3}$.\end{small}

By Dirichlet's \begin{small}$S$\end{small}-units theorem \cite[Cor.~3.7.1, \S1]{gras}, \begin{small}$\dim_{\F_3} \OO_{S(L)}^*/\OO_{S(L)}^{*3}=|S(L)|+2.$\end{small}
Hence, we deduce \begin{small}$ \dim_{\F_3} N(S,a) = \dim_{\F_3} Cl_{S(L)}(L)[3] + \dim_{\F_3} \frac{\OO_{S(L)}^*}{\OO_{S(L)}^{*3}}= h^3_{S(L)}+ |S(L)|+2.$\end{small}

When \begin{small}$a \in K^{*2}$,\end{small} we can see that \begin{small}$N'(S,a) \cong  \Big(\frac{\OO_S^*}{\OO_S^{*3}} \times \frac{\OO_S^*}{\OO_S^{*3}}\Big)_{N=1}$.\end{small} Thus we obtain \begin{small}$\dim_{\F_3} N'(S,a) = |S|+1$,\end{small} by Dirichlet's \begin{small}$S$\end{small}-units theorem \cite[Cor.~3.7.1, \S1]{gras}. 
\end{proof}

We give an alternative (equivalent) description of \begin{small}$N(S,a)$, $N'(S,a)$\end{small} and \begin{small}$M(S,a)$,\end{small} which fits well with the definition of Selmer groups given in \eqref{eq:newseldef} and \eqref{eq:newseldefsq} and will be used later to give lower and upper bounds on them.
\begin{proposition}\label{newM1M2sel}
When \begin{small}$a \notin K^{*2}$,\end{small} we have
\begin{small}
$$M(S,a)=\{ \overline{x} \in \Big(\frac{L^*}{L^{*3}}\Big)_{N=1} \mid \overline{x} \in \Big(\frac{A_\q^*}{A_\q^{*3}}\Big)_{N=1} \text{ if } \q \notin S \cup \{\p\}, \text{ } \overline{x}=\overline{1} \text{ if } \q \in S \text{ and } \overline{x} \in \frac{V_3}{A_\p^{*3}} \text{ if } \p \notin S \},$$
$$N(S,a)=\{ \overline{x} \in \Big(\frac{L^*}{L^{*3}}\Big)_{N=1} \mid \overline{x} \in \Big(\frac{A_\q^*}{A_\q^{*3}}\Big)_{N=1} \text{ for all } \q \notin S \}.$$
\end{small}
On the other hand, for \begin{small}$a \in K^{*2}$,\end{small} we have \begin{small}$$N'(S,a)=\{ (\overline{x}_1,\overline{x}_2) \in \Big(\frac{K^*}{K^{*3}} \times \frac{K^*}{K^{*3}}\Big)_{N=1} \mid (\overline{x}_1,\overline{x}_2) \in \Big(\frac{A_\q^*}{A_\q^{*3}}\Big)_{N=1} \text{ for all } \q \notin S \}.$$\end{small}
\end{proposition}

\begin{proof}
We take  \begin{small}$\overline{x} \in \big(L^*/L^{*3}\big)_{N=1}$.\end{small} Then by Lemma \ref{norm1mult3}, \begin{small}$\overline{x} \in \big(A_\q^*/A_\q^{*3}\big)_{N=1} \Leftrightarrow \upsilon_\mathfrak{q}(\overline{x}) \equiv 0 \pmod 3$.\end{small} Hence, \begin{small}$(x)=I^3$\end{small} for some fractional ideal \begin{small}$I$\end{small} of \begin{small}$\OO_{S(L)}$\end{small} if and only if \begin{small}$\overline{x} \in \big(A_\q^*/A_\q^{*3}\big)_{N=1}$\end{small} for all \begin{small}$\q \notin S$.\end{small} Thus, the two definitions of \begin{small}$N(S,a)$\end{small} in the Definition \ref{defofM1M2} and here are equivalent.

On the other hand, let \begin{small}$\overline{x} \in M(S,a) \subset \big(L^*/L^{*3}\big)_{N=1}$.\end{small} To show the equivalence of two definitions of \begin{small}$M(S,a)$\end{small} in the Definition \ref{defofM1M2} and in this Proposition, we start with a prime \begin{small}$\q \notin S \cup \{\p\}$.\end{small}  By \cite[Theorem~6.3(i), Ch.~1]{gras}, \begin{small}$L(\sqrt[3]{x})/L$\end{small} is unramified at a prime \begin{small}$\mathfrak{Q}$\end{small} dividing \begin{small}$\q$\end{small} if and only if \begin{small}$\upsilon_\mathfrak{Q}(x) \equiv 0 \pmod 3$.\end{small} By Lemma \ref{norm1mult3}, this is equivalent to  \begin{small}$\overline{x} \in \big(A_\q^*/A_\q^{*3}\big)_{N=1}$.\end{small}

Next consider \begin{small}$\q \in S$\end{small} (includes \begin{small}$\q=\p$\end{small}). Then by the Definition \ref{defofM1M2}, \begin{small}$x \in L_\q^{*3}$\end{small} or equivalently, \begin{small}$\overline{x}=\overline{1} \in (L_{\mathfrak{q}}^*/{L_{\mathfrak{q}}^{*3}})_{N=1}$.\end{small}

Finally, we consider the prime \begin{small}$\p$\end{small} where \begin{small}$\p \notin S.$\end{small} Then \begin{small}$L(\sqrt[3]{x})/L$\end{small} is unramified at any prime \begin{small}$\mathfrak{P} \mid \p$,\end{small} by the Definition \ref{defofM1M2}. Then by \cite[Theorem~6.3(ii), Ch.~1]{gras}, we obtain that \begin{small}$x \in V_{L_\mathfrak{P}}$\end{small} for any lift \begin{small}$x \in L^*$\end{small} of \begin{small}$\overline{x}$\end{small}. In particular, \begin{small}$\upsilon_\mathfrak{P}(\overline{x}) \equiv 0 \pmod 3$.\end{small} Hence, by Lemma \ref{norm1mult3}, \begin{small}$N_{L_\mathfrak{P}/K_\p}(x) \in \OO_{K_\p}^{*3}$.\end{small} Consequently, from the Definition \ref{defnofV3}, \begin{small}$\overline{x} \in V_3/A_\p^{*3}$.\end{small}
Conversely, let \begin{small}$\p \notin S$\end{small} and \begin{small}$\overline{x} \in V_3/A_\p^{*3}$.\end{small} Then again by \cite[Theorem~6.3(ii), Ch.~1]{gras}, \begin{small}$L(\sqrt[3]{x})/L$\end{small} is unramified at every prime \begin{small}$\mathfrak{P} \mid \p$.\end{small}
The result follows from these discussions.

\noindent The alternate description of \begin{small}$N'(S,a)$\end{small} can be obtained from Lemma \ref{norm1mult3} in a similar way.
\end{proof}

\section{$3$-Selmer Group}\label{type1curves}
 
\subsection{Local Theory}\label{type1theory}
In this subsection,  we give an explicit description of the image of the Kummer map \begin{small}$\delta_{\phi,K_\q}$\end{small}  for \begin{small}$E_a$\end{small} for all \begin{small}$\q$.\end{small} We start with a general result.

\begin{lemma}\label{goodreduction}
Let \begin{small}$\varphi: E \to \widehat{E}$\end{small} be a $3$-isogeny defined over \begin{small}$K$.\end{small} Recall from \S \ref{iso} that \begin{small}$\delta_{\varphi,K_\q}(E(K_\q)) \subset (L_\q^*/L_\q^{*3})_{N=1}$.\end{small}
If \begin{small}$E$\end{small} (and hence \begin{small}$\widehat{E}$\end{small}) has good reduction at a prime \begin{small}$\q \nmid 3$\end{small} in \begin{small}$K$,\end{small} then for any \begin{small}$P \in \widehat{E}(K_\q)$, \ $v_\q(\delta_{\varphi,K_\q}(P)) \equiv 0 \pmod 3$.\end{small} In particular, \begin{small}$\delta_{\varphi,K_\q}(\widehat{E}(K_\q)) \subset \big(A_\q^*/A_\q^{*3}\big)_{N=1}$.\end{small}
\end{lemma}

\begin{proof}
Since \begin{small}$E$\end{small} has good reduction at \begin{small}$\q \nmid 3$,\end{small} for \begin{small}$P \in {\widehat{E}(K_\q)}$,\end{small} we have \begin{small}$\delta_{\varphi,K_\q}(P) \in \mbox{Ker}\big(H^1(G_{K_\q},E[\varphi]) \overset{\mbox{res}}{\longrightarrow} H^1(G_{{K}_\q^{unr}},E[\varphi])\big)$,\end{small} by \cite[Lemma~3.1]{ss}, where \begin{small}$K_\q^{unr}$\end{small} is the maximal unramified extension of \begin{small}$K_\q$.\end{small} By Prop. \ref{cohom}, we have \begin{small}$H^1(G_{K_\q},E[\varphi]) \cong (L_\q^*/L_\q^{*3})_{N=1}$\end{small} and \begin{small}$H^1(G_{{K}_\q^{unr}},E[\varphi]) \cong \big((L_\q^{unr})^*/(L_\q^{unr})^{*3}\big)_{N=1}$.\end{small} Via these isomorphisms, we identify \begin{small}$H^1(G_{K_\q},E[\varphi])$\end{small} and \begin{small}$H^1(G_{{K}_\q^{unr}}, E[\varphi])$\end{small} with their respective images in \begin{small}$L_\q^*/L_\q^{*3}$\end{small} and \begin{small}$(L_\q^{unr})^*/(L_\q^{unr})^{*3}$.\end{small} Then we see that, if \begin{small}$P \in {\widehat{E}(K_\q)}$,\end{small} then \begin{small}$\delta_{\varphi,K_\q}(P) \in \mbox{Ker}\big(L_\q^*/L_\q^{*3} \overset{\mbox{inf}}{\longrightarrow} (L_\q^{unr})^*/(L_\q^{unr})^{*3}\big)$,\end{small} where  `inf' is induced by the canonical injection \begin{small}$L_\q \hookrightarrow L_\q^{unr}$.\end{small}

Now for any \begin{small}$\overline{x} \in L_\q^*/L_\q^{*3}$\end{small} and its lift \begin{small}$x \in L_\q^*$,\end{small} we have \begin{small}$\overline{x} \in \mbox{Ker}(\mbox{inf}) \Leftrightarrow x \in L_\q^* \cap {(L_\q^{unr})}^{*3} \Leftrightarrow L_\q(\sqrt[3]{x})/L_\q$\end{small} is unramified \begin{small}$\Leftrightarrow v_\q(x) \equiv 0 \pmod 3$.\end{small} The last implications follow from \cite[Theorem~6.3(i), \S1]{gras} and 
thus, \begin{small}$v_\q(\delta_{\varphi,K_\q}(P)) \equiv 0 \pmod 3$.\end{small} That \begin{small}$\delta_{\varphi,K_\q}(\widehat{E}(K_\q)) \subset \big(A_\q^*/A_\q^{*3}\big)_{N=1}$,\end{small} follows from Lemma \ref{norm1mult3}.
\end{proof}

\begin{rem}
Assume that \begin{small}$\q \nmid 3$\end{small} and the Tamagawa numbers \begin{small}$c_\q(E)$, $c_\q(\widehat{E})$\end{small} are not divisible by $3$. Then the image of the Kummer map \begin{small}$\delta_{\varphi,K_\q}(\widehat{E}(K_\q))$\end{small} in \begin{small}$H^1(G_{K_\q},E[\varphi])$\end{small} is equal to the unramified subgroup
\begin{small}$H^1_{unr}(G_{K_\q},E[\varphi])=\mbox{Ker}\big(H^1(G_{K_\q},E[\varphi]) \to H^1(I_\q,E[\varphi])\big)$\end{small} (see \cite[Lemma~4.5]{ss}). 
\end{rem}

For \begin{small}$P \in \widehat{E}(K_\q)$\end{small}, \begin{small}$\delta_{{\varphi},K_\q}(P) = \overline{t} \in L_\q^*/L_\q^{*3}$\end{small}, if \begin{small}$L_\q$\end{small} is a field and \begin{small}$\delta_{{\varphi},K_\q}(P)=(\overline{t_1},\overline{t_2})$,\end{small} if \begin{small}$L_\q \cong K_\q \times K_\q$.\end{small}  
To ease the notation, we write \begin{small}$\delta_{{\varphi},K_\q}(P)=t$\end{small} or \begin{small}$(t_1,t_2)$,\end{small} respectively, in \S\ref{type1theory}. 
The following construction  follows Cassels \cite[\S14, \S15]{cass} and the proofs are omitted.
\begin{proposition}\label{type1}
Let \begin{small}$\phi: E_a \to {E}_a$\end{small} be the \begin{small}$K$\end{small}-isogeny defined in \eqref{eq:defofphi}. We have the Kummer map \begin{small}$\delta_{\phi,K_\q}: E_a(K_\q) \to  (L_\q^*/L_\q^{*3})_{N=1}$\end{small} with \begin{small}$\Ker(\delta_{\phi,K_\q}) = {\phi}(E_a(K_\q))$.\end{small} Let \begin{small}$P=(x(P),y(P)) \in E_a(K_\q)$.\end{small}
\begin{small}\\
$\text{If } L_\q \text{ is a field, then } \delta_{{\phi},K_\q} \text{ is given by }
\delta_{{\phi},K_\q}(P) = \begin{cases}
1,               & \text{ if } P = O,\\
y(P) - \sqrt{a}, & \text{ otherwise.}
\end{cases}$\\
$\text{If } L_\q \cong K_\q \times K_\q, \text{ then } \delta_{{\phi},K_\q} \text{ is given by } 
\delta_{{\phi},K_\q}(P) = \begin{cases}
(1,1),                                  & \text{if } P = O,\\
(\frac{1}{2 \sqrt{a} }, \ 2\sqrt{a}),     & \text{if } P = (0, \sqrt{a}),\\
(-2{\sqrt{a}}, \ -\frac{1}{2 {\sqrt{a}}}),    & \text{if } P = (0, -\sqrt{a}),\\
({y(P) - \sqrt{a}}, \ {y(P) + \sqrt{a}}),     & \text{if } P \notin E_a[{\phi}](K_\q). \qed
\end{cases}$\end{small}
\end{proposition}
Let \begin{small}$F$\end{small} be a local field of characteristic $0$ and \begin{small}$\omega$\end{small} be a uniformizer of  \begin{small}$\OO_F$\end{small}. Suppose \begin{small}$\varphi: E \to \widehat{E}$\end{small} be an isogeny of elliptic curves and \begin{small}$c_\omega(E)$\end{small} and \begin{small}$c_\omega(\widehat{E})$\end{small} denote the Tamagawa numbers of \begin{small}$E$\end{small} and \begin{small}$\widehat{E}$\end{small} at  \begin{small}$\omega$\end{small}. By a theorem of Schaefer \cite[Lemma~3.8]{sch}, we know that
\begin{small}\begin{equation}\label{eq:schaeferfor}
|\widehat{E}(F)/\varphi(E(F))|  = \frac{||\varphi'(0)||_\omega^{-1} \times |E(F)[\varphi]| \times c_\omega(\widehat{E}) }{c_\omega(E)},
\end{equation}\end{small}
where \begin{small}$||.||_\omega$\end{small} is the norm in \begin{small}$F$\end{small} and \begin{small}$\varphi'(0)$\end{small} is the leading coefficient of the power series representation of \begin{small}$\varphi$\end{small} on the formal group of \begin{small}$\widehat{E}$\end{small} \cite[\S4]{sil1}.
Using \eqref{eq:schaeferfor}, we will compute \begin{small}$|\widehat{E}(K_\q)/\varphi(E(K_\q))|$\end{small} for all \begin{small}$\q$.\end{small}

\subsubsection{\bf {Local theory at primes $\q \nmid 3$}}
In this subsection, we assume that \begin{small}$\q$\end{small} is a prime of \begin{small}$K=\Q({\zeta})$\end{small} not dividing $3$.
\begin{lemma}\label{charnot3}
For  \begin{small}$\q \nmid 3$,\end{small}  we have 
\begin{small}$\Big|\frac{E_a(K_\q)}{\phi(E_a(K_\q))}\Big|  = |E_a[\phi](K_q)| = \begin{cases}
1, & \text{ if } a \notin K_\q^{*2} \\
3, & \text{ if } a \in K_\q^{*2}. \end{cases} $\end{small}
\end{lemma}

\begin{proof}
We have \begin{small}$|E_a(K_\q)/\phi(E_a(K_\q))| = ||\phi'(0)||_\q^{-1} \times |E_a[\phi](K_\q)|$\end{small} (by \eqref{eq:schaeferfor}).
Note that \begin{small}$||\phi'(0)||_\q=1$\end{small} (\cite[Comment after Lemma~3.8]{sch}) for all \begin{small}$\q$\end{small} which do not divide the degree of the isogeny \begin{small}$\phi$.\end{small} Also, it is easy to see that\\
\begin{small}$|E_a[\phi](K_\q)| = \begin{cases} 1, & \text{ if } a \notin K_\q^{*2}\\ 3, & \text{ if } a \in K_\q^{*2}. \end{cases}$\end{small}
Hence the claim follows.
\end{proof}
\begin{corollary}\label{kumcharnot3}
When \begin{small}$a \notin K_\q^{*2}$,\end{small} it follows from Prop. \ref{N1subgrp} and Lemma \ref{charnot3} that \begin{small}$\delta_{\phi, K_\q}(E_a(K_\q)) \cong \frac{E_a(K_\q)}{\phi(E_a(K_\q))} \cong (A_\q^*/A_\q^{*3})_{N=1} = \{1\}.$\end{small}
\qed
\end{corollary}

\begin{proposition}\label{kumcharnot3sqr}
Consider a prime \begin{small}$\q \nmid 3$\end{small} of \begin{small}$K$\end{small} such that \begin{small}$a \in K_\q^{*2}$.\end{small} Then for \begin{small}$P \in E_a(K_\q)$:\end{small} 
\begin{enumerate}
    \item If \begin{small}$\upsilon_{\q}(4a) \equiv 0 \pmod{6}$,\end{small} then \begin{small}$\upsilon_{\q}(\delta_{\phi, K_\q}(P)) \equiv 0 \pmod 3.$\end{small} In fact, we have \begin{small}$\delta_{\phi, K_\q}(E_a(K_\q)) = (A_\q^*/A_\q^{*3})_{N=1}.$\end{small}
  
\item If \begin{small}$\upsilon_{\q}(4a) \not\equiv 0 \pmod{6}$,\end{small} then \begin{small}$\delta_{\phi, K_\q}(E_a(K_\q)) \cap (A_\q^*/A_\q^{*3})_{N=1}=\{1\}.$\end{small}
\end{enumerate}
\end{proposition}

\proof
As \begin{small}$a \in K_\q^{*2}$,\end{small} we have \begin{small}$E_{a}[{\phi}](K_\q) = \{ O, (0,  \sqrt{a} ), (0,-\sqrt{a})\}$\end{small} and \begin{small}$L_\q \cong K_\q \times K_\q$.\end{small}
\begin{enumerate}
    \item Consider the case \begin{small}$\upsilon_{\q}(4a) \equiv 0 \pmod{6}$.\end{small} We can assume \begin{small}$P \neq O$.\end{small} For \begin{small}$P=(0, \sqrt{a})$,\end{small} observe that \begin{small}$\upsilon_\q(2\sqrt{a}) \equiv 0 \pmod 3$.\end{small} Thus from Prop. \ref{type1}, \begin{small}$\upsilon_\q(\delta_{\phi, K_\q}(P))=\upsilon_\q \big(\frac{1}{2\sqrt{a}},2\sqrt{a} \big) \equiv 0 \pmod 3.$\end{small} The argument for \begin{small}$(0,-\sqrt{a})$\end{small} is similar.
	Next, consider \begin{small}$P \in E_a({K_\q})\setminus E_a[\phi](K_\q)$.\end{small} If \begin{small}$\upsilon_\q(y(P)) \neq \upsilon_\q(\sqrt{a})$,\end{small} then \begin{small}$3 \upsilon_\q(x(P)) = 2 \text{ min}\{ \upsilon_\q(y(P)), \upsilon_\q(\sqrt{a})  \} =2 \upsilon_\q (y(P) \pm \sqrt{a})$.\end{small} Hence again from Prop. \ref{type1}, \begin{small}$\upsilon_\q(\delta_{\phi, K_\q}(P)) \equiv 0 \pmod 3$.\end{small} Thus, it reduces to consider \begin{small}$\upsilon_\q(y(P)) = \upsilon_\q(\sqrt{a})$.\end{small}
	
	First we assume that \begin{small}$\q \neq 2$.\end{small} Then \begin{small}$\upsilon_{\q}(4a) = \upsilon_{\q}(a) \equiv 0 \pmod{6}$.\end{small} But \begin{small}$a$\end{small} is sixth-power free by assumption, whence \begin{small}$\upsilon_{\q}(a)=0$.\end{small} Then both \begin{small}$\sqrt{a}, \enspace y(P) \in \OO^*_{K_\q}$.\end{small} Consequently, at least one of \begin{small}$\upsilon_{\q}(y(P)-\sqrt{a})$\end{small} and \begin{small}$\upsilon_{\q}(y(P)+\sqrt{a})$\end{small} has be equal to $0$.
Further, using \begin{small}$3\upsilon_{\q}(x(P))=\upsilon_{\q}(y(P)-\sqrt{a})+\upsilon_{\q}(y(P)+\sqrt{a})$,\end{small} it follows that \begin{small}$\upsilon_{\q}(y(P) \pm \sqrt{a}) \equiv 0 \pmod{3}$.\end{small} Thus, \begin{small}$\upsilon_\q(\delta_{\phi, K_\q}(P)) \equiv 0 \pmod 3$.\end{small} 

Finally, we consider \begin{small}$\q = 2$.\end{small} By the given hypothesis and using \begin{small}$a$\end{small} is sixth-power free, we deduce \begin{small}$\upsilon_{2}({a})= 4$.\end{small} Observe that \begin{small}$a \in K_2^{*2}$\end{small} if and only if \begin{small}$a \in \Q_2^{*2}$\end{small} or \begin{small}$-3a \in \Q_2^{*2}$.\end{small} From this, it is easy to see that we can write \begin{small}$a = 2^4(4t+1)$, $t \in \Z$.\end{small} For \begin{small}$E_a: y^2=x^3+16(4t+1)$,\end{small} one has \begin{small}$\Delta_{E_a}=-2^{12} \cdot 3^3 \cdot (4t+1)^2$\end{small} and by Tate's algorithm, this is not a minimal model of the curve. The change of variables \begin{small}$(x,y)=(4x',8y'+4)$\end{small} transforms \begin{small}$E_a$\end{small} into \begin{small}$E_a':y'^2+y'=x'^3+t$\end{small} with  \begin{small}$\Delta_{E'_a}=-3^3 \cdot (4t+1)^2$,\end{small} which is clearly not divisible by $2$. Thus, \begin{small}$E_a$\end{small} has good reduction at $2$ and by Lemma \ref{goodreduction}, we obtain \begin{small}$\upsilon_2(\delta_{\phi, K_2}(P)) \equiv 0 \pmod 3$.\end{small}

By Lemma \ref{norm1mult3}, it follows that for all \begin{small}$\q \nmid 3$, $\delta_{\phi, K_\q}(E_a(K_\q)) \subseteq (A_\q^*/A_\q^{*3})_{N=1}$,\end{small} whenever \begin{small}$\upsilon_{\q}(4a) \equiv 0 \pmod{6}$.\end{small}
The equality follows from Prop. \ref{N1subgrp} and Lemma \ref{charnot3}.\\

\item Assume now that \begin{small}$\upsilon_{\q}(4a) \not\equiv 0 \pmod{6}$.\end{small} As explained above, \begin{small}$|(A_\q^*/A_\q^{*3})_{N=1}|=|\delta_{\phi, K_\q}(E_a(K_\q))|$ $ =3$,\end{small} for \begin{small}$a \in K_\q^{*2}$.\end{small}
Now, if \begin{small}$\upsilon_{\q}(4a) \equiv 2 \pmod{6}$,\end{small} then \begin{small}$\upsilon_\q(2\sqrt{a})=1$,\end{small} as \begin{small}$a$\end{small} is sixth-power free. Hence, for \begin{small}$O \neq P \in E_a[\phi](K_\q)$\end{small} we get that \begin{small}$\upsilon_\q(\delta_{\phi,K_\q}(P)) = 1.$\end{small} So by Lemma \ref{norm1mult3}, \begin{small}$\delta_{\phi, K_\q}(E_a(K_\q)) \cap (A_\q^*/A_\q^{*3})_{N=1} = \{1\}$\end{small} in this case.
On the other hand, if \begin{small}$\upsilon_{\q}(4a) \equiv 4 \pmod{6}$,\end{small} then \begin{small}$\upsilon_\q(2\sqrt{a})=2$.\end{small} Hence, for \begin{small}$O \neq P \in E_a[\phi](K_\q)$, $\upsilon_\q(\delta_{\phi,K_\q}(P)) = 2.$\end{small} So we arrive at the same conclusion. \qed
\end{enumerate}

\vspace{3mm}
\subsubsection{\bf{Local theory at the prime $\p \mid 3$} }
By Tate's algorithm, we see that, \begin{small}$E_a:y^2=x^3+a$\end{small} is a minimal Weierstrass equation over \begin{small}$\Z_3[\p]$,\end{small} where \begin{small}$\p =1- \zeta$.\end{small}  We compute \begin{small}$\phi'(0)$\end{small} using formal groups of elliptic curves, following \cite[pg.~92]{sch}. 
Let \begin{small}$\phi(x,y)=(X,Y)$.\end{small} Write \begin{small}$Z:=-{X}/{Y}$\end{small} as a power series in \begin{small}$z:=-{x}/{y}$\end{small} as follows:
\begin{small}
\begin{equation}\label{eq:formlgrp}
Z=-\frac{X}{Y}=-\frac{\frac{x^3+4a}{\p^2x^2}}{\frac{y(x^3-8a)}{\p^3x^3}}=-\p\frac{x}{y}\Big(1+\frac{12a}{x^3-8a}\Big)=\p z \Big(1+ \frac{12a}{x^3-8a}\Big)=\p z+O(z^2),
\end{equation}
\end{small}
whence \begin{small}$\phi'(0)=\p$\end{small} and \begin{small}$||\phi'(0)||_\p^{-1}=3$.\end{small}

\begin{lemma}\label{char3}
For \begin{small}$\p \mid 3$,\end{small} we have
\begin{small}
$\Big|\frac{E_a(K_\p)}{\phi(E_a(K_\p))}\Big| = ||\phi'(0)||_\p^{-1} \times |E_a[\phi](K_\p)|  = \begin{cases}
3, & \text{ if } a \notin K_\p^{*2} \\
9, & \text{ if } a \in K_\p^{*2}.
\end{cases}$
\end{small}
\end{lemma}

\begin{proof}
Using \eqref{eq:schaeferfor}, we have \begin{small}$|E_a(K_\p)/\phi(E_a(K_\p))| = ||\phi'(0)||_\p^{-1} \times |E_a[\phi](K_\p)|$.\end{small}
From \eqref{eq:formlgrp}, we know that \begin{small}$||\phi'(0)||_\p^{-1}=3$\end{small} and 
\begin{small}$|E_a[\phi](K_\p)| = \begin{cases} 1, & \text{ if } a \notin K_\p^{*2},\\ 3, & \text{ if } a \in K_\p^{*2}. \end{cases}$\end{small}
Hence the claim follows.
\end{proof}

\begin{lemma}\label{kumchar3notsqr}
Assume that \begin{small}$a \notin K_\p^{*2}$.\end{small} Then we have  \begin{small}$\delta_{\phi, K_\p}(E_a(K_\p)) \subseteq (A_\p^*/A_\p^{*3})_{N=1}.$\end{small} Moreover, we have that the index \begin{small}$ [(A_\p^*/A_\p^{*3})_{N=1} : \delta_{\phi, K_\p}(E_a(K_\p))]=3.$\end{small}
\end{lemma}

\begin{proof}
Let \begin{small}$\mathfrak{P}$\end{small} be a prime of \begin{small}$L$\end{small} dividing \begin{small}$\p$.\end{small}
Note that \begin{small}$L_\mathfrak{p}/K_\p$\end{small} is an unramified quadratic field extension and \begin{small}$\upsilon_{\mathfrak{P}}(x)=\upsilon_\p(x)$\end{small} for every \begin{small}$x \in K_\p$.\end{small} 
We may assume that \begin{small}$P=(x(P),y(P)) \neq O$.\end{small}
If \begin{small}$\upsilon_{\mathfrak{P}}(y(P)) \neq \upsilon_{\mathfrak{P}}(\sqrt{a})$,\end{small} then \begin{small}$3 \ \upsilon_{\mathfrak{P}}(x(P)) = \upsilon_{\mathfrak{P}} (y(P) + \sqrt{a}) +\upsilon_{\mathfrak{P}} (y(P) - \sqrt{a}) =  2 \ \upsilon_{\mathfrak{P}} (y(P) \pm \sqrt{a})$\end{small} whence, \begin{small}$\upsilon_{\mathfrak{P}}(y(P) \pm \sqrt{a}) \equiv 0 \pmod{3}$.\end{small} So it reduces to the case when \begin{small}$\upsilon_{\mathfrak{P}}(y(P)) = \upsilon_{\mathfrak{P}} (\sqrt{a}) =n$,\end{small} say. Once we show \begin{small}$\upsilon_{\mathfrak{P}}(y(P) \pm \sqrt{a}) \equiv 0 \pmod 3$,\end{small} then \begin{small}$\delta_{\phi, K_\p}(P) \in (A_\p^*/A_\p^{*3})_{N=1}$\end{small} follows from Lemma \ref{norm1mult3}.

Write \begin{small}$y(P)= \p^n \alpha$, \ $ a= \p^{2n} \beta$\end{small} with \begin{small}$\alpha \in \OO_{K_\p}^*$\end{small} and \begin{small}$\beta \in \OO_{K_\p}^* \setminus \OO_{K_\p}^{*2}$.\end{small} Let \begin{small}$\overline{\alpha}, \ \overline{\beta}$\end{small} be the corresponding images in  the residue field \begin{small}$\kappa_\p$.\end{small} Note \begin{small}$\overline{\alpha}^2 \in \kappa_\p^{*2}$\end{small} and by Hensel's lemma \begin{small}$\overline{\beta} \notin \kappa_\p^{*2}$.\end{small} It follows that \begin{small}$\alpha^2 - \beta \in \OO_{K_\p}^*$.\end{small} Hence, \begin{small}$\upsilon_{\mathfrak{P}}(\alpha - \sqrt{\beta}) +  \upsilon_{\mathfrak{P}}(\alpha + \sqrt{\beta}) =0$\end{small} and thus \begin{small}$\upsilon_{\mathfrak{P}}(\alpha \pm \sqrt{\beta})=0$.\end{small} Therefore, we deduce that \begin{small}$\upsilon_{\mathfrak{P}}(y(P) \pm \sqrt{a})=n$.\end{small} Since \begin{small}$3 \upsilon_{\mathfrak{P}}(x(P)) = 2n$,\end{small} we get \begin{small}$n \equiv 0 \pmod{3}$.\end{small}

The index \begin{small}$[(A_\p^*/A_\p^{*3})_{N=1}:\delta_{\phi, K_\p}(E_a(K_\p))]=3$\end{small} follows from Prop. \ref{N1subgrp} and Lemma \ref{char3}.
\end{proof}
\begin{rem}\label{anotsqV3}
If \begin{small}$a \notin K_\p^{*2}$,\end{small} then by Prop. \ref{sizeofV3}, \begin{small}$\{1\}=V_3/A_\p^{*3} \subseteq \delta_{\phi, K_\p}(E_a(K_\p))$.\end{small}
\end{rem}
Finally, we consider the case \begin{small}$a \in K_\p^{*2}$.\end{small} 

\begin{proposition}\label{propos5.7}
Assume that \begin{small}$a \in K_\p^{*2}$.\end{small} We have 
\begin{enumerate}
	\item If \begin{small}$\upsilon_{\p}(4a) \equiv 0 \pmod{6}$,\end{small} then we have that \begin{small}$\delta_{\phi, K_\p}(E_a(K_\p)) \subseteq (A_\p^*/A_\p^{*3})_{N=1}$.\end{small} Moreover, one has that the index \begin{small}$[(A_\p^*/A_\p^{*3})_{N=1} : \delta_{\phi, K_\p}(E_a(K_\p))]=3.$\end{small}
	\item If \begin{small}$\upsilon_{\p}(4a) \not\equiv 0 \pmod{6}$,\end{small} then we have that  \begin{small}$\delta_{\phi, K_\p}(E_a(K_\p)) \not\subset (A_\p^*/A_\p^{*3})_{N=1}$.\end{small} Further, one has that the intersection \begin{small}$|\delta_{\phi, K_\p}(E_a(K_\p)) \cap (A_\p^*/A_\p^{*3})_{N=1}|=3.$\end{small}
\end{enumerate}
\end{proposition}
\proof
Observe that \begin{small}$\upsilon_{\p}(4a) \equiv 0 \pmod{6}$\end{small} if and only if \begin{small}$\p \nmid a$.\end{small} Also recall \begin{small}$L_\p \cong K_\p \times K_\p$\end{small} and \begin{small}$|(A_\p^*/A_\p^{*3})_{N=1}|=27$,\end{small}  \begin{small}$|\delta_{\phi, K_\p}(E_a(K_\p))|=9$\end{small} (by Prop. \ref{N1subgrp} and Lemma \ref{char3}, resp.)
\begin{enumerate}
    \item The proof of this case is similar to Prop. \ref{kumcharnot3sqr}(1) for primes not dividing $2$.
Once again the index \begin{small}$[(A_\p^*/A_\p^{*3})_{N=1} : \delta_{\phi, K_\p}(E_a(K_\p))]$\end{small} is computed directly from Prop. \ref{N1subgrp} and Lemma \ref{char3}.
\item In this case, \begin{small}$\p \mid a$.\end{small} Then for the point \begin{small}$P=(0,\sqrt{a})$,\end{small}  we have  \begin{small}$\upsilon_{\p}(\delta_{\phi, K_\p}(P)) = \upsilon_\p\big((\frac{1}{2\sqrt{a}}, 2\sqrt{a})\big) \not\equiv 0 \pmod{3}$\end{small} and so \begin{small}$\delta_{\phi, K_\p}(E_a(K_\p)) \not\subseteq (A_\p^*/A_\p^{*3})_{N=1}$.\end{small} Note that both \begin{small}$\delta_{\phi, K_\p}(E_a(K_\p))$\end{small} and \begin{small}$(A_\p^*/A_\p^{*3})_{N=1}$\end{small} are subgroups of \begin{small}$(L_\p^*/L_\p^{*3})_{N=1}$; \end{small} the last group being isomorphic to \begin{small}$K_\p^*/K_\p^{*3}$\end{small}  has order $81$. Then looking at their respective cardinalities, 
	we deduce that \begin{small}$|\delta_{\phi, K_\p}(E_a(K_\p)) \cap (A_\p^*/A_\p^{*3})_{N=1}|=3$.\end{small} \qed
\end{enumerate}

\begin{proposition}\label{V3fortype1}
Suppose that \begin{small}$a \in K_\p^{*2}$.\end{small} If \begin{small}$\upsilon_\p(4a) \equiv 0 \pmod 6$,\end{small} then \begin{small}$V_3/A_\p^{*3} \subset \delta_{\phi, K_\p}(E_a(K_\p))$.\end{small}
\end{proposition}
\begin{proof}
Recall from  the Definition \ref{defnofV3} that
\begin{small}$V_3 = \{ (u_1,u_2) \in V_{K_\p} \times V_{K_\p}  \mid u_1u_2 \in \OO_{K_\p}^{*3} \}$.\end{small}
By definition, \begin{small}$V_3/A_\p^{*3}$\end{small} is a subgroup of \begin{small}$(A_\p^*/A_\p^{*3})_{N=1}$\end{small} and it has order $3$ (Prop. \ref{sizeofV3}). Similarly, \begin{small}$\delta_{\phi,K_\p}(E_a(K_\p))$\end{small} is a  subgroup of \begin{small}$(A_\p^*/A_\p^{*3})_{N=1}$\end{small} of order $9$, by Prop. \ref{propos5.7} and Lemma \ref{char3}. Thus, it suffices to show that \begin{small}$V_3/A_\p^{*3}$\end{small} and \begin{small}$\delta_{\phi,K_\p}(E_a(K_\p))$\end{small} have non-trivial intersection.

We start by choosing a unit \begin{small}$u = 1+ \p^3 \beta \in U_{K_\p}^3$\end{small} such that \begin{small}$\beta \notin \p$\end{small} and \begin{small}$\overline{u} \notin (\OO_{K_\p}^*/U_{K_\p}^4)^{3}$\end{small} (existence of such a unit follows from Prop. \ref{U3cubes}(1)).  Note that the Frobenius map \begin{small}$x \mapsto x^3$\end{small} is the identity map on \begin{small}$\OO_{K_\p}/\p \cong \F_3$.\end{small} So for \begin{small}$w := \frac{\beta}{\sqrt{a}} \in \OO_{K_\p}^*$,\end{small} we have \begin{small}$w \equiv w^3 \pmod{\p}$.\end{small} We set \begin{small}$z:= \p w $\end{small} and notice that \begin{small}$\upsilon_\p(z) =1$.\end{small} Let \begin{small}$P(z)=(x(P(z)),y(P(z)))$\end{small} be the corresponding point on \begin{small}$E_a(K_\p)$\end{small} obtained by using the formal group of the elliptic curve. Then \begin{small}$\upsilon_\p(y(P(z))) = -3$, $\upsilon_\p(x(P(z))) = -2$\end{small} and 
\begin{small}$z^3(y(P(z)) - \sqrt{a}) = -1 - \sqrt{a} z^3 + O(z^4)$\end{small} by \cite[pg.~118, \S4]{sil1}.

Hence, \begin{small}${\p^{3}\omega^{3}(y(P(z))- \sqrt{a})} \equiv -1 \pmod {\p^3}$\end{small} which implies \begin{small}$\overline{\p^{3}\omega^{3}(y(P(z))- \sqrt{a})} \in \big({\OO_{K_\p}^*}/{U^3_{K_\p}}\big)^3$.\end{small} Therefore, by definition, \begin{small}$\p^{3}\omega^{3}(y(P(z))- \sqrt{a}) \in V_{K_\p}$.\end{small} Similarly, \begin{small}$\p^{3}\omega^{3}(y(P(z))+ \sqrt{a}) \in V_{K_\p}$\end{small} as well. Now clearly, \begin{small}$\p^{3}\omega^{3}(y(P(z))- \sqrt{a}) \cdot \p^{3}\omega^{3}(y(P(z))+ \sqrt{a})=\big(\p^2 \omega^2 x(P(z))\big)^3 \in \OO_{K_\p}^{*3}.$\end{small} Thus, by definition, \begin{small}$\big(\p^{3}\omega^{3}(y(P(z))- \sqrt{a}), \ \p^{3}\omega^{3}(y(P(z))+\sqrt{a})\big) \in V_3$.\end{small} We now claim that \begin{small}$\big(\overline{\p^{3}\omega^{3}(y(P(z))- \sqrt{a})}, \ \overline{\p^{3}\omega^{3}(y(P(z))+\sqrt{a})}\big)$\end{small} is a non-trivial element of \begin{small}$V_3/A_\p^{*3}$.\end{small} Indeed, \begin{small}${\p^{3}\omega^{3}(y(P(z))- \sqrt{a})} \equiv -1 - \p^3 \sqrt{a} w^3 \equiv -u \pmod{\p^4}$\end{small} and \begin{small}$\overline{u} \notin (\OO_{K_\p}^*/U_{K_\p}^4)^{3}$.\end{small} Hence, it follows that \begin{small}${\p^{3}\omega^{3}(y(P(z))- \sqrt{a})} \notin \OO_{K_\p}^{*3}$\end{small} and the claim is established.\\
Now, \begin{small}$\delta_{\phi,K_\p}(P(z)) =(\overline{y(P(z)) - \sqrt{a}}, \ \overline{y(P(z)) + \sqrt{a}})=  (\overline{\p^{3}\omega^3(y(P(z)) -  \sqrt{a} )}, \ \overline{\p^{3}\omega^3(y(P(z)) + \sqrt{a})})$ $\in (A_\p^*/A_\p^{*3})_{N=1}.$\end{small}
Thus, \begin{small}$\delta_{\phi,K_\p}(P(z))$\end{small} is a non-trivial element of \begin{small}$V_3/A_\p^{*3} \cap \delta_{\phi,K_\p}(E_a(K_\p))$.\end{small}
\end{proof}

\subsection{Bounds for $\phi$-Selmer group}\label{type1global}

\subsubsection{{\bf Selmer Group over $K$}}
Recall from \S\ref{iso},  we have a \begin{small}$K$\end{small}-rational $3$-isogeny \begin{small}$\phi:E_a \to E_a$.\end{small} We now use the local theory developed in \S\ref{type1theory} to give bounds on \begin{small}$\dim_{\F_3}{\rm Sel}^\phi(E_a/K)$.\end{small}

\begin{defn}\label{defofSa}
We define a finite subset \begin{small}$S_a$\end{small} of the set \begin{small}$\Sigma_K$\end{small} of all finite places of $F$, as follows:
\begin{small}$$ S_a:=\{ \q \in \Sigma_K  \mid a \in K_\q^{*2} \text{ and } \upsilon_\q(4a) \not\equiv 0 \pmod 6 \}.$$\end{small}
Further, for \begin{small}$a \notin K^{*2}$,\end{small} define \begin{small}$S_a(L) :=\{ \mathfrak{Q} \in \Sigma_L  \mid  \mathfrak{Q} \cap \OO_K \in S_a \}.$\end{small}

Note that if \begin{small}$a \in K^{*2}$,\end{small} then \begin{small}$S_a:=\{ \q \in \Sigma_K  \mid \upsilon_\q(4a) \not\equiv 0 \pmod 6 \}.$\end{small}
\end{defn}

\begin{rem}\label{numerical1}
Using Tate's algorithm, we observe that if a prime \begin{small}$\q \in S_a$,\end{small} then \begin{small}$E_a$\end{small} has additive reduction of Kodaira  type \rom{4} or \rom{4}$^*$ at \begin{small}$\q$\end{small} and the Tamagawa number \begin{small}$c_\q(E_a)=3$.\end{small}
Thus, a prime \begin{small}$\q \neq \p$\end{small} of \begin{small}$K$\end{small} is in \begin{small}$S_a \Leftrightarrow$\end{small} the Tamagawa number \begin{small}$c_\q(E_a)$\end{small} of \begin{small}$E_a$\end{small} at \begin{small}$\q$\end{small} is $3$. 
The prime \begin{small}$\p \mid 3$\end{small} of \begin{small}$K$\end{small} is in \begin{small}$S_a \Leftrightarrow 3 \mid a$\end{small} and the Tamagawa number \begin{small}$c_\p(E_a)=3$.\end{small} 

This numerical condition on \begin{small}$\upsilon_\q(4a)$\end{small} is perhaps easier to verify than computing the Tamagawa numbers \begin{small}$c_\q(E_a)$,\end{small} which is discussed in \cite{ss}. As an example, for a given \begin{small}$a \notin K^{*2}$,\end{small} we have  a criterion to decide whether \begin{small}$\p \in S_a$: \end{small}
If \begin{small}$a \notin K^{*2}$\end{small} i.e. \begin{small}$a \neq n^2$\end{small} or \begin{small}$-3n^2$,\end{small} then \begin{small}$\p \in S_a$\end{small} if and only if one of the following holds:
\begin{enumerate}
    \item \begin{small}$3 \mid\mid a$\end{small} and \begin{small}$\frac{a}{3} \equiv 2 \pmod 3$,\end{small} 
	\item \begin{small}$9 \mid\mid a$\end{small} and \begin{small}$\frac{a}{9} \equiv 1 \pmod 3$.\end{small}
\end{enumerate} 
\end{rem}

\vspace{2mm}
Recall from \eqref{eq:newseldef},
\begin{small}${\rm Sel}^\phi(E_a/K)=\{ \overline{x} \in (L^*/L^{*3})_{N=1} \mid \overline{x} \in \text{Im } \delta_{\phi,K_\q} \text{ for all } \q \in \Sigma_K \}$.\end{small} The subgroups \begin{small}$M(S_a,a)$, $N(S_a,a)$, $N'(S_a,a)$\end{small} of \begin{small}$\big(L^*/L^{*3}\big)_{N=1}$\end{small} are defined as in Prop. \ref{newM1M2sel}. We are now ready to discuss the bounds on \begin{small}${\rm Sel}^\phi(E_a/K)$\end{small}.

\begin{theorem}\label{type1selmer}
We have the inclusion of \begin{small}$\F_3$-\end{small}modules and the inequalities of their $3$-ranks: 
\begin{enumerate}
    \item If \begin{small}$a \not\in K^{*2}$,\end{small} then \begin{small}$M(S_a,a) \subset {\rm Sel}^\phi(E_a/K) \subset N(S_a,a)$.\end{small} Hence, we have that \begin{small}$$h^3_{S_a(L)} \le \dim_{\F_3} {\rm Sel}^\phi(E_a/K) \le h^3_{S_a(L)}+|S_a(L)|+2.$$\end{small}
    \item If \begin{small}$a \in K^{*2}$,\end{small} then \begin{small}${\rm Sel}^\phi(E_a/K) \subset N'(S_a,a)$.\end{small} Hence, we have that \begin{small}$\dim_{\F_3} {\rm Sel}^\phi(E_a/K) \le |S_a|+1$.\end{small}
\end{enumerate}
\end{theorem}
\begin{proof}
First we consider the case  \begin{small}$a \notin K^{*2}$\end{small} and show that \begin{small}${\rm Sel}^\phi(E_{a}/K) \subset N(S_a,a)$.\end{small} By Prop. \ref{newM1M2sel}, it is enough to show that \begin{small}$\text{Im } \delta_{\phi,K_\q} \subset (A_\q^*/A_\q^{*3})_{N=1}$\end{small} for all \begin{small}$\q \notin S_a$.\end{small} This is clear from the local theory (at \begin{small}$\q \notin S_a$\end{small}) of the elliptic curve \begin{small}$E_{a}$\end{small} (Corollary  \ref{kumcharnot3}, Prop. \ref{kumcharnot3sqr}(1) for \begin{small}$\q \nmid 3$\end{small} and Lemma \ref{kumchar3notsqr}, Prop. \ref{propos5.7}(1) for \begin{small}$\p \mid 3$\end{small}). 

Next, we prove that \begin{small}$M(S_a,a) \subset {\rm Sel}^\phi(E_{a}/K)$,\end{small} when \begin{small}$a \notin K^{*2}$.\end{small} From Prop. \ref{newM1M2sel}, it suffices to show that \begin{small}$(A_\q^*/A_\q^{*3})_{N=1} \subset \text{Im }\delta_{\phi, K_\q}$\end{small} for \begin{small}$\q \notin S_a \cup \{\p\}$\end{small} and \begin{small}$V_3/A_\p^{*3} \subset \text{Im }\delta_{\phi, K_\p}$\end{small} for \begin{small}$\p \notin S_a$.\end{small} These results follow from Corollary \ref{kumcharnot3}, Prop. \ref{kumcharnot3sqr}(1) for \begin{small}$\q \nmid 3$\end{small} and  Remark \ref{anotsqV3}, Prop. \ref{V3fortype1} for \begin{small}$\p \mid 3$.\end{small} 

When \begin{small}$a \in K^{*2}$,\end{small} due to Prop. \ref{newM1M2sel}, \begin{small}${\rm Sel}^\phi(E_a/K) \subset N'(S_a,a)$\end{small} can be obtained from Corollary \ref{kumcharnot3}, Prop. \ref{kumcharnot3sqr}(1) for \begin{small}$\q \nmid 3$\end{small} and Lemma \ref{kumchar3notsqr}, Prop. \ref{propos5.7}(1) for \begin{small}$\p \mid 3$.\end{small} 

In both cases (1) and (2), the statements about ranks follow directly from Prop. \ref{boundsgen}.
This completes the proof of the theorem.
\end{proof}

\subsubsection{{\bf Selmer Group over $\Q$}}
Recall from \S\ref{iso}, we have defined rational $3$-isogenies \begin{small}$E_a \underset{\widehat{\phi_a}}{\overset{\phi_a}{\rightleftarrows}} \widehat{E}_a$\end{small} with \begin{small}$a$\end{small} sixth-power free in \begin{small}$\Q$.\end{small} 
Further recall that we have the corresponding quadratic {\'e}tale algebras   \begin{small}$\widetilde{\Q}_{\widehat{\phi}_a}$, $\widetilde{\Q}_{{\phi}_a}$\end{small}  over \begin{small}$\Q$.\end{small}

\begin{defn}\label{defofSaQ}
We define \begin{small}$S_a(\Q):=\big\{ \ell \in \Z \setminus \{3\}  \mid  -3a \in \Q_\ell^{*2}  \text{ and }  \upsilon_\ell(4a) \not\equiv 0 \pmod 6 \big\} \ \bigcup \ T_3(\Q)$,\\
 where $T_3(\Q):=\begin{cases} \{3\}, & \text{ if } -3a \in \Q_3^{*2} \text{ and } \upsilon_3(a)=1 \text{ or } 5, \\ \emptyset, & \text{ otherwise.} \end{cases}$ 
\end{small}
\end{defn}
\begin{defn}
(1) If \begin{small}$\widetilde{\Q}_{\widehat{\phi}_a}$\end{small} is a field i.e. $-3a \notin \Q^{*2}$, put  \begin{small}$S_a(\widetilde{\Q}_{\widehat{\phi}_a}):=\big\{\mathfrak{l} \in \Sigma_{\widetilde{\Q}_{\widehat{\phi}_a}} \mid \mathfrak{l} \cap \Q \in S_a(\Q)\big\}$\end{small}  
\begin{small}
$$\text{and define } N(S_a(\Q),a):=\big\{\overline{x} \in \big(\widetilde{\Q}_{\widehat{\phi}_a}^*/\widetilde{\Q}_{\widehat{\phi}_a}^{*3}\big)_{N=1} \mid (x)=I^3 \text{ for some fractional ideal } I \text{ of } \OO_{S_a(\widetilde{\Q}_{\widehat{\phi}_a})}\big\}.$$\end{small}
(2) If \begin{small}$\widetilde{\Q}_{\widehat{\phi}_a} \cong \Q \times \Q$\end{small} i.e. \begin{small}$-3a \in \Q^{*2}$,\end{small} define \begin{small}$$N'(S_a(\Q),a):=\big\{(\overline{x}_1,\overline{x}_2) \in \big(\widetilde{\Q}_{\widehat{\phi}_a}^*/\widetilde{\Q}_{\widehat{\phi}_a}^{*3}\big)_{N=1} \mid \upsilon_\ell(\overline{x}_1,\overline{x}_2) \equiv 0 \pmod 3, \ \forall \ell \notin S_a(\Q)\big\}.$$\end{small}
\end{defn}
Let $r_1$, $r_2$ be the number of real and conjugate pairs of complex embeddings of \begin{small}$\widetilde{\Q}_{\widehat{\phi}_a}$,\end{small} respectively. Then similar to Prop. \ref{boundsgen}, we have the following result.
\begin{proposition}
\begin{small}$(1) \dim_{\F_3} N(S_a(\Q),a) = h^3_{S_a(\widetilde{\Q}_{\widehat{\phi}_a})}+|S_a(\widetilde{\Q}_{\widehat{\phi}_a})|+t, \text{ where } t=\begin{cases} 0, & \text{ if } r_2=1 \text{ and } \widetilde{\Q}_{\widehat{\phi}_a} \neq K, \\ 1, & \text{ if either } r_1=2 \text{ or } \widetilde{\Q}_{\widehat{\phi}_a} = K. \end{cases}$\end{small}\\
\begin{small}$(2)\dim_{\F_3} N'(S_a(\Q),a)=|S_a(\Q)|$.\end{small} \qed
\end{proposition}

For every rational prime $\ell$, we define the sets \begin{small}$\big(A_{\ell}^*/A_{\ell}^{*3}\big)_{N=1}$\end{small} similar to the sets \begin{small}$\big(A_{\q}^*/A_{\q}^{*3}\big)_{N=1}$\end{small} for $\q \in \Sigma_K$. Then doing the local theory as in the subsection \S\ref{type1theory} and following the proof for the upper bound in  Theorem \ref{type1selmer}, we arrive at the theorem below:

\begin{theorem}\label{boundsoverQ}
Let \begin{small}$E_a$, $\widehat{E}_a$, $\phi_a$, $\widehat{\phi}_a$\end{small}  and \begin{small}$\widetilde{\Q}_{\widehat{\phi}_a}$\end{small} be as above. 
Then \begin{enumerate}
    \item If \begin{small}$\widetilde{\Q}_{\widehat{\phi}_a} $\end{small} is an imaginary quadratic field other than \begin{small}$K$,\end{small} then \begin{small}$\dim_{\F_3}{\rm Sel}^{\phi_a}(E_a/\Q) \le $ $ h^3_{S_a(\widetilde{\Q}_{\widehat{\phi}_a})}+|S_a(\widetilde{\Q}_{\widehat{\phi}_a})|.$\end{small}
    \item If \begin{small}$\widetilde{\Q}_{\widehat{\phi}_a}$\end{small} is either $K$ or a real quadratic field, then \begin{small}$\dim_{\F_3}{\rm Sel}^{\phi_a}(E_a/\Q) \le h^3_{S_a(\widetilde{\Q}_{\widehat{\phi}_a})}+|S_a(\widetilde{\Q}_{\widehat{\phi}_a})|+1.$\end{small}
    \item If \begin{small}$\widetilde{\Q}_{\widehat{\phi}_a} \cong \Q \times \Q$,\end{small} then \begin{small}$\dim_{\F_3}{\rm Sel}^{\phi_a}(E_a/\Q) \le |S_a(\Q)|$.\end{small} 
\end{enumerate}
Replacing \begin{small}$\widehat{\phi}_a$\end{small} by $\phi_a$ and $a$ by $a\alpha^2$, we get the corresponding bounds for \begin{small}$\dim_{\F_3}{\rm Sel}^{\widehat{\phi}_a}(\widehat{E}_a/\Q)$.\end{small} \qed
\end{theorem}

\begin{rem}\label{compi1}
In \cite{ban2}, the author computes an upper bound on \begin{small}${\rm Sel}^{\phi_a}(E_a/\Q)$\end{small} using a finite set $S'_1(\Q)$ of $\Q$. It seems \cite{ban2} only considers the case when $\Q(\sqrt{-3a})$ is a field, although it is not mentioned explicitly. 
Note that Selmer groups over $K$ and lower bounds on them are not considered in \cite{ban2}.
When \begin{small}$\Q(\sqrt{-3a})$\end{small} is a field, our set $S_a(\Q)$ and the upper bound for \begin{small}${\rm Sel}^{\phi_a}(E_a/\Q)$\end{small} coincide with that of \cite{ban2}. We bypass the   computations for generators of the image of the Kummer map using explicit divisors on $\widehat{E}_a$. Instead, we use the group \begin{small}$\big(A_\ell^*/A_\ell^{*3}\big)_{N=1}$.\end{small} 
\end{rem}
The elliptic curve \begin{small}$\widehat{E}_a$\end{small} is a quadratic twist by $-3$ or $-1/3$ of  \begin{small}$E_a$\end{small} and so they are isomorphic over \begin{small}$K$.\end{small} Let \begin{small}$\text{rk } E(F)$\end{small} denote the Mordell-Weil rank of the curve $E$ over the field $F$. The result below is easy to deduce (cf. \cite[Lemma~3.1]{op}):
\begin{lemma}\label{ono}
Let \begin{small}$E_a$, $\widehat{E}_a$\end{small} and \begin{small}$\phi_a$, $\widehat{\phi}_a$\end{small} be as defined in \S\ref{iso}. Then 
\begin{enumerate}
	\item \begin{small}$\text{rk } E_a(K)=\text{ rk } E_a(\Q)+ \text{ rk } \widehat{E}_a(\Q)=2 \text{ rk } E_a(\Q)=2 \text{ rk } \widehat{E}_a(\Q)$.\end{small}
	\item \begin{small}${\rm Sel}^3(E_a/K) \cong {\rm Sel}^{3}(E_a/\Q) \oplus {\rm Sel}^{3}(\widehat{E}_a/\Q)$.\end{small} So, \begin{small}$\dim_{\F_3} {\rm Sel}^3(E_a/K)= \dim_{\F_3} {\rm Sel}^3(E_a/\Q)+\dim_{\F_3} {\rm Sel}^{3}(\widehat{E}_a/\Q)$.\end{small}
	\item \begin{small}${\rm Sel}^\phi(E_a/K) \cong {\rm Sel}^{\phi_a}(E_a/\Q) \oplus {\rm Sel}^{\widehat{\phi}_a}(\widehat{E}_a/\Q)$.\end{small} So, \begin{small}$\dim_{\F_3} {\rm Sel}^\phi(E_a/K)= \dim_{\F_3} {\rm Sel}^{\phi_a}(E_a/\Q)+\dim_{\F_3} {\rm Sel}^{\widehat{\phi}_a}(\widehat{E}_a/\Q)$.\end{small}
	\item \begin{small}$\dim_{\F_3} \Sh(E_a/K)[3]= \dim_{\F_3} \Sh(E_a/\Q)[3]+\dim_{\F_3} \Sh(\widehat{E}_a/\Q)[3]$.\end{small} \qed
\end{enumerate}
\end{lemma}

\begin{rem}\label{rmkSel3overQ}
An upper bound on \begin{small}$\dim_{\F_3}{\rm Sel}^{3}(E_a/\Q)$\end{small} can be computed using \begin{small}$\dim_{\F_3}{\rm Sel}^{3}(E_a/\Q) \le \dim_{\F_3}{\rm Sel}^{\phi_a}(E_a/\Q) + \dim_{\F_3}{\rm Sel}^{\widehat{\phi_{a}}}(\widehat{E}_{a}/\Q)=$ $\dim_{\F_3} {\rm Sel}^\phi(E_a/K)$\end{small} and the upper bounds on these terms.
\end{rem}

\subsection{Refined Bounds for Selmer Group over $K$}\label{type1boundsec}
\begin{theorem}\label{type1bounds}
Suppose that \begin{small}$a \not\in K^{*2}$\end{small} and \begin{small}$a$\end{small} is sixth-power free in \begin{small}$K$.\end{small} Then we have \begin{small}$$h^3_{S_a(L)} \le \dim_{\F_3} {\rm Sel}^\phi(E_a/K) \le \text{min}\big\{h^3_{S_a(L)} + |S_a(L)| + 2,\quad h^3_{S_a(\widetilde{\Q}_{\widehat{\phi}_a})}+h^3_{S_{a\alpha^2}(\widetilde{\Q}_{{\phi}_a})}+|S_a(L)|+1\big\}.$$\end{small}
In particular, if \begin{small}$a \not\in K^{*2}$\end{small} and \begin{small}$S_a=\emptyset$,\end{small} then \begin{small}$\dim_{\F_3} {\rm Sel}^\phi(E_a/K) \in \big\{h^3_L, h^3_L+1 \big\}$\end{small} and it is uniquely determined by the root number of \begin{small}$E_a/\Q$.\end{small}
\end{theorem}

\begin{proof}
When \begin{small}$a \notin K^{*2}$,\end{small} then by Theorem \ref{type1selmer},  \begin{small}$h^3_{S_a(L)} \le \dim_{\F_3} {\rm Sel}^\phi(E_a/K)$.\end{small} 
Now as \begin{small}$a \notin K^{*2}$,\end{small} both \begin{small}$\widetilde{\Q}_{\widehat{\phi}_a}$\end{small} and \begin{small}$\widetilde{\Q}_{{\phi}_a}$\end{small} are quadratic fields and exactly one of them is an imaginary quadratic field (not isomorphic to $K$), while the other is a real quadratic field. Hence, by Theorem \ref{boundsoverQ}, 
\begin{small}\[\dim_{\F_3}{\rm Sel}^{\phi_a}(E_a/\Q) + \dim_{\F_3}{\rm Sel}^{\widehat{\phi}_a}(\widehat{E}_a/\Q) \le h^3_{S_a(\widetilde{\Q}_{\widehat{\phi}_a})}+h^3_{S_{a\alpha^2}(\widetilde{\Q}_{{\phi}_a})}+|S_a(\widetilde{\Q}_{\widehat{\phi}_a})|+|S_{a\alpha^2}(\widetilde{\Q}_{{\phi}_a})|+1.\]\end{small}
Note that \begin{small}${\rm Sel}^\phi(E_a/K) \cong {\rm Sel}^{\phi_a}(E_a/K) \cong {\rm Sel}^{\phi_a}(E_a/\Q) \oplus {\rm Sel}^{\widehat{\phi}_a}(\widehat{E}_a/\Q)$\end{small} (by Lemma \ref{ono}) and also that  \begin{small}$|S_a(L)|=|S_a(\widetilde{\Q}_{\widehat{\phi}_a})|+|S_{a\alpha^2}(\widetilde{\Q}_{{\phi}_a})|$.\end{small} Whence by Theorems \ref{type1selmer}, \ref{boundsoverQ} and Prop. \ref{boundsgen}, we have \begin{small}$$\dim_{\F_3} {\rm Sel}^\phi(E_a/K) \le \text{min }\big\{h^3_{S_a(L)} + |S_a(L)| + 2, \ h^3_{S_a(\widetilde{\Q}_{\widehat{\phi}_a})}+h^3_{S_{a\alpha^2}(\widetilde{\Q}_{{\phi}_a})}+|S_a(L)|+1\big\}.$$\end{small}
In particular, if \begin{small}$S_a=\emptyset$,\end{small} then it follows immediately from Lemma \ref{hergoltzresult}  that \begin{small}$\dim_{\F_3} {\rm Sel}^\phi(E_a/K) \in \big\{h^3_L, \  h^3_L+1\big\}$\end{small}. 

By \cite[Lemma~6.1]{ss}, we have the following exact sequence
\begin{small}\begin{equation}\label{eq:selmerseqoverQ}
0 \rightarrow \frac{\widehat{E}_a(\Q)[\widehat{\phi_a}]}{\phi_a(E_a(\Q)[3])} \rightarrow {\rm Sel}^{\phi_a}(E_a/\Q) \rightarrow {\rm Sel}^3(E_a/\Q) \rightarrow {\rm Sel}^{\widehat{\phi_a}}(\widehat{E}_a/\Q) \rightarrow \frac{\Sh(\widehat{E}_a/\Q)[\widehat{\phi_a}]}{\phi_a(\Sh(E_a/\Q)[3])} \rightarrow 0,\end{equation}\end{small}
and so \begin{small}$\dim_{\F_3} {\rm Sel}^3(E_a/\Q)= \dim_{\F_3} {\rm Sel}^\phi(E_a/K) - \dim_{\F_3} \frac{\Sh(\widehat{E}_a/\Q)[\widehat{\phi}_a]}{\phi_a(\Sh(E_a/\Q)[3])}$.\end{small} The last term is even by \cite[Prop.~49]{bes}. Hence, \begin{small}$\dim_{\F_3} {\rm Sel}^3(E_a/\Q) \equiv \dim_{\F_3} {\rm Sel}^\phi(E_a/K) \pmod 2$.\end{small} Now, the $3$-parity conjecture for \begin{small}$E_a/\Q$\end{small} is known due to Nekov{\'a}{\v r} and Dokchitser-Dokchitser. If we denote  \begin{small}$r:=\dim_{\Q_3} \text{Hom}_{\Z_3}\big({\rm Sel}^{3^\infty}(E_a/\Q), \Q_3/\Z_3\big) \otimes \Q_3$,\end{small} then the global root number \begin{small}$\omega(E_a/\Q)$\end{small} of \begin{small}$E_a/\Q$\end{small} satisfies \begin{small}$\omega(E_a/\Q)=(-1)^r$.\end{small} Also, it is easy to see that \begin{small}$\dim_{\F_3} {\rm Sel}^3(E_a/\Q) \equiv r \pmod 2.$\end{small} So, \begin{small}$\omega(E_a/\Q)=(-1)^{\dim_{\F_3} {\rm Sel}^\phi(E_a/K)}$\end{small}
and hence, \begin{small}$\dim_{\F_3}{\rm Sel}^\phi(E_a/K)$\end{small} is determined uniquely by \begin{small}$\omega(E_a/\Q).$\end{small}
\end{proof}

Theorems \ref{type1selmer} and \ref{type1bounds} have the following consequence on the $3$-Selmer group ${\rm Sel}^3(E_a/K)$:

\begin{corollary}\label{corforsel3}
If \begin{small}$a \notin K^{*2}$,\end{small} then we have that \begin{small}$$\text{max }\big\{h^3_{S_a(L)}, \text{rk }E_a(K)\big\} \le \dim_{\F_3}{\rm Sel}^{3}(E_a/K) \le 2 \ \text{min }\big\{h^3_{S_a(L)} + |S_a(L)| + 2, \ h^3_{S_a(\widetilde{\Q}_{\widehat{\phi}_a})}+h^3_{S_{a\alpha^2}(\widetilde{\Q}_{{\phi}_a})}+|S_a(L)|+1\big\}.$$\end{small}
In particular, if \begin{small}$a \notin K^{*2}$\end{small} and \begin{small}$S_a=\emptyset$,\end{small} then \begin{small}$\text{max } \big\{h^3_L, \text{ rk }E_a(K)\big\} \le \dim_{\F_3}{\rm Sel}^{3}(E_a/K) \le 2h^3_L + 2.$\end{small}

On the other hand, when \begin{small}$a \in K^{*2}$,\end{small} we have \begin{small}$\dim_{\F_3} {\rm Sel}^3(E_a/K) \leq 2|S_a|+2$\end{small} and hence \begin{small}$\text{rk } E_a(K) \le 2|S_a|$. \end{small}
\end{corollary}

\begin{proof}
As \begin{small}$\frac{E_a(K)}{3E_a(K)} \hookrightarrow {\rm Sel}^3(E_a/K)$,\end{small} we know \begin{small}$\text{rk } E_a(K) \leq \dim_{\F_3} {\rm Sel}^3(E_a/K)$.\end{small}
We have the following exact sequence (see for example \cite[Lemma~6.1]{ss})
\begin{small}
\begin{equation}\label{eq:selmerseq}
0 \rightarrow \frac{E_a(K)[\phi]}{\phi(E_a(K)[3])} \rightarrow {\rm Sel}^\phi(E_a/K) \rightarrow {\rm Sel}^3(E_a/K) \rightarrow {\rm Sel}^\phi(E_a/K) \rightarrow \frac{\Sh(E_a/K)[\phi]}{\phi(\Sh(E_a/K)[3])} \rightarrow 0.
\end{equation}
\end{small}
Note that if $a \notin K^{*2}$, then $E_a(K)[\phi]=\{O\}$, hence, equation \eqref{eq:selmerseq} becomes
\begin{small}
\begin{equation}\label{eq:simplifiedseq}
0 \rightarrow {\rm Sel}^\phi(E_a/K) \rightarrow {\rm Sel}^3(E_a/K) \rightarrow {\rm Sel}^\phi(E_a/K) \rightarrow \frac{\Sh(E_a/K)[\phi]}{\phi(\Sh(E_a/K)[3])} \rightarrow 0
\end{equation}
\end{small}
Thus, \begin{small}
$h^3_{S_a(L)} \le \dim_{\F_3} {\rm Sel}^\phi(E_a/K) \leq \dim_{\F_3} {\rm Sel}^3(E_a/K)= 2 \dim_{\F_3} {\rm Sel}^\phi(E_a/K) - \dim_{\F_3} \frac{\Sh(E_a/K)[\phi]}{\phi(\Sh(E_a/K)[3])} $ $\leq 2 \dim_{\F_3} {\rm Sel}^\phi(E_a/K) \le 2\ \text{min}\big\{ h^3_{S_a(L)} + |S_a(L)| + 2, \ h^3_{S_a(\widetilde{\Q}_{\widehat{\phi}_a})}+h^3_{S_{a\alpha^2}(\widetilde{\Q}_{{\phi}_a})}+|S_a(L)|+1 \big\}.$\end{small}

For \begin{small}$a \in K^{*2}$, $\dim_{\F_3} {\rm Sel}^3(E_a/K)= 2 \dim_{\F_3} {\rm Sel}^\phi(E_a/K) - \dim_{\F_3} \frac{E_a(K)[\phi]}{\phi(E_a(K)[3])} -  \dim_{\F_3} \frac{\Sh(E_a/K)[\phi]}{\phi(\Sh(E_a/K)[3])} \leq 2 \dim_{\F_3} {\rm Sel}^\phi(E_a/K) \leq 2|S_a|+2,$\end{small}   the last inequality is due to Theorem \ref{type1selmer}. As \begin{small}$E_a(K)$\end{small} has a non-trivial $3$-torsion point, it follows that \begin{small}$\text{rk } E_a(K) < 2|S_a| +2$\end{small}, but since \begin{small}$\text{rk } E_a(K)$\end{small} is even, the claim follows. 
\end{proof}

\begin{rem}[Non-trivial \begin{small}${\Sh(E_a/K)[3]}$\end{small}]\label{rmkforsha}
Observe that if \begin{small}$a \notin K^{*2}$,\end{small} then by the duplication formula, \begin{small}$E_a(K)[3]=\{O\}$.\end{small} Then \begin{small}$\dim_{\F_3} \Sh(E_a/K)[3] \geq \dim_{\F_3} \Sh(E_a/K)[\phi] = \dim_{\F_3} {\rm Sel}^\phi(E_a/K) - \dim_{\F_3} \frac{E_a(K)}{\phi(E_a(K))} \geq \dim_{\F_3} {\rm Sel}^\phi(E_a/K) - \dim_{\F_3} \frac{E_a(K)}{3(E_a(K))} \geq h^3_{S_a(L)} - \text{rk } E_a(K).$\end{small} In the Table \ref{tab:type1examples}, we have computed explicit examples where \begin{small}$0 \leq \text{rk } E_a(K) < h^3_{S_a(L)}$\end{small} when \begin{small}$a \notin K^{*2}$.\end{small} 
 This shows that \begin{small}$\Sh(E_a/K)[\phi]$\end{small} (and hence \begin{small}$\Sh(E_a/K)[3]$\end{small}) is non-trivial in all these cases.
\end{rem}


We now compute \begin{small}$\text{rk } E_a(\Q)$\end{small} in terms of \begin{small}$\dim_{\F_3} {\rm Sel}^\phi(E_a/K)$\end{small} and \begin{small}$\dim_{\F_3}\Sh(E_a/K)[\phi]$\end{small}:
\begin{corollary}\label{cortotype1bounds}
We have \begin{small}$\text{rk } E_a(\Q)=  \dim_{\F_3} {\rm Sel}^\phi(E_a/K) - \dim_{\F_3}E_a(K)[\phi] -  \dim_{\F_3}\Sh(E_a/K)[\phi]$.\end{small} 
As a consequence, we see that \begin{small}$\text{rk } E_a(\Q) \le \dim_{\F_3} {\rm Sel}^\phi(E_a/K)$\end{small} and the equality holds 
if and only if \begin{small}$a \notin K^{*2}$\end{small} and  \begin{small}${\Sh(E_a/K)[\phi]} = 0$.\end{small} 
Moreover, if \begin{small}$a \in K^{*2}$,\end{small} then \begin{small}$\text{rk } E_a(\Q) \le |S_a|$\end{small} and the equality holds if and only if \begin{small}${\Sh(E_a/K)[\phi]} = 0$.\end{small}
\end{corollary}

\begin{proof}
We have the following exact sequences of $\F_3$-modules
\begin{small}
\begin{equation}\label{eq:type1phi3}
0 \longrightarrow E_a(K)[\phi] \longrightarrow E_a(K)[3] \overset{\phi}{\longrightarrow} \phi(E_a(K)[3]) \longrightarrow 0
\end{equation}
\begin{equation}\label{eq:type1sha3}
0 \longrightarrow \Sh(E_a/K)[\phi] \longrightarrow \Sh(E_a/K)[3] \overset{\phi}{\longrightarrow} \phi(\Sh(E_a/K)[3]) \longrightarrow 0
\end{equation}
\end{small}
So, by \eqref{eq:type1phi3} we get, \begin{small}$\dim_{\F_3} E_a(K)[3] - \dim_{\F_3}\phi(E_a(K)[3])  = \dim_{\F_3}E_a(K)[\phi]$\end{small} and by \eqref{eq:type1sha3} we have \begin{small}$\dim_{\F_3}\Sh(E_a/K)[3] - \dim_{\F_3}\phi(\Sh(E_a/K)[3]) = \dim_{\F_3}\Sh(E_a/K)[\phi]$.\end{small}

Also, we know that \begin{small}
$\text{rk } E_a(K) = \dim_{\F_3} {\rm Sel}^3(E_a/K) - \dim_{\F_3}\Sh(E_a/K)[3] - \dim_{\F_3}E_a(K)[3]$\end{small} 
and by \eqref{eq:selmerseq}, we have that \begin{small}$\dim_{\F_3} {\rm Sel}^3(E_a/K)= 2\dim_{\F_3} {\rm Sel}^\phi(E_a/K) - \dim_{\F_3} \frac{\Sh(E_a/K)[\phi]}{\phi(\Sh(E_a/K)[3])} - \dim_{\F_3} \frac{E_a(K)[\phi]}{\phi(E_a(K)[3])}$.
\begin{multline*}
\text{So, } \text{rk } E_a(K) = 2 \dim_{\F_3} {\rm Sel}^\phi(E_a/K) - \dim_{\F_3}E_a(K)[\phi] + \dim_{\F_3}\phi(E_a(K)[3]) - \dim_{\F_3}\Sh(E_a/K)[\phi] + \\ 
 \dim_{\F_3}\phi(\Sh(E_a/K)[3]) - \dim_{\F_3}\Sh(E_a/K)[3] - \dim_{\F_3}E_a(K)[3] \\
 = 2 \big(\dim_{\F_3} {\rm Sel}^\phi(E_a/K) - \dim_{\F_3}E_a(K)[\phi] -  \dim_{\F_3}\Sh(E_a/K)[\phi] \big). 
\end{multline*}
\end{small}
Now, by Lemma \ref{ono}(1), \begin{small}$\text{rk } E_a(K)=2\enspace\text{rk } E_a(\Q)$\end{small}. Also, \begin{small}$\dim_{\F_3} E_a(K)[\phi]=0$\end{small} if \begin{small}$a \notin K^{*2}$\end{small} and $1$ if \begin{small}$a \in K^{*2}$.\end{small} The result then follows immediately from Theorem \ref{type1selmer}.
\end{proof}

\vspace{2mm}
\noindent{\bf Example (1):} Let \begin{small}$a=359$\end{small} and notice that \begin{small}$a \notin K^{*2}$, $S_a=S_a(\widetilde{\Q}_{\widehat{\phi}_a})=S_{a\alpha^2}(\widetilde{\Q}_{{\phi}_a})=\emptyset$.\end{small} Also note that  \begin{small}$h^3_L=2$, $h^3_{\widetilde{\Q}_{\widehat{\phi}_a}}=1$\end{small} and \begin{small}$h^3_{\widetilde{\Q}_{{\phi}_a}}=1$.\end{small} Then by Theorem \ref{type1bounds}, \begin{small}$2 \le \dim_{\F_3} {\rm Sel}^\phi(E_a/K) \le 3$.\end{small} Again using \cite[\S9]{liv}, we obtain the global root number \begin{small}$\omega(E_a/\Q)=-1$.\end{small} As argued in Theorem \ref{type1bounds}, \begin{small}$\omega(E_a/\Q)=(-1)^{\dim_{\F_3}{\rm Sel}^\phi(E_a/K)}$,\end{small} which implies \begin{small}$\dim_{\F_3} {\rm Sel}^\phi(E_a/K)=3$.\end{small} 
   Along with Remark \ref{rmkSel3overQ}, this gives \begin{small}$\dim_{\F_3} {\rm Sel}^3(E_a/\Q) \le 3$\end{small} and \begin{small}$\dim_{\F_3} {\rm Sel}^3(\widehat{E}_a/\Q) \le 3$.\end{small} We also have \begin{small}$\text{rk } E_a(\Q)=\text{rk } \widehat{E}_{a}(\Q)=3$\end{small} (using SageMath). Thus, \begin{small}$\dim_{\F_3} {\rm Sel}^3(E_a/\Q) = \dim_{\F_3} {\rm Sel}^3(\widehat{E}_a/\Q)= 3$\end{small} and   \begin{small}$\Sh({E}_{a}/\Q)[3]=\Sh(\widehat{E}_{a}/\Q)[3]=0$.\end{small} Thus, by Lemma \ref{ono}, \begin{small}$\dim_{\F_3}{\rm Sel}^3(E_a/K)=6$\end{small} and \begin{small}$\Sh(E_a/K)[3]=0$.\end{small} 
Thus, both \begin{small}$\dim_{\F_3}{\rm Sel}^3(E_a/K)$\end{small} and \begin{small}$\dim_{\F_3}{\rm Sel}^\phi(E_a/K)$\end{small} attain the upper bounds in Corollary \ref{corforsel3} and Theorem \ref{type1bounds}.

\vspace{2mm}
\noindent{\bf Example (2):} If we take \begin{small}$a=822$,\end{small} then \begin{small}$a \notin K^{*2}$, $S_a=S_a(\widetilde{\Q}_{\widehat{\phi}_a})=S_{a\alpha^2}(\widetilde{\Q}_{{\phi}_a})=\emptyset$\end{small} and \begin{small}$h^3_L=$ $h^3_{\widetilde{\Q}_{\widehat{\phi}_a}}=1$\end{small} and \begin{small}$h^3_{\widetilde{\Q}_{{\phi}_a}}=0$.\end{small} Therefore, by Theorem \ref{type1bounds}, we get \begin{small}$1 \le \dim_{\F_3} {\rm Sel}^\phi(E_a/K) \le 2$.\end{small} From \cite[\S9]{liv}, we see that \begin{small}$\omega(E_a/\Q)=-1$\end{small} which implies \begin{small}$\dim_{\F_3} {\rm Sel}^\phi(E_a/K)=1$.\end{small} 
As explained in the proof of Theorem \ref{type1bounds},  we have  \begin{small}$\dim_{\F_3}{\rm Sel}^3({E}_a/\Q) \equiv \dim_{\F_3}{\rm Sel}^\phi({E}_a/K) \pmod 2$.\end{small} Similarly, \begin{small}$\dim_{\F_3}{\rm Sel}^3(\widehat{E}_a/\Q) \equiv \dim_{\F_3}{\rm Sel}^\phi({E}_a/K) \pmod 2$.\end{small} So, \begin{small}$\dim_{\F_3}{\rm Sel}^3({E}_a/\Q)=\dim_{\F_3}{\rm Sel}^3(\widehat{E}_a/\Q)=1$\end{small} and hence, \begin{small}$\dim_{\F_3}{\rm Sel}^3(E_a/K)=2$.\end{small}
Thus, both \begin{small}$\dim_{\F_3}{\rm Sel}^3(E_a/K)$\end{small} and \begin{small}$\dim_{\F_3}{\rm Sel}^\phi(E_a/K)$\end{small} achieve their respective lower bounds in Corollary \ref{corforsel3} and Theorem \ref{type1bounds}.

\section{Applications}\label{applications}
 We now highlight some applications of the results obtained  in \S\ref{type1curves} in the subsections \ref{cubesumsubsection}, \ref{positiveprop} and \ref{cubictwistsec} respectively.

\subsection{The cube sum problem}\label{cubesumsubsection}
Recall that a cube-free integer \begin{small}$D > 2$\end{small} can be written as a sum of two rational cubes if and only if \begin{small}$\text{rk } E_{-432D^2}(\Q) > 0$.\end{small} 
The elliptic curve $E_{aD^2}:y^2=x^3+aD^2$ is a cubic twist by $D$ of   \begin{small}$E_a:y^2=x^3+a$\end{small}.  
Let \begin{small}$\ell$\end{small} be a rational prime. The Sylvester's conjecture predicts that \begin{small}$$\text{rk } E_{-432\ell^2}(\Q)=\begin{cases} 0, & \text{ if } \ell \equiv 2,5 \pmod 9\\ 1, & \text{ if } \ell \equiv 4,7,8 \pmod 9\\ 0 \text{ or } 2,& \text{ if } \ell \equiv 1 \pmod 9. \end{cases}$$\end{small}
For \begin{small}$\ell \equiv 2,5 \pmod 9$,\end{small} the proof goes back to the works of P\'epin, Lucas and Sylvester.
 In \cite{dv}, for primes \begin{small}$\ell \equiv 4,7 \pmod 9$,\end{small} assuming that $3$ is not a cube modulo \begin{small}$\ell$,\end{small} it is shown that \begin{small}$\text{rk } E_{-432\ell^2}(\Q)=\text{rk } E_{-432\ell^4}(\Q)=1$.\end{small} 
\cite{rz} considers the primes \begin{small}$\ell \equiv 1 \pmod 9$\end{small} and provides three non-trivial criteria to check whether \begin{small}$\text{rk } E_{-432\ell^2}(\Q)$\end{small} is $0$ or $2$ (assuming the BSD conjecture). 
In \cite{ms, jms}, the authors produce infinitely many primes \begin{small}$\ell \equiv \pm 1 \pmod 9$\end{small} that are cube sums.

 Sylvester also proved that if \begin{small}$\ell \equiv 5 \pmod 9$\end{small} (respectively, if \begin{small}$\ell \equiv 2 \pmod 9$\end{small}), then \begin{small}$2\ell$\end{small} (respectively, \begin{small}$2\ell^2$\end{small}) is not a rational cube sum. If \begin{small}$\ell \equiv 2 \pmod 9$\end{small} (resp. \begin{small}$\ell \equiv 5 \pmod 9$\end{small}), Satg\'e \cite{sa} proved that \begin{small}$2\ell$\end{small} (resp. \begin{small}$2\ell^2$\end{small}) is a cube sum. The cube sum problem for \begin{small}$2\ell, 2\ell^2$\end{small} for primes $\ell \equiv 1 \pmod 3$ and $\ell \equiv 8 \pmod 9$ seems to be not covered in the literature. 

In Corollaries \ref{cortosylvester23} and \ref{cortosylvester2psqr}, we study the cube sum problem for \begin{small}$2\ell$\end{small} and \begin{small}$2\ell^2$.\end{small} For every prime \begin{small}$\ell \ge 5$\end{small} and \begin{small}$D\in\{2\ell, 2\ell^2\}$,\end{small} we determine \begin{small}$\dim_{\F_3}{\rm Sel}^\phi(E_{-432D^2}/K)$\end{small} precisely in Theorems \ref{sylvester2}-\ref{sylvester2psqr2}. Then we use this to prove results on \begin{small}$\text{rk } E_{-432D^2}(\Q)$\end{small}, some of them unconditionally (Theorem \ref{sylvesterintro}) while some other assuming that  \begin{small}$\dim_{\F_3}\Sh(E_a/\Q)[3]$\end{small} is even, in Corollaries \ref{cortosylvester23}-\ref{cortosylvester2psqr}.
Note that \begin{small}$E_{-432D^2}$\end{small} is isomorphic to \begin{small}$E_{16D^2}$\end{small} over \begin{small}$K$\end{small} and they are $3$-isogenous over \begin{small}$\Q$.\end{small} Thus, \begin{small}$\text{rk }E_{-432(2D)^2}(\Q)=\text{rk }E_{-27D^2}(\Q)=\text{rk } E_{D^2}(\Q)$\end{small}. 

Let  \begin{small}$a \in K^{*2}$\end{small} and  \begin{small}$S_a \subset \Sigma_K$\end{small} be as defined in \S\ref{type1global}. Then recall by Theorem \ref{type1selmer}, Props. \ref{boundsgen} and \ref{cohom} that \begin{small}${\rm Sel}^\phi(E_{a}/K) \subset \Big(\frac{A_{S_a}^*}{A_{S_a}^{*3}}\Big)_{N=1}:=\Big(\frac{\OO_{S_a}^*}{\OO_{S_a}^{*3}} \times \frac{\OO_{S_a}^*}{\OO_{S_a}^{*3}}\Big)_{N=1} \subset \Big(\frac{K^*}{K^{*3}} \times \frac{K^*}{K^{*3}}\Big)_{N=1}  \cong H^1(G_K,E_a[\phi])$.\end{small}   
Thus, \begin{small}$(\overline{x},\overline{x}^2) \in \Big(\frac{A_{S_{a}}^*}{A_{S_{a}}^{*3}}\Big)_{N=1}$\end{small} is in \begin{small}${\rm Sel}^\phi(E_{a}/K) $\end{small} if and only if \begin{small}$(\overline{x},\overline{x}^2) \in \text{Im } \delta_{\phi,K_\q}$\end{small} for all \begin{small}$\q$.\end{small}
Also, by Theorem \ref{type1selmer}, \begin{small}$ \dim_{\F_3} {\rm Sel}^\phi(E_{a}/K) \le |S_a|+1$.\end{small} We will use these facts for \begin{small}$a \in \{\ell^2, \ell^4\}$\end{small} in Theorems \ref{sylvester2}-\ref{sylvester2psqr2} to determine \begin{small}$\dim_{\F_3} {\rm Sel}^\phi(E_{a}/K)$.\end{small}

First, we compute \begin{small}$\dim_{\F_3}{\rm Sel}^\phi(E_{\ell^2}/K)$\end{small} for primes \begin{small}$\ell \equiv 2 \pmod 3$.\end{small}
\begin{theorem}\label{sylvester2}
Let \begin{small}$\ell \ge 5$\end{small} be a prime such that \begin{small}$\ell \equiv 2 \pmod 3$\end{small} and \begin{small}$\phi: E_{\ell^2} \to E_{\ell^2}$\end{small} be the $3$-isogeny given in \eqref{eq:defofphi}. Then \begin{small}$$\dim_{\F_3} {\rm Sel}^\phi(E_{\ell^2}/K) =\begin{cases} 1,& \text{ if } \ell \equiv 5 \pmod 9, \\ 2, & \text{ if } \ell \equiv 2,8 \pmod 9.\end{cases}$$\end{small}
\end{theorem}
\begin{proof}
We note that \begin{small}$S_{\ell^2}=\{2\OO_K, \ell\OO_K\}$\end{small} and \begin{small}$\OO^*_{S_{\ell^2}}=\langle \pm \zeta, 2, \ell \rangle=\langle \pm \zeta, 2, 2\ell \rangle$.\end{small} Then \begin{small}${\rm Sel}^\phi(E_{\ell^2}/K) \subset \Big(\frac{A_{S_{\ell^2}}^*}{A_{S_{\ell^2}}^{*3}}\Big)_{N=1} = \langle(\overline{\zeta}^2,\overline{\zeta}), (\overline{4},\overline{2}), (\overline{4\ell}^2,\overline{2\ell})\rangle$\end{small} 
 and \begin{small}$0 \le \dim_{\F_3} {\rm Sel}^\phi(E_{\ell^2}/K) \le 3$.\end{small}
 
From Prop. \ref{kumcharnot3sqr}(2), \begin{small}$\delta_{\phi,K_2}({E}_{\ell^2}(K_2)) \cap \big(A_2^*/A_2^{*3}\big)_{N=1}=\{(\overline{1},\overline{1})\}$\end{small} and hence \begin{small}$(\overline{\zeta}^2,\overline{\zeta}) \notin \text{Im } \delta_{\phi,K_2}$.\end{small}
Observe that \begin{small}$\delta_{\phi,K_\q}(0,\ell)=(\frac{1}{\overline{2\ell}}, \ \overline{2\ell})=(\overline{4\ell}^2, \ \overline{2\ell})$\end{small} for all \begin{small}$\q$\end{small} and hence \begin{small}$(\overline{4\ell}^2, \ \overline{2\ell}) \in {\rm Sel}^\phi(E_{\ell^2}/K)$.\end{small}  Thus, \begin{small}$1 \le \dim_{\F_3} {\rm Sel}^\phi(E_{\ell^2}/K) \le 2$\end{small} for \begin{small}$\ell \equiv 2 \pmod 3$.\end{small}

We now investigate the remaining generator \begin{small}$(\overline{4}, \ \overline{2})$\end{small} of \begin{small}$\Big(\frac{A_{S_{\ell^2}}^*}{A_{S_{\ell^2}}^{*3}}\Big)_{N=1}$.\end{small} Note that for primes \begin{small}$\q \notin S_{\ell^2} \cup \{\p\}$,\end{small} one has  \begin{small}$(\overline{4}, \ \overline{2}) \in \big(A_\q^*/A_\q^{*3}\big)_{N=1}= \text{Im } \delta_{\phi,K_\q}$\end{small} (Prop. \ref{kumcharnot3sqr}(1)). Thus it reduces to verify whether \begin{small}$(\overline{4}, \ \overline{2}) \in \text{Im } \delta_{\phi,K_\q}$\end{small} for \begin{small}$\q \in S_{\ell^2} \cup \{\p\}$.\end{small} Notice that the cubic residues \begin{small}$\big(\frac{2}{\ell}\big)_3=\big(\frac{\ell}{2}\big)_3=1$,\end{small} as \begin{small}$\ell \equiv 2 \pmod 3$.\end{small} Hence, \begin{small}$(\overline{4}, \ \overline{2})=(\overline{1},\overline{1}) = \big({A_{\ell}^*}/{A_{\ell}^{*3}}\big)_{N=1} \cap \text{Im } \delta_{\phi,K_\ell}$.\end{small} Similarly, using \begin{small}$\big(\frac{\ell}{2}\big)_3=1$,\end{small} we get \begin{small}$(\overline{4}, \ \overline{2})=(\overline{4\ell}^2, \ \overline{2\ell}) \in \text{Im } \delta_{\phi,K_2}$.\end{small} Thus, it remains to check whether or not \begin{small}$(\overline{4},\overline{2}) \in \text{Im } \delta_{\phi,K_\p}$.\end{small}

Let us first consider the primes \begin{small}$\ell \equiv 2,8 \pmod 9$.\end{small} Then \begin{small}$4\ell^2 \equiv 2\ell \equiv 1 \pmod 3$,\end{small} which gives \begin{small}$4\ell^2, 2\ell \in 1+3\Z_3 \subset U^2_{K_\p}$;\end{small} whereas \begin{small}$4\ell^2, 2\ell \notin U^3_{K_\p}$.\end{small} Note that \begin{small}$\frac{U^3_{K_\p}}{U^4_{K_\p}} \cong \frac{V_3}{A_\p^{*3}}$\end{small} by Lemma \ref{V3U3U4} and \begin{small}$\frac{V_3}{A_\p^{*3}} \subset \text{Im } \delta_{\phi,K_\p}$\end{small} by Prop. \ref{V3fortype1}. It follows that \begin{small}$\text{Im } \delta_{\phi,K_\p} = \langle (\overline{4\ell}^2,\overline{2\ell}), (\overline{(1+\p^3)^2},\overline{(1+\p^3)}) \rangle \cong \frac{U^2_{K_\p}}{U^4_{K_\p}}$.\end{small} As \begin{small}$\overline{4} \in \frac{U^2_{K_\p}}{U^4_{K_\p}}$,\end{small} we see that \begin{small}$(\overline{4},\overline{2}) \in \text{Im } \delta_{\phi,K_\p}$.\end{small} Hence, \begin{small}$\dim_{\F_3} {\rm Sel}^\phi(E_{\ell^2}/K) =2$,\end{small} when \begin{small}$\ell \equiv 2,8 \pmod 9$.\end{small}

Finally, we take \begin{small}$\ell \equiv 5 \pmod 9$.\end{small} In this situation, \begin{small}$\ell^2-1 \equiv -3 \pmod 9$, \ $\exists \ \alpha \in 1+3\Z_3$\end{small} such that \begin{small}$\ell^2-1=-3\alpha^2$.\end{small} Consequently, we obtain a point \begin{small}$P:=(-1,\p\zeta\alpha) \in E_{\ell^2}(K_\p)$.\end{small} We compute that \begin{small}$\delta_{\phi,K_\p}(P)$\end{small} is \begin{small}$(\overline{\zeta}^2, \ \overline{\zeta})$\end{small} if  \begin{small}$\ell \equiv 14 \pmod{27}$, \  $(\overline{\zeta^2(1+\p^3)^2}, \ \overline{\zeta(1+\p^3)})$\end{small} if \begin{small}$\ell \equiv 5 \pmod{27}$\end{small} and \begin{small}$(\overline{\zeta(1+\p^3)}, \ \overline{\zeta^2(1+\p^3)^2})$\end{small} if \begin{small}$\ell \equiv 23 \pmod{27}$.\end{small}  As \begin{small}$(\overline{(1+\p^3)^2}, \ \overline{(1+\p^3)}) \in \frac{V_3}{A_\p^{*3}} \subset \text{Im } \delta_{\phi,K_\p}$,\end{small} we conclude that \begin{small}$(\overline{\zeta}^2, \ \overline{\zeta}) \in \text{Im } \delta_{\phi,K_\p}$.\end{small} In fact, \begin{small}$\text{Im } \delta_{\phi,K_\p} = \langle(\overline{\zeta}^2, \ \overline{\zeta}), (\overline{(1+\p^3)^2}, \ \overline{(1+\p^3)})\rangle$.\end{small} 
 Let \begin{small}$f\big({U^2_{K_\p}}/{U^4_{K_\p}}\big)$\end{small} be the isomorphic copy of \begin{small}${U^2_{K_\p}}/{U^4_{K_\p}}$\end{small} under the isomorphism \begin{small}$f:{U^1_{K_\p}}/{U^4_{K_\p}} \isomto  \big({A_\p^*}/{A_\p^{*3}}\big)_{N=1}$.\end{small}  As  \begin{small}$\overline{\zeta} \notin {U^3_{K_\p}}/{U^4_{K_\p}}$,\end{small} we see that \begin{small}$\text{Im } \delta_{\phi,K_\p} \cap f\big(\frac{U^2_{K_\p}}{U^4_{K_\p}}\big) \cong \frac{U^3_{K_\p}}{U^4_{K_\p}}.$\end{small} Since \begin{small}$\overline{4} \in \frac{U^2_{K_\p}}{U^4_{K_\p}} \setminus \frac{U^3_{K_\p}}{U^4_{K_\p}}$,\end{small} we deduce that \begin{small}$(\overline{4},\overline{2}) \notin \text{Im } \delta_{\phi,K_\p}$\end{small} and hence, is not in \begin{small}${\rm Sel}^\phi(E_{\ell^2}/K)$.\end{small} 
 
 We have shown above for \begin{small}$\ell \equiv 5 \pmod 9$\end{small} that \begin{small}$(\overline{4\ell^2}, \ \overline{2\ell}) \in {\rm Sel}^\phi(E_{\ell^2}/K)$\end{small} and \begin{small}$(\overline{4}, \ \overline{2}), (\overline{\zeta}^2, \ \overline{\zeta}) \notin {\rm Sel}^\phi(E_{\ell^2}/K)$.\end{small} There are four subgroups of \begin{small}$\Big(\frac{A_{S_{\ell^2}}^*}{A_{S_{\ell^2}}^{*3}}\Big)_{N=1}$\end{small} of order $9$ containing \begin{small}$(\overline{4\ell}^2, \ \overline{2\ell})$.\end{small} 
To conclude that \begin{small}$\dim_{\F_3} {\rm Sel}^\phi(E_{\ell^2}/K) = 1$,\end{small} it suffices to show that 
\begin{small}$\langle (\overline{4\zeta}^2, \ \overline{2\zeta}) \rangle \not\subset {\rm Sel}^\phi(E_{\ell^2}/K)$\end{small} and \begin{small}$\langle (\overline{2\zeta}^2, \ \overline{4\zeta}) \rangle \not\subset {\rm Sel}^\phi(E_{\ell^2}/K)$.\end{small} First we consider \begin{small}$(\overline{4\zeta}^2, \ \overline{2\zeta})$.\end{small} As \begin{small}$\big(\frac{2}{\ell}\big)_3=1$,\end{small} we identify \begin{small}$(\overline{4\zeta}^2, \ \overline{2\zeta})$\end{small} with \begin{small}$(\overline{\zeta}^2, \ \overline{\zeta})$\end{small} in \begin{small}$\big(\frac{A_\ell^*}{A_\ell^{*3}}\big)_{N=1}.$\end{small} Recall that \begin{small}$\text{Im } \delta_{\phi,K_\ell} \cap \big(\frac{A_\ell^*}{A_\ell^{*3}}\big)_{N=1} = \{(\overline{1},\overline{1})\}$\end{small} by Prop. \ref{kumcharnot3sqr}(2).  Given \begin{small}$\ell \equiv 5 \pmod 9$,\end{small} we have \begin{small}$(\overline{4\zeta}^2, \ \overline{2\zeta})=(\overline{\zeta}^2, \ \overline{\zeta}) \notin \text{Im } \delta_{\phi,K_\ell}$\end{small} and so it is not in \begin{small}${\rm Sel}^\phi(E_{\ell^2}/K)$.\end{small} A similar argument works for \begin{small}$(\overline{2\zeta}^2, \ \overline{4\zeta})$.\end{small} This shows \begin{small}$\dim_{\F_3} {\rm Sel}^\phi(E_{\ell^2}/K) =1$,\end{small} when \begin{small}$\ell \equiv 5 \pmod 9$.\end{small}
\end{proof}

\noindent We use a similar method as in Theorem \ref{sylvester2} to obtain \begin{small}$\dim_{\F_3}{\rm Sel}^\phi(E_{\ell^4}/K)$\end{small}  for \begin{small}$\ell \equiv 2 \pmod 3$\end{small}:
\begin{theorem}\label{sylvester2psqr}
Let \begin{small}$\ell \ge 5$\end{small} be a prime such that \begin{small}$\ell \equiv 2 \pmod 3$\end{small} and \begin{small}$\phi: E_{\ell^4} \to E_{\ell^4}$\end{small} be the $3$-isogeny in \eqref{eq:defofphi}. Then \begin{small}$$\dim_{\F_3} {\rm Sel}^\phi(E_{\ell^4}/K) =\begin{cases} 1,& \text{ if } \ell \equiv 2 \pmod 9, \\ 2, & \text{ if } \ell \equiv 5,8 \pmod 9. \end{cases}$$\end{small}
\end{theorem}
\begin{proof}
The proof differs from  Theorem \ref{sylvester2} in showing \begin{small}$(\overline{4},\overline{2}) \in \text{Im } \delta_{\phi,K_\p}$\end{small} for \begin{small}$\ell \equiv 5,8\pmod 9$.\end{small} In these cases, one needs to choose a suitable \begin{small}$\alpha$\end{small}  in order to compute \begin{small}$\text{Im } \delta_{\phi,K_\p}$.\end{small}
\end{proof}

In the following theorem, we compute \begin{small}$\dim_{\F_3} {\rm Sel}^\phi(E_{\ell^2}/K)$\end{small} for primes \begin{small}$\ell \equiv 1 \pmod 3$.\end{small}
\begin{theorem}\label{sylvester3}
Let \begin{small}$\ell \equiv 1 \pmod 3$\end{small} be a prime and \begin{small}$\phi: E_{\ell^2} \to E_{\ell^2}$\end{small} be the $3$-isogeny given in \eqref{eq:defofphi}. Let the prime \begin{small}$\ell$\end{small} split as \begin{small}$\pi_\ell \cdot \pi'_\ell$\end{small} in \begin{small}$\OO_K$\end{small} and \begin{small}$\big(\frac{\cdot}{\cdot}\big)_3$\end{small} be the cubic residue symbol. 
Then \begin{small}$$\dim_{\F_3} {\rm Sel}^\phi(E_{\ell^2}/K) =\begin{cases} 1, & \text{ if } \ell \equiv 1,7 \pmod 9 \text{ and } \big(\frac{2}{\pi_{\ell}}\big)_3 \neq 1, \\ 2, & \text{ if } \ell \equiv 4 \pmod 9, \\ 3, & \text{ if } \ell \equiv 1,7 \pmod 9 \text{ and } \big(\frac{2}{\pi_{\ell}}\big)_3 = 1.\end{cases}$$\end{small}
\end{theorem}
\begin{proof}
We write \begin{small}$\ell=\pi_\ell \cdot \pi'_\ell$\end{small} such that \begin{small}$\pi_\ell=m+n\zeta$.\end{small} In this case, it is possible to choose \begin{small}$m \equiv 1 \pmod 3$\end{small} and \begin{small}$n \equiv 0 \pmod 3$.\end{small} Here, \begin{small}$S_{\ell^2}=\{2\OO_K, \pi_\ell, \pi'_\ell\}$\end{small} and  \begin{small}$\OO^*_{S_{\ell^2}}=\langle \pm \zeta, 2, \pi_\ell, \pi'_\ell \rangle=\langle \pm \zeta, 2, \pi_\ell, 2\ell \rangle$.\end{small} Also \begin{small}${\rm Sel}^\phi(E_{\ell^2}/K) \subset \Big(\frac{A_{S_{\ell^2}}^*}{A_{S_{\ell^2}}^{*3}}\Big)_{N=1}=\langle(\overline{\zeta}^2, \ \overline{\zeta}), \  (\overline{4}, \ \overline{2}), \  (\overline{\pi_\ell}^2, \ \overline{\pi_\ell}), \  (\overline{4\ell}^2, \ \overline{2\ell})\rangle$\end{small} and \begin{small}$\dim_{\F_3} {\rm Sel}^\phi(E_{\ell^2}/K) \le 4$.\end{small} As in the previous case, \begin{small}$(\overline{4\ell}^2, \ \overline{2\ell}) \in {\rm Sel}^\phi(E_{\ell^2}/K)$,\end{small} being the image of \begin{small}$(0, \ \ell)$\end{small} under \begin{small}$\delta_{\phi,K_\q}$\end{small} for all \begin{small}$\q$.\end{small} On the other hand, \begin{small}$(\overline{\zeta}^2, \ \overline{\zeta}) \notin \text{Im } \delta_{\phi,K_2}$\end{small} and hence is not in \begin{small}${\rm Sel}^\phi(E_{\ell^2}/K)$.\end{small} This gives \begin{small}$1 \le \dim_{\F_3} {\rm Sel}^\phi(E_{\ell^2}/K) \le 3$.\end{small} Moreover, for primes \begin{small}$\q \notin S_{\ell^2} \cup \{\p\},$\end{small} \begin{small}$(\overline{4}, \ \overline{2}), \ (\overline{\pi_\ell}^2, \ \overline{\pi_\ell}) \in \big({A_\q^*}/{A_\q^{*3}}\big)_{N=1} =\text{Im } \delta_{\phi,K_\q}$.\end{small}

First, consider the case \begin{small}$\ell \equiv 1,7 \pmod 9$\end{small} and \begin{small}$\big(\frac{2}{\pi_{\ell}}\big)_3 = 1$.\end{small} We now  check whether \begin{small}$(\overline{4}, \ \overline{2})$,\end{small}  \begin{small}$(\overline{\pi_\ell}^2, \ \overline{\pi_\ell}) \in \text{Im } \delta_{\phi,K_\q}$\end{small} at primes \begin{small}$\q \in S_{\ell^2} \cup \{\p\}$; \end{small} we start with \begin{small}$(\overline{4}, \ \overline{2})$.\end{small} Our assumption \begin{small}$\big(\frac{2}{\pi_{\ell}}\big)_3 = 1$\end{small} implies \begin{small}$\big(\frac{2}{\pi'_{\ell}}\big)_3 = 1$.\end{small} Thus, \begin{small}$(\overline{4}, \ \overline{2})=(\overline{1},\overline{1}) \in \text{Im } \delta_{\phi,K_{\pi_\ell}}$\end{small} and \begin{small}$\text{Im } \delta_{\phi,K_{\pi'_\ell}}$.\end{small} Using the law of cubic reciprocity, \begin{small}$\big(\frac{\pi_{\ell}}{2}\big)_3 =\big(\frac{\pi'_{\ell}}{2}\big)_3 = 1$.\end{small} This implies \begin{small}$\big(\frac{\ell}{2}\big)_3 = 1$\end{small} and so, \begin{small}$(\overline{4}, \ \overline{2})=(\overline{4\ell}^2, \ \overline{2\ell}) \in \text{Im } \delta_{\phi,K_{2}}$.\end{small} Now, following the proof of Theorem \ref{sylvester2}, \begin{small}$\ell \equiv 2,8 \pmod 9$\end{small} case, we see that \begin{small}$\text{Im } \delta_{\phi,K_\p} = \langle (\overline{4\ell}^2, \ \overline{2\ell}), \  (\overline{(1+\p^3)^2}, \ \overline{(1+\p^3)}) \rangle \cong \frac{U^2_{K_\p}}{U^4_{K_\p}}$.\end{small} 
As \begin{small}$4 \in U^2_{K_\p}$,\end{small} \begin{small}$(\overline{4}, \ \overline{2}) \in \text{Im } \delta_{\phi,K_\p}$\end{small} and thus, it is in \begin{small}${\rm Sel}^\phi(E_{\ell^2}/K)$.\end{small} Next, we investigate \begin{small}$(\overline{\pi_\ell}^2, \ \overline{\pi_\ell})$.\end{small} Again by the trick of Evan, \begin{small}$\big(\frac{\pi_{\ell}}{\pi'_\ell}\big)_3 = 1$,\end{small} we get \begin{small}$(\overline{\pi_\ell}^2, \ \overline{\pi_\ell})=(\overline{1},\overline{1}) \in \text{Im } \delta_{\phi,K_{\pi'_\ell}}$.\end{small} Further, \begin{small}$\big(\frac{\pi_{\ell}}{2}\big)_3 = 1$\end{small} gives \begin{small}$(\overline{\pi_\ell}^2, \ \overline{\pi_\ell})=(\overline{1},\overline{1}) \in \text{Im } \delta_{\phi,K_{2}}$.\end{small} These two observations also imply that \begin{small}$(\overline{\pi_\ell}^2, \ \overline{\pi_\ell})=(\overline{4\ell}^2, \ \overline{2\ell}) \in \text{Im } \delta_{\phi,K_{\pi_\ell}}$.\end{small} Finally, using \begin{small}$\pi_\ell=m+n\zeta$,\end{small} where \begin{small}$m \equiv 1 \pmod 3$\end{small} and \begin{small}$n \equiv 0 \pmod 3$,\end{small} we get \begin{small}$\pi_\ell \equiv 1 \pmod{\p^2}$\end{small} i.e. \begin{small}$\pi_\ell \in U^2_{K_\p}$.\end{small} Now using  \begin{small}$\text{Im } \delta_{\phi,K_\p} \cong \frac{U^2_{K_\p}}{U^4_{K_\p}}$,\end{small} we have \begin{small}$(\overline{\pi_\ell}^2, \ \overline{\pi_\ell}) \in \text{Im } \delta_{\phi,K_\p}$\end{small} and hence \begin{small}$\dim_{\F_3} {\rm Sel}^\phi(E_{\ell^2}/K)=3$.\end{small} 

Next, consider the case \begin{small}$\ell \equiv 1,7 \pmod 9$\end{small} and \begin{small}$\big(\frac{2}{\pi_{\ell}}\big)_3 \neq 1$.\end{small} We already have \begin{small}$(\overline{4\ell}^2, \ \overline{2\ell}) \in {\rm Sel}^\phi(E_{\ell^2}/K)$.\end{small} There are $13$ distinct subgroups of \begin{small}$\Big(\frac{A_{S_{\ell^2}}^*}{A_{S_{\ell^2}}^{*3}}\Big)_{N=1}$\end{small} of order $9$ containing \begin{small}$(\overline{4\ell}^2, \ \overline{2\ell})$.\end{small}
There are $13$ generators corresponding to order $3$ subgroups of \begin{small}$\Big(\frac{A_{S_{\ell^2}}^*}{A_{S_{\ell^2}}^{*3}}\Big)_{N=1} \Big/ \big\langle (\overline{4\ell}^2, \ \overline{2\ell}) \big\rangle$\end{small} in \begin{small}$\Big(\frac{A_{S_{\ell^2}}^*}{A_{S_{\ell^2}}^{*3}}\Big)_{N=1}$\end{small} viz. \begin{small}$(\overline{4}, \ \overline{2})$, \ $(\overline{\zeta}^2, \ \overline{\zeta})$, \ $(\overline{\pi_\ell}^2, \ \overline{\pi_\ell})$, \ $(\overline{4\zeta^2}, \ \overline{2\zeta})$, \ $(\overline{2\zeta^2}, \ \overline{4\zeta})$, \ $(\overline{4\pi^2_\ell}, \ \overline{2\pi_\ell})$, \ $(\overline{4\pi_\ell}, \ \overline{2\pi^2_\ell})$, \ $(\overline{\zeta^2\pi^2_\ell}, \ \overline{\zeta\pi_\ell})$, \ $(\overline{\zeta\pi^2_\ell}, \ \overline{\zeta^2\pi_\ell})$, \ $(\overline{4\zeta^2\pi^2_\ell}, \ \overline{2\zeta\pi_\ell})$, \ $(\overline{2\zeta^2\pi^2_\ell}, \ \overline{4\zeta\pi_\ell})$, \ $(\overline{4\zeta\pi^2_\ell}, \ \overline{2\zeta^2\pi_\ell})$\end{small} and \begin{small}$(\overline{2\zeta\pi^2_\ell}, \ \overline{4\zeta^2\pi_\ell})$.\end{small} A long and tedious calculation using the theory of cubic residues shows that \begin{small}$(\overline{\zeta}^2, \ \overline{\zeta}), \ (\overline{\pi^2_\ell}, \ \overline{\pi_\ell}) \notin \text{Im } \delta_{\phi,K_{2}}$,\end{small} whereas \begin{small}$(\overline{4}, \ \overline{2}), \ (\overline{4\pi^2_\ell}, \ \overline{2\pi_\ell}), \  (\overline{4\pi_\ell}, \ \overline{2\pi^2_\ell}) \notin \text{Im } \delta_{\phi,K_{\pi'_\ell}}$\end{small} and the rest of the generators are not in \begin{small}$\text{Im } \delta_{\phi,K_{\p}}$.\end{small} So, none of these $13$ generators belong to \begin{small}${\rm Sel}^\phi(E_{\ell^2}/K)$.\end{small} Thus we conclude  \begin{small}$\dim_{\F_3} {\rm Sel}^\phi(E_{\ell^2}/K)=1$\end{small} in this case.

Finally, we consider the case \begin{small}$\ell \equiv 4 \pmod 9$.\end{small} We already have \begin{small}$(\overline{4\ell^2}, \ \overline{2\ell}) \in {\rm Sel}^\phi(E_{\ell^2}/K)$.\end{small} We divide this case into $3$ subcases, depending on whether the cubic residue \begin{small}$\big(\frac{2}{\pi_{\ell}}\big)_3$\end{small} is equal to $1$, \begin{small}$\zeta$\end{small} or \begin{small}$\zeta^2$.\end{small} In each of these subcases, using the theory of cubic residues, we examine all the $13$ generators mentioned above and observe the following:

If \begin{small}$\big(\frac{2}{\pi_{\ell}}\big)_3=1$,\end{small} then \begin{small}$(\overline{4\pi_\ell}, \ \overline{2\pi^2_\ell}) \in {\rm Sel}^\phi(E_{\ell^2}/K)$,\end{small} whereas \begin{small}$(\overline{4}, \ \overline{2})$, \ $(\overline{\pi^2_\ell}, \ \overline{\pi_\ell}), \ (\overline{4\pi^2_\ell}, \ \overline{2\pi_\ell}) \notin \text{Im } \delta_{\phi,K_{\p}}$\end{small} and the rest of the generators are not in \begin{small}$\text{Im } \delta_{\phi,K_{\pi'_\ell}}$.\end{small}

If \begin{small}$\big(\frac{2}{\pi_{\ell}}\big)_3=\zeta$,\end{small} then \begin{small}$(\overline{2\zeta\pi^2_\ell}, \ \overline{4\zeta^2\pi_\ell}) \in {\rm Sel}^\phi(E_{\ell^2}/K)$,\end{small} whereas \begin{small}$(\overline{4\zeta^2}, \ \overline{2\zeta}), \ (\overline{4\zeta^2\pi^2_\ell}, \ \overline{2\zeta\pi_\ell}) \notin \text{Im } \delta_{\phi,K_{\pi_\ell}}$\end{small} and \begin{small}$(\overline{\pi^2_\ell}, \ \overline{\pi_\ell}) \notin \text{Im } \delta_{\phi,K_{2}}$,\end{small} while the rest of the generators are not in \begin{small}$\text{Im } \delta_{\phi,K_{\pi'_\ell}}$.\end{small}

If \begin{small}$\big(\frac{2}{\pi_{\ell}}\big)_3=\zeta^2$,\end{small} then \begin{small}$(\overline{2\zeta^2\pi^2_\ell}, \ \overline{4\zeta\pi_\ell}) \in {\rm Sel}^\phi(E_{\ell^2}/K)$\end{small}, whereas \begin{small}$(\overline{2\zeta^2}, \ \overline{4\zeta}), \  (\overline{4\zeta\pi^2_\ell}, \ \overline{2\zeta^2\pi_\ell}) \notin \text{Im } \delta_{\phi,K_{\pi_\ell}}$\end{small} and \begin{small}$(\overline{\pi^2_\ell},\overline{\pi_\ell}) \notin \text{Im } \delta_{\phi,K_{2}}$,\end{small} while the rest of the generators are not in \begin{small}$\text{Im } \delta_{\phi,K_{\pi'_\ell}}$.\end{small}

Thus, for each of these cases, 
\begin{small}$\dim_{\F_3} {\rm Sel}^\phi(E_{\ell^2}/K)=2$,\end{small} when \begin{small}$\ell \equiv 4 \pmod 9$.\end{small}
\end{proof}

The following theorem computes \begin{small}$\dim_{\F_3} {\rm Sel}^\phi(E_{\ell^4}/K)$\end{small} for primes \begin{small}$\ell \equiv 1 \pmod 3$,\end{small} using a method similar to that of Theorem \ref{sylvester3}.
\begin{theorem}\label{sylvester2psqr2}
Let \begin{small}$\ell \equiv 1 \pmod 3$\end{small} be a prime and \begin{small}$\phi: E_{\ell^4} \to E_{\ell^4}$\end{small} be the $3$-isogeny given in \eqref{eq:defofphi}. Let the prime \begin{small}$\ell$\end{small}  split as \begin{small}$\pi_\ell \cdot  \pi'_\ell$\end{small} in \begin{small}$\OO_K$\end{small} and \begin{small}$\big(\frac{\cdot}{\cdot}\big)_3$\end{small} be the cubic residue symbol.  
Then \begin{small}$$\dim_{\F_3} {\rm Sel}^\phi(E_{\ell^4}/K) =\begin{cases} 1, & \text{ if } \ell \equiv 1,4 \pmod 9 \text{ and } \big(\frac{2}{\pi_{\ell}}\big)_3 \neq 1, \\ 2, & \text{ if } \ell \equiv 7 \pmod 9, \\ 3, & \text{ if } \ell \equiv 1,4 \pmod 9 \text{ and } \big(\frac{2}{\pi_{\ell}}\big)_3 = 1.\end{cases}$$\end{small}
\end{theorem}
\begin{proof}
In the case, when \begin{small}$\ell \equiv 1,4 \pmod 9$\end{small} and \begin{small}$\big(\frac{2}{\pi_\ell}\big)_{3} \neq 1$,\end{small} the proof is similar to that of \begin{small}$\ell \equiv 1,7 \pmod 9$\end{small} and \begin{small}$\big(\frac{2}{\pi_\ell}\big)_{3} \neq 1$\end{small} in Theorem \ref{sylvester3}. Also, when \begin{small}$\ell \equiv 1,4 \pmod 9$\end{small} and \begin{small}$\big(\frac{2}{\pi_\ell}\big)_{3} = 1$,\end{small} the proof that \begin{small}$\dim_{\F_3} {\rm Sel}^\phi(E_{\ell^4}/K)=3$\end{small} is similar to that of \begin{small}$\ell \equiv 1,7 \pmod 9$\end{small} and \begin{small}$\big(\frac{2}{\pi_\ell}\big)_{3} = 1$\end{small} in Theorem \ref{sylvester3}. 

Finally, in the case \begin{small}$\ell \equiv 7 \pmod 9$,\end{small} 
    the proof differs from Theorem \ref{sylvester3} in examining the $13$ generators corresponding to the order $3$ subgroups of  \begin{small}$\Big(\frac{A_{S_{\ell^4}}^*}{A_{S_{\ell^4}}^{*3}}\Big)_{N=1} \Big/ \langle (\overline{4\ell},\overline{2\ell}^2) \rangle$\end{small} in \begin{small}$\Big(\frac{A_{S_{\ell^4}}^*}{A_{S_{\ell^4}}^{*3}}\Big)_{N=1}$\end{small} mentioned therein. When \begin{small}$\ell \equiv 7 \pmod 9$,\end{small} we have:

    If \begin{small}$\big(\frac{2}{\pi_{\ell}}\big)_3=1$,\end{small} then \begin{small}$(\overline{4\pi^2_\ell}, \ \overline{2\pi_\ell}) \in {\rm Sel}^\phi(E_{\ell^4}/K)$,\end{small} whereas \begin{small}$(\overline{4}, \ \overline{2})$, \ $(\overline{\pi_\ell}, \ \overline{\pi^2_\ell}), \ (\overline{4\pi_\ell}, \ \overline{2\pi^2_\ell}) \notin \text{Im } \delta_{\phi,K_{\p}}$\end{small} and the rest of the generators are not in \begin{small}$\text{Im } \delta_{\phi,K_{\pi'_\ell}}$.\end{small}

If \begin{small}$\big(\frac{2}{\pi_{\ell}}\big)_3=\zeta$,\end{small} then \begin{small}$(\overline{4\zeta\pi^2_\ell}, \ \overline{2\zeta^2\pi_\ell}) \in {\rm Sel}^\phi(E_{\ell^4}/K)$\end{small}, whereas \begin{small}$(\overline{2\zeta^2}, \ \overline{4\zeta}), \  (\overline{2\zeta^2\pi^2_\ell}, \ \overline{4\zeta\pi_\ell}) \notin \text{Im } \delta_{\phi,K_{\pi_\ell}}$\end{small} and \begin{small}$(\overline{\pi_\ell}, \ \overline{\pi^2_\ell}) \notin \text{Im } \delta_{\phi,K_{2}}$,\end{small} while the rest of the generators are not in \begin{small}$\text{Im } \delta_{\phi,K_{\pi'_\ell}}$.\end{small}

If \begin{small}$\big(\frac{2}{\pi_{\ell}}\big)_3=\zeta^2$,\end{small} then \begin{small}$(\overline{4\zeta^2\pi^2_\ell}, \ \overline{2\zeta\pi_\ell}) \in {\rm Sel}^\phi(E_{\ell^4}/K)$\end{small}, whereas \begin{small}$(\overline{\pi_\ell}, \ \overline{\pi^2_\ell}), \  (\overline{2\zeta\pi^2_\ell}, \ \overline{4\zeta^2\pi_\ell}) \notin \text{Im } \delta_{\phi,K_{2}}$\end{small} and \begin{small}$(\overline{2\zeta^2\pi^2_\ell}, \ \overline{4\zeta\pi_\ell}) \notin \text{Im } \delta_{\phi,K_{\pi_\ell}}$,\end{small} while the rest of the generators are not in \begin{small}$\text{Im } \delta_{\phi,K_{\pi'_\ell}}$.\end{small}

Thus for each of these cases, \begin{small}$\dim_{\F_3} {\rm Sel}^\phi(E_{\ell^4}/K)=2$,\end{small} when \begin{small}$\ell \equiv 7 \pmod 9$.\end{small}
\end{proof}

\noindent Now we are ready to discuss the rational cube sum problems for \begin{small}$2\ell$\end{small}  and  \begin{small}$2\ell^2$,\end{small} where \begin{small}$\ell$\end{small} is a prime. Recall from the proofs of Theorems \ref{sylvester3} and  \ref{sylvester2psqr2} that \begin{small}$\ell \equiv 1 \pmod 3$\end{small}  splits as \begin{small}$\ell =\pi_\ell \cdot \pi_\ell'$\end{small} in \begin{small}$\OO_K$.\end{small} This will be used in the next two corollaries.  
We start with \begin{small}$D=2\ell$\end{small} and use Theorems \ref{sylvester2} and \ref{sylvester3} to compute \begin{small}$\text{rk } E_{-432\cdot(2\ell)^2}(\Q) = \text{rk } E_{\ell^2}(\Q)$.\end{small}

\begin{corollary}\label{cortosylvester23} 
Consider the cube sum problem for \begin{small}$D=2\ell$\end{small} for a prime \begin{small}$\ell \ge 5$.\end{small}
\begin{enumerate}
    \item If \begin{small}$\ell \equiv 5 \pmod 9$,\end{small} then \begin{small}$\text{rk } E_{-432D^2}(\Q) =0$\end{small} and   \begin{small}$\Sh(E_{-432D^2}/\Q)[3]=0$.\end{small}
    \item If \begin{small}$\ell \equiv 1,7 \pmod 9$\end{small} and \begin{small}$\big(\frac{2}{\pi_{\ell}}\big)_3 \neq 1$,\end{small} then \begin{small}$\text{rk } E_{-432D^2}(\Q)  =0$\end{small} and  \begin{small}$\Sh(E_{-432D^2}/\Q)[3]=0$.\end{small}
    \item If \begin{small}$\ell \equiv 2,4,8 \pmod 9$,\end{small} then either \begin{small}$\text{rk } E_{-432D^2}(\Q) =1$\end{small} or \begin{small}$\Sh(E_{-432D^2}/\Q)[3^\infty] \cong \Q_3/\Z_3$.\end{small} 
    
    \noindent In particular, if we assume that \begin{small}$\dim_{\F_3}\Sh(E_{-432D^2}/\Q)[3]$\end{small} is even, then \begin{small}$\text{rk } E_{-432D^2}(\Q) =1$\end{small} and  in fact, \begin{small}$\Sh(E_{-432D^2}/\Q)[3]=0$.\end{small}
    \item If \begin{small}$\ell \equiv 1,7 \pmod 9$\end{small}  and  \begin{small}$\big(\frac{2}{\pi_{\ell}}\big)_3 = 1$,\end{small} then one of the following holds:
    \begin{enumerate}
        \item either \begin{small}$\text{rk } E_{-432D^2}(\Q)  \in \{0,2\}$\end{small} or,
        \item \begin{small}$\text{rk } E_{-432D^2}(\Q) =1$\end{small} and \begin{small}$\Sh(E_{-432D^2}/\Q)[3^\infty] \cong \Q_3/\Z_3$.\end{small}
    \end{enumerate}  
    In particular, if we assume that \begin{small}$\dim_{\F_3}\Sh(E_{-432D^2}/\Q)[3]$\end{small} is even, then \begin{small}$\text{rk } E_{-432D^2}(\Q) \in \{0,2\}$.\end{small} 
\end{enumerate}
\end{corollary}

\begin{proof}
Take \begin{small}$a=-27\ell^2$\end{small} and let  \begin{small}$\widehat{E}_a := E_{\ell^2}$\end{small} and \begin{small}$\phi_a:E_a \to \widehat{E}_a$,\end{small} \begin{small}$\widehat{\phi}_a:\widehat{E}_a \to E_a$\end{small} be the rational $3$-isogenies defined in \S\ref{iso}. 
Note that \begin{small}$|E_a(\Q)[3]|=|E_a(\Q)[\phi_a]|=1$\end{small} and \begin{small}$|\widehat{E}_a(\Q)[3]|=|\widehat{E}_a(\Q)[\widehat{\phi}_a]|=3$.\end{small} Setting \begin{small}$R:=\frac{\Sh(\widehat{E}_{a}/\Q)[\widehat{\phi}_a]}{{\phi}_a(\Sh({E}_a/\Q)[3])}$,\end{small} 
from the exact sequence \eqref{eq:selmerseqoverQ}, we have \begin{small}$\dim_{\F_3} {\rm Sel}^3({E}_a/\Q) = \dim_{\F_3} {\rm Sel}^{\phi}({E}_a/K)-\dim_{\F_3} R-1 $\end{small}. Since \begin{small}$\dim_{\F_3} R$\end{small} is even by  \cite[Prop.~49]{bes}, by Theorem \ref{sylvester2} and Theorem \ref{sylvester3} we obtain 
\begin{small}$$\text{rk } {E}_a(\Q) + \dim_{\F_3} \Sh({E}_a/\Q)[3] = \dim_{\F_3} {\rm Sel}^3({E}_{a}/\Q) = \begin{cases} 0, & \text{if } \ell \equiv 5 \pmod 9 \text{ or } \\ & \quad \ell \equiv 1,7 \pmod 9 \text{ and } \big(\frac{2}{\pi_\ell}\big)_3 \neq 1, \\ 1, & \text{if } \ell \equiv 2,4,8 \pmod 9, \\ 0 \text{ or } 2, & \text{if } \ell \equiv 1,7 \pmod 9 \text{ and } \big(\frac{2}{\pi_\ell}\big)_3 = 1. \end{cases}$$\end{small}
Hence, \begin{small}$\text{rk } {E}_a(\Q) = 0$\end{small} and also \begin{small}$\Sh({E}_a/\Q)[3]=0$\end{small} if either \begin{small}$\ell \equiv 5 \pmod 9$\end{small} or \begin{small}$\ell \equiv 1,7 \pmod 9$\end{small} and \begin{small}$\big(\frac{2}{\pi_\ell}\big)_3 \neq 1$.\end{small}

Now if \begin{small}$\ell \equiv 2,4,8 \pmod 9$,\end{small} then \begin{small}$\dim_{\F_3}{\rm Sel}^3({E}_a/\Q)=1$\end{small} implies either \begin{small}$\text{rk } {E}_{a}(\Q)=1$\end{small} or \begin{small}$\dim_{\F_3}\Sh({E}_a/\Q)[3]=1$.\end{small} It is easy to see that \begin{small}$\dim_{\F_3}\Sh({E}_a/\Q)[3]=1$\end{small} implies \begin{small}$\Sh({E}_a/\Q)[3^\infty] \cong \Q_3/\Z_3$.\end{small}

Finally, when \begin{small}$\ell \equiv 1,7 \pmod 9$\end{small} and \begin{small}$\big(\frac{2}{\pi_\ell}\big)_3 = 1$,\end{small} if \begin{small}$\dim_{\F_3}{\rm Sel}^3({E}_a/\Q)=0$,\end{small} then \begin{small}$\text{rk } {E}_a(\Q) = \Sh({E}_a/\Q)[3]=0$.\end{small} On the other hand, if
\begin{small}$\dim_{\F_3}{\rm Sel}^3({E}_a/\Q)= 2$,\end{small} then either \begin{small}$\text{rk } E_{a}(\Q) \in \{0,2\}$\end{small} or \begin{small}$\text{rk } E_{a}(\Q)=1$\end{small} and \begin{small}$\dim_{\F_3}\Sh(E_{a}/\Q)[3] =1$;\end{small} equivalently, either \begin{small}$\text{rk } E_{a}(\Q) \in \{0,2\}$\end{small} or \begin{small}$\Sh(E_{a}/\Q)[3^\infty] \cong \Q_3/\Z_3$.\end{small}
\end{proof}

\noindent Using Theorems  \ref{sylvester2psqr} and \ref{sylvester2psqr2}, we calculate \begin{small}$\text{rk } E_{-432\cdot(2\ell^2)^2}(\Q)=\text{rk } E_{\ell^4}(\Q)$.\end{small}
The proof is similar to that of Corollary \ref{cortosylvester23} and is omitted.

\begin{corollary}\label{cortosylvester2psqr}
We discuss the cube sum problem for \begin{small}$D=2\ell^2$\end{small} for a prime \begin{small}$\ell \ge 5$.\end{small}
\begin{enumerate}
    \item If \begin{small}$\ell \equiv 2 \pmod 9$,\end{small} then \begin{small}$\text{rk } E_{-432D^2}(\Q) =0$\end{small} and   \begin{small}$\Sh(E_{-432D^2}/\Q)[3]=0$.\end{small}
    \item If \begin{small}$\ell \equiv 1,4 \pmod 9$\end{small} and \begin{small}$\big(\frac{2}{\pi_{\ell}}\big)_3 \neq 1$,\end{small} then \begin{small}$\text{rk } E_{-432D^2}(\Q)  =0$\end{small} and  \begin{small}$\Sh(E_{-432D^2}/\Q)[3]=0$.\end{small}
    \item If \begin{small}$\ell \equiv 5,7,8 \pmod 9$,\end{small} then either \begin{small}$\text{rk } E_{-432D^2}(\Q) =1$\end{small} or \begin{small}$\Sh(E_{-432D^2}/\Q)[3^\infty] \cong \Q_3/\Z_3$.\end{small} 
    
    \noindent In particular, if we assume that \begin{small}$\dim_{\F_3}\Sh(E_{-432D^2}/\Q)[3]$\end{small} is even, then \begin{small}$\text{rk } E_{-432D^2}(\Q) =1$\end{small} and  in fact, \begin{small}$\Sh(E_{-432D^2}/\Q)[3]=0$.\end{small}
    \item If \begin{small}$\ell \equiv 1,4 \pmod 9$\end{small} and \begin{small}$\big(\frac{2}{\pi_{\ell}}\big)_3 = 1$,\end{small} then one of the following holds:
    \begin{enumerate}
        \item either \begin{small}$\text{rk } E_{-432D^2}(\Q)  \in \{0,2\}$\end{small} or,
        \item \begin{small}$\text{rk } E_{-432D^2}(\Q) =1$\end{small} and \begin{small}$\Sh(E_{-432D^2}/\Q)[3^\infty] \cong \Q_3/\Z_3$.\end{small}
    \end{enumerate}  
    In particular, if we assume that  \begin{small}$\dim_{\F_3}\Sh(E_{-432D^2}/\Q)[3]$\end{small} is even, then \begin{small}$\text{rk } E_{-432D^2}(\Q) \in \{0,2\}$.\end{small} \qed
\end{enumerate}
\end{corollary}

\begin{rem}[The assumption on even rank of  {$\Sh(E_{16D^2}/\Q)[3]$}]\label{shaevenremark}
 Let \begin{small}$\Sh(E/\Q)[3^\infty]_{\rm div}$\end{small} denote the maximal divisible subgroup of \begin{small}$\Sh(E/\Q)[3^\infty]$.\end{small} 
Then \begin{small}$\Sh(E/\Q) \cong \Sh(E/\Q)_{{\rm div}} \oplus \big(\Sh(E/\Q)/\Sh(E/\Q)_{{\rm div}}\big)$.\end{small}
By Cassels-Tate pairing, it is well-known that \begin{small}$\dim_{\F_3} \big(\Sh(E/\Q)[3^\infty]/\Sh(E/\Q)[3^\infty]_{{\rm div}}\big)[3]$\end{small} is even. It follows from the above discussion:\\
\begin{small}$\dim_{\F_3} \Sh(E/\Q)_{\rm}[3]$\end{small} is even \begin{small}$\Longleftrightarrow \dim_{\F_3} \Sh(E/\Q)_{\rm div}[3]$\end{small} is even \begin{small}$\Longleftrightarrow \Sh(E/\Q)_\mathrm{div}[3^\infty] \cong \big(\Q_3/\Z_3\big)^{2t}$\end{small} for some $t \ge 0.$

Of course, a part of the BSD conjecture predicts that \begin{small}$\Sh(E/\Q)$\end{small} is finite, in particular \begin{small}$\Sh(E/\Q)[3^\infty]$\end{small} is finite i.e. $t=0$ above. Then it is true that \begin{small}$\dim_{\F_3}\Sh(E/\Q)[3]$\end{small} is even.
Note that  the assumption that  \begin{small}$\dim_{\F_2}\Sh(E/\Q)[2]$\end{small} is even also appears in Mazur-Rubin's work \cite[Conjecture~$\Sh T_2(K)$]{mr}.  
\end{rem}

\begin{rem}\label{adhochsharemove}
We consider the cube sum problem for \begin{small}$D=2\ell$\end{small} and take \begin{small}$\ell$\end{small} to be a prime, \begin{small}$\ell \equiv 1 \pmod 3$.\end{small} The assumption in Corollary \ref{cortosylvester23} that \begin{small}$\dim_{\F_3}\Sh(E_{-432D^2}/\Q)[3]$\end{small} is even, can be removed in the following particular families of elliptic curves, for which we produce explicit rational points of infinite order:
\begin{enumerate}
	\item Let \begin{small}$\ell$\end{small} be of the form \begin{small}$\ell=t^2+27$,\end{small}
	then \begin{small}$(\ell,t\ell) \in E_{-432D^2}(\Q)$\end{small} is a point of infinite order.
	\item Let \begin{small}$\ell$\end{small} be of the form \begin{small}$\ell=s^6+3t^2$,\end{small} where \begin{small}$3 \nmid t$,\end{small}
	then \begin{small}$\big(\frac{3\ell}{s^2},\frac{9t\ell}{s^3}\big) \in E_{-432D^2}(\Q)$\end{small} is a point of infinite order.
	\item Let \begin{small}$\ell$\end{small} be of the form \begin{small}$\ell=s^6+27t^2$,\end{small}
	then \begin{small}$\big(\frac{3\ell}{s^2},\frac{27t\ell}{s^3}\big) \in E_{-432D^2}(\Q)$\end{small} is a point of infinite order.
\end{enumerate}
Thus in all the cases above, \begin{small}$\text{rk } E_{-432D^2}(\Q) \ge 1$\end{small} and \begin{small}$2\ell$\end{small} is a rational cube sum.
\end{rem}

\begin{rem}
In several cases of  Corollaries \ref{cortosylvester23} - \ref{cortosylvester2psqr}, we had rk \begin{small}$E_{a}(\Q) =0$\end{small} and \begin{small}$\Sh({E_a}/\Q)[3]=0$.\end{small} Using Lemma \ref{ono} we also obtain \begin{small}$\Sh({E_a}/K)[3]=0$\end{small} in each of these cases.
\end{rem}

\subsection{Positive proportion of $E_a$'s have $3$-Selmer rank  over $\Q $ at most $1$}\label{positiveprop}
In this subsection, we provide explicit families of the curves \begin{small}$E_a$\end{small} to show that a positive proportion of them have $3$-Selmer rank over \begin{small}$\Q$\end{small} equal to $0$ (respectively $1$). Fix a positive square-free integer \begin{small}$m$\end{small} and a positive integer \begin{small}$N$.\end{small} If  
 \begin{small}$$\liminf_{X \to \infty} \frac{\# \{\text{non-isomorphic  curves $E_a$ with }a \text{ square-free, } a  \equiv m\pmod{N},  Sel^3(E_a/\Q)=0  \text{ and }  0 <a < X  \}}{\# \{\text{non-isomorphic  curves $E_a$ with } a \text{ square-free, } a \equiv m \pmod{N}  \text{ and }  0 <a < X\} }  \ge \frac{1}{2},$$\end{small}
then we say that at least $50\%$ of the elliptic curves \begin{small}$E_a$,\end{small} where \begin{small}$a$\end{small} varies over positive, square-free integers congruent to \begin{small}$m\pmod {N}$\end{small} have \begin{small}$Sel^3(E_a/\Q)=0$.\end{small} Keeping a similar convention in mind, we have the following proposition:

\begin{proposition}\label{prop4.1}
As \begin{small}$k$\end{small} varies over the set of positive square-free  integers, we have that \begin{small}$Sel^3(E_a/\Q)=0$\end{small} for at least $50\%$ of each of the  following families of elliptic curves \begin{small}$E_a$.\end{small}
\begin{small}\begin{center}
\begin{tabular}{l l}
 (a) $a=k$, $k \equiv 7,14, 23,34 \pmod{36}$ & (b) $a=-k$, $k \equiv 1,5,6, 10,14, 17, 25, 26, 33, 34 \pmod{36}$  \\ 
 (c) $a=2^2k$, $k \equiv  22,26,31,35\pmod{36}$ & (d) $a=-2^2k$, $k \equiv  2,6,13,17, 22, 25, 26, 29, 33, 34\pmod{36}$  \\  
 (e) $a=2^4k$, $k \equiv 1,2 \pmod 9$ &  (f) $a=-2^4k$, $k \equiv 1,2, 4, 5,6 \pmod 9$\\
  & (g)  $a=- 3^2k$, $k \equiv 1,10, 13, 22, 25, 34 \pmod{36}$.  
\end{tabular}
\end{center}\end{small}
\end{proposition}

\begin{proof}
    One checks that in all the above mentioned cases \begin{small}$S_a = \emptyset$\end{small} and the global root number \begin{small}$\omega(E_a/\Q) = +1$.\end{small} Let \begin{small}$K^{i} \neq \Q(\zeta_3)$\end{small} be the imaginary quadratic subfield of \begin{small}$L_a:=\Q(\sqrt{a}, \sqrt{-3a})$.\end{small} If \begin{small}$K^{r}$\end{small} is the real quadratic subfield of \begin{small}$L_a$\end{small} then by \cite{her}, we have \begin{small}$h^3_{L_a}=h^3_{K^{i}}+h^3_{K^{r}}$.\end{small}  Moreover, by Scholz reflection principle \cite{sc}, if \begin{small}$h^3_{K^{i}}=0$\end{small} then \begin{small}$h^3_{L_a} =0$.\end{small} By a result of Nakagawa-Horie (which can be conveniently found in \cite[Lemma 2.2]{by}), we see that in each of the above cases for at least $50\%$ of \begin{small}$k$,\end{small} we have \begin{small}$h^3_{K^{i}}=0$.\end{small} Hence, for at least $50\%$ of \begin{small}$a$,\end{small} we have \begin{small}$h^3_{L_a} =0$\end{small} and the result follows from Theorem \ref{type1bounds}.
\end{proof}

As a consequence, we have:
\begin{theorem}\label{dens1} A positive proportion of the elliptic curves \begin{small}$E_a:y^2=x^3+a$\end{small} has \begin{small}$Sel^3(E_a/\Q)=0$.\end{small} More precisely,
 \begin{small}$$\liminf_{X \to \infty} \frac{\# \{\text{non-isomorphic elliptic curves $E_a$ with } Sel^3(E_a/\Q)=0 \text{ and }  0 <a < X\} }{\# \{\text{non-isomorphic elliptic curves $E_a$ with }0 <a < X\}}  > 0.075986 .$$
 $$ \liminf_{X \to \infty} \frac{\# \{\text{non-isomorphic elliptic curves $E_a$ with } Sel^3(E_a/\Q)=0 \text{ and }  -X <a < 0\} }{\# \{\text{non-isomorphic elliptic curves $E_a$ with } -X <a < 0\} }  > 0.184294.$$ \end{small}
\end{theorem}

\begin{proof}
 By Remark \ref{rmkSel3overQ}, it follows that if \begin{small}$Sel^{\phi}(E_a/K) = 0$,\end{small} then both  \begin{small}$Sel^3(E_a/\Q)$\end{small} and \begin{small}$Sel^3(E_{-27a}/\Q)$\end{small} are $0$. Since \begin{small}$$\sum_{\substack{0< n \le X \\ n \equiv m \pmod{M}}} \mu^2(n) = \frac{X}{M \zeta(2)} \Big(\prod_{p \mid \frac{M}{d}} \frac{1}{1- p^{-2} } \Big) + o(X), \text{ where } d=gcd(m,M), \text{ and } d \text{ is sqaure free} $$\end{small} it follows that number of  square-free integers \begin{small}$0 < a =k \le X$,\end{small} with \begin{small}$k \equiv 7,14,23,34 \pmod{36}$\end{small} is of the order \begin{small}$\frac{4X}{36 \zeta(2)} \frac{36}{24} = \frac{X}{6 \zeta(2)}$.\end{small} Hence, the number of elliptic curves \begin{small}$E_a$\end{small} with \begin{small}$Sel^3(E_a/\Q) = 0$,\end{small} with \begin{small}$a$\end{small} satisfying condition (a) of Proposition \ref{prop4.1} is at least \begin{small}$\frac{1}{2} \frac{X}{6 \zeta(2)}$.\end{small} Similarly, the number of square free integers \begin{small}$0 < a = 27k \le X$,\end{small} with  \begin{small}$k \equiv 1,5,6,10,14,17,25,26,33,34 \pmod{36}$\end{small} is of the order \begin{small}$\frac{10X}{27* 24 \zeta(2)}$,\end{small} and the number of elliptic curve with \begin{small}$Sel^3(E_{-27a}/\Q) = 0$,\end{small} with \begin{small}$a$\end{small} satisfying condition (b) of Proposition \ref{prop4.1} is at least \begin{small}$\frac{1}{2} \frac{5X}{12*27 \zeta(2)}$.\end{small} Proceeding in this way, we see that the number of non-isomorphic elliptic curves \begin{small}$E_a$\end{small} with \begin{small}$Sel^3(E_a/\Q)=0$\end{small} and  \begin{small}$0 < a \le X$\end{small} is at least \begin{small}$\frac{1}{2} \frac{X}{\zeta(2)} (\frac{1}{6} + \frac{5}{12*27} + \frac{1}{24} + \frac{5}{4*12*27} + \frac{1}{64}  + \frac{5}{8*16*27} + \frac{1}{4*9*27} ) = \frac{7643}{62208} \frac{X}{\zeta(2)}$.\end{small} Since the number of sixth power free integers positive integers up to \begin{small}$X$\end{small} is of the order \begin{small}$\frac{X}{\zeta(6)}$,\end{small} it follows that
 \begin{small}$$ \liminf_{X \to \infty} \frac{\# \{ \text{non-isomorphic elliptic curves $E_a$ with } Sel^3(E_a/\Q)=0 \text{ and }  0 <a < X  \}}{\# \{ \text{non-isomorphic elliptic curves $E_a$ with }0 <a < X \}}  \ge \frac{7643}{62208} \frac{\zeta(6)}{\zeta(2)} > 0.075986 .$$\end{small}

 The proof in the other case is similar.
\end{proof}

\begin{proposition}\label{den12}
As \begin{small}$k$\end{small} varies over the set of positive square-free  integers, we have that \begin{small}$\dim_{\F_3}Sel^3(E_a/\Q)=1$\end{small} for at least $83.33\%$ of each of the following families of elliptic curves \begin{small}$E_a$.\end{small}
\begin{small}\begin{center}
\begin{tabular}{ l l }
 (a) $a=k$, $k \equiv 2, 3, 10, 11, 19, 22, 26, 30, 31, 35 \pmod{36}$ & (b) $a=-k$, $k \equiv 2,13,22, 29 \pmod{36}$  \\ 
 (c) $a=2^2k$, $k \equiv 2, 3, 7, 10, 11, 14, 19, 23, 30, 34\pmod{36}$ & (d) $a=-2^2k$, $k \equiv 1, 5,10,14 \pmod{36}$  \\  
 (e) $a=2^4k$, $k \equiv 3,4,5,7,8 \pmod 9$ &  (f) $a=-2^4k$, $k \equiv 7,8 \pmod 9$  \\
 (g) $a=3^2k$, $k \equiv 2,11, 14, 23, 26, 35 \pmod{36}$. &
\end{tabular}
\end{center}\end{small}
\end{proposition}

\begin{proof}
    One checks that in all the above mentioned cases \begin{small}$S_a = \emptyset$\end{small} and the global root number \begin{small}$\omega(E_a/\Q) = -1$.\end{small} Let \begin{small}$K^{r}$\end{small} and \begin{small}$K^{i} \neq K$\end{small} be the real and imaginary quadratic subfields of \begin{small}$L_a:= \Q(\sqrt{a}, \sqrt{-3a})$,\end{small} respectively. Recall that  \begin{small}$h^3_{L_a}=h^3_{K^{r}}+h^3_{K^{i}}$\end{small} by \cite{her}. Also by Scholz reflection principle, if \begin{small}$h^3_{K^{r}}=0$\end{small} then \begin{small}$h^3_{L_a} \le 1$.\end{small} Then from Theorem \ref{type1bounds}, it follows that \begin{small}$\dim_{\F_3} \mathrm{Sel}^\phi(E_a/K) =1$.\end{small} Further, by Remark \ref{rmkSel3overQ}, it follows that if \begin{small}$\dim_{\F_3} Sel^{\phi}(E_a/K) = 1$,\end{small} then  \begin{small}$\dim_{\F_3} Sel^3(E_a/\Q) \le 1$\end{small} and hence by the $3$-parity conjecture it follows that \begin{small}$\dim_{\F_3} \mathrm{Sel}^3(E_a/\Q) = 1$.\end{small} Now applying  a result of Nakagawa-Horie (cf. \cite[Lemma 2.2]{by}), we see that in each of the cases (a)-(g), for at least $5/6$-th of \begin{small}$k$,\end{small} we have \begin{small}$h^3_{K^{r}}=0$,\end{small} the result follows.
\end{proof}

Using Proposition \ref{den12} (instead of Proposition \ref{prop4.1}) and arguing as in the proof of Theorem \ref{dens1}, we deduce that a positive proportion of the curves \begin{small}$E_a$\end{small} have $3$-Selmer rank  over \begin{small}$\Q$\end{small} equal to $1$.

\begin{theorem}\label{dens2} A positive proportion of  elliptic curves \begin{small}$E_a$\end{small} has \begin{small}$\dim_{\F_3} Sel^3(E_a/\Q)=1$.\end{small} More precisely,
\begin{small}$$ \liminf_{X \to \infty} \frac{\#\{\text{non-isomorphic elliptic curves $E_a$ with } \dim_{\F_3} Sel^3(E_a/\Q)=1 \text{ and }  0 <a < X\} }{\# \{\text{non-isomorphic elliptic curves $E_a$ with }0 <a < X\}}  > 0.307157 .$$
 $$ \liminf_{X \to \infty} \frac{\# \{\text{non-isomorphic elliptic curves $E_a$ with } \dim_{\F_3} Sel^3(E_a/\Q)=1 \text{ and }  -X <a < 0 \}}{\# \{\text{non-isomorphic elliptic curves $E_a$ with } -X <a < 0 \}}  > 0.126613. $$\end{small} 
\end{theorem}

Combining the results of Theorem \ref{dens1} and Theorem \ref{dens2}, we obtain,

 \begin{corollary} At least $34.7025\%$ of the curves \begin{small}$E_a$\end{small} have \begin{small}$\dim_{\F_3} \Sh(E_a/\Q)[3] <2$.\end{small}  \qed
 \end{corollary}

\subsection{Cubic twists}\label{cubictwistsec}
Let \begin{small}$p \ge 5$\end{small} be a fixed prime.
Now we study the variation of \begin{small}$\phi$\end{small}-Selmer rank in the family \begin{small}$E_{16p\ell^2}$\end{small} of cubic twists by primes \begin{small}$\ell$\end{small} of \begin{small}$E_{16p}$.\end{small} The following result is of a similar flavor as \cite[Theorem~3.8]{spt}.

\begin{proposition}\label{cubictwistsof16p}
Let \begin{small}$p \ge 5$\end{small} be a (fixed) prime and \begin{small}$L= \Q(\zeta, \sqrt{p})$\end{small} be the bi-quadratic field.  We consider the family of elliptic curves \begin{small}$E_{16p\ell^2}$\end{small} obtained as cubic twists of the curve \begin{small}$E_{16p}$\end{small} by rational primes \begin{small}$\ell$.\end{small} The density of primes \begin{small}$\ell$\end{small} such that \begin{small}$\dim_{\F_3} {\rm Sel}^\phi(E_{16p\ell^2}/K) \in \{h^3_L, \ h^3_L+1\}$\end{small} is at least $1/4$.
\end{proposition}
\begin{proof}
Define the set \begin{small}$T_p:=\{\ell \in \Sigma_\Q : \ \ell \nmid 6p, \ \ell \equiv 1 \pmod 3 \ \text{ and } \ \big(\frac{p}{\ell}\big)=-1\}$.\end{small} 
Then \begin{small}$\ell \in T_p$\end{small} satisfies either \begin{small}$\ell \equiv 1 \pmod{12}$\end{small} or \begin{small}$\ell \equiv 7 \pmod{12}$.\end{small} If \begin{small}$\ell \equiv 1 \pmod{12}$,\end{small} then we know that \begin{small}$\big(\frac{p}{\ell}\big)=-1$\end{small} if and only if \begin{small}$\big(\frac{\ell}{p}\big)=-1$.\end{small}
It follows that the density of primes \begin{small}$\ell$\end{small} such that \begin{small}$\ell \equiv 1 \pmod{12}$\end{small} and \begin{small}$\big(\frac{p}{\ell}\big)=-1$\end{small} is $1/8$. Similarly, the density of primes \begin{small}$\ell$\end{small} such that \begin{small}$\ell \equiv 7 \pmod{12}$\end{small} and \begin{small}$\big(\frac{p}{\ell}\big)=-1$\end{small} is again $1/8$. We deduce that the density of the set \begin{small}$T_p$\end{small} is $1/4$.
Now, observe that \begin{small}$S_{16p}=\emptyset$\end{small} and \begin{small}$S_{16p\ell^2}=S_{16p} \sqcup \{\mathfrak{l} \in \Sigma_K : \mathfrak{l} \mid \ell \text{ and } p \in K_{\mathfrak{l}}^{*2}\}$.\end{small} Hence, for all \begin{small}$\ell \in T_p$,\end{small} one has \begin{small}$S_{16p\ell^2}=\emptyset$.\end{small} So by Theorem \ref{type1bounds}, \begin{small}$\dim_{\F_3}{\rm Sel}^\phi(E_{16p\ell^2}/K) \in \{h^3_L, \  h^3_L+1\}.$\end{small}
\end{proof}

\begin{rem}\label{gencubictwistbound}
    In the general case, take \begin{small}$a \notin K^{*2}$\end{small} and consider the cubic twists of \begin{small}$E_a$\end{small} by rational primes \begin{small}$\ell$.\end{small} Then one can show that there is a set of primes \begin{small}$\ell$\end{small} of positive density such that \begin{small}$\dim_{\F_3} {\rm Sel}^\phi(E_{a\ell^2}/K) \in \{h^3_{S_a(L)}, \ h^3_{S_a(L)}+|S_a(L)|+2\}$,\end{small} where \begin{small}$L= \Q(\zeta, \sqrt{a})$.\end{small}
\end{rem}

\section*{Tables}
Table \ref{tab:type1examples}  contains  the bounds on \begin{small}$\dim_{\F_3}{\rm Sel}^\phi(E_a/K)$\end{small} and \begin{small}$\dim_{\F_3}{\rm Sel}^3(E_a/K)$\end{small}, computed  using SageMath/Magma for the elliptic curves \begin{small}$E_a$\end{small}. We define \begin{small}$s_l^\phi$\end{small} and \begin{small}$s_u^\phi$\end{small} to be the lower and upper bounds on the \begin{small}$\text{Sel}^\phi(E_a/K)$\end{small}, respectively, and \begin{small}$s_l^3$\end{small} and \begin{small}$s_u^3$\end{small} to be the lower and upper bounds on the \begin{small}$\text{Sel}^3(E_a/K)$\end{small}, respectively, following Theorem \ref{type1bounds} and Corollary \ref{corforsel3}. Also, $r$ denotes \begin{small}$\text{rk } E_{a}(K)$.\end{small}

\begin{scriptsize}
\begin{table}[h]
    \centering
		\begin{tabular}{ |c||c|c|c|c|c|c|c|c|c|c| }
			\hline
			\multirow{2}{*}{$a$} & \multirow{2}{*}{$S_a$} & \multirow{2}{*}{$S_a(\Q)$} & \multirow{2}{*}{$S_{a\alpha^2}(\Q)$} & \multirow{2}{*}{$|S_a(L)|$} & \multirow{2}{*}{$Cl_{S_a(L)}(L)$} & \multirow{2}{*}{$r$} & \multirow{2}{*}{$s^\phi_l$} & \multirow{2}{*}{$s^\phi_u$} & \multirow{2}{*}{$s^3_l$} & \multirow{2}{*}{$s^3_u$}\\
			& & & & & & & & & &\\
			\hline\hline
			$5$ & $\{2\OO_K\}$ & $\{2\}$ & $\emptyset$ & $2$ & trivial & $2$ & $0$ & $3$  & $2$ & $6$\\
			\hline
			$15$ & $\{\p\}$ & $\{3\}$ & $\emptyset$ & $2$ & trivial & $4$ & $0$ & $3$  & $4$ & $6$\\
			\hline
			$29$ & $\{2\OO_K\}$ & $\{2\}$ & $\emptyset$ & $2$ & trivial & $0$ & $0$ & $3$  & $0$ & $6$\\
			\hline
			$33$ & $\{2\OO_K,\p\}$ & $\{3\}$ & $\{2\}$ & $4$ & trivial & $2$ & $0$ & $5$  & $2$ & $10$\\
			\hline
			$77$ & $\{2\OO_K\}$ & $\{2\}$ & $\emptyset$ & $2$ & $\Z/2$ & $2$ & $0$ & $3$  & $2$ & $6$\\
			\hline
			$85$ & $\{2\OO_K\}$ & $\{2\}$ & $\emptyset$ & $2$ & $\Z/2$ & $0$ & $0$ & $3$  & $0$ & $6$ \\
			\hline
			$1025$ & $\{2\OO_K,5\OO_K\}$ & $\emptyset$ & $\{2,5\}$ & $4$ & trivial & $6$ & $0$ & $5$ & $6$ & $10$ \\
			\hline
			$1373$ & $\{2\OO_K\}$ & $\{2\}$ & $\emptyset$ & $2$ & $\Z/3 \times \Z/3$ & $6$ & $2$ & $5$  & $6$ & $10$   \\
			\hline
			$58$ & $\emptyset$ & $\emptyset$ & $\emptyset$ & $0$ & $\Z/12$ & $2$ & $1$ & $2$  & $2$ & $4$ \\
			\hline
			$62$ & $\emptyset$ & $\emptyset$ & $\emptyset$ & $0$ & $\Z/6$ & $2$ & $1$ & $2$  & $2$ & $4$ \\
			\hline
			$67$ & $\emptyset$ & $\emptyset$ & $\emptyset$ & $0$ & $\Z/6$ & $2$ & $1$ & $2$  & $2$ & $4$ \\
			\hline
			$74$ & $\emptyset$ & $\emptyset$ & $\emptyset$ & $0$ & $\Z/6 \times \Z/2$ & $2$ & $1$ & $2$  & $2$ & $4$ \\
			\hline
			$79$ & $\emptyset$ & $\emptyset$ & $\emptyset$ & $0$ & $\Z/6 \times \Z/3$ & $4$ & $2$ & $3$  & $4$ & $6$ \\
			\hline
			$82$ & $\emptyset$ & $\emptyset$ & $\emptyset$ & $0$ & $\Z/24$ & $2$ & $1$ & $2$  & $2$ & $4$ \\
			\hline
			$83$ & $\emptyset$ & $\emptyset$ & $\emptyset$ & $0$ & $\Z/6$ & $2$ & $1$ & $2$  & $2$ & $4$ \\
			\hline
			$142$ & $\emptyset$ & $\emptyset$ & $\emptyset$ & $0$ & $\Z/12 \times \Z/3$ & $4$ & $2$ & $3$  & $4$ & $6$ \\
			\hline
			$235$ & $\emptyset$ & $\emptyset$ & $\emptyset$ & $0$ & $\Z/12 \times \Z/3$ & $2$ & $2$ & $3$  & $2$ & $6$\\
			\hline
			$254$ & $\emptyset$ & $\emptyset$ & $\emptyset$ & $0$ & $\Z/6 \times \Z/3$ & $2$ & $2$ & $3$  & $2$ & $6$ \\
			\hline
			$326$ & $\emptyset$ & $\emptyset$ & $\emptyset$ & $0$ & $\Z/12 \times \Z/3 $ & $2$ & $2$ & $3$  & $2$ & $6$ \\
			\hline
			$568$ & $\emptyset$ & $\emptyset$ & $\emptyset$ & $0$ & $\Z/12 \times \Z/3 $ & $6$ & $2$ & $3$  & $6$ & $6$ \\
			\hline
			$574$ & $\emptyset$ & $\emptyset$ & $\emptyset$ & $0$ & $\Z/12 \times \Z/6 $ & $4$ & $2$ & $3$  & $4$ & $6$ \\
			\hline
			$730$ & $\emptyset$ & $\emptyset$ & $\emptyset$ & $0$ & $\Z/12 \times \Z/6 \times \Z/2 $ & $6$ & $2$ & $3$  & $6$ & $6$  \\
			\hline
			$842$ & $\emptyset$ & $\emptyset$ & $\emptyset$ & $0$ & $\Z/30 \times \Z/6 $ & $4$ & $2$ & $3$  & $4$ & $6$  \\
			\hline
			$874$ & $\emptyset$ & $\emptyset$ & $\emptyset$ & $0$ & $\Z/12 \times \Z/6 $ & $2$ & $2$ & $3$  & $2$ & $6$  \\
			\hline
			$892$ & $\emptyset$ & $\emptyset$ & $\emptyset$ & $0$ & $\Z/6 \times \Z/3 $ & $6$ & $2$ & $3$  & $6$ & $6$  \\
			\hline
			$895$ & $\emptyset$ & $\emptyset$ & $\emptyset$ & $0$ & $\Z/12 \times \Z/6 $ & $2$ & $2$ & $3$  & $2$ & $6$  \\
			\hline
			$898$ & $\emptyset$ & $\emptyset$ & $\emptyset$ & $0$ & $\Z/60 \times \Z/3 $ & $4$ & $2$ & $3$  & $4$ & $6$  \\
			\hline
			$899$ & $\emptyset$ & $\emptyset$ & $\emptyset$ & $0$ & $\Z/12 \times \Z/6 $ & $6$ & $2$ & $3$  & $6$ & $6$  \\
			\hline
			$934$ & $\emptyset$ & $\emptyset$ & $\emptyset$ & $0$ & $\Z/12 \times \Z/3 $ & $4$ & $2$ & $3$  & $4$ & $6$  \\
			\hline
			$1090$ & $\emptyset$ & $\emptyset$ & $\emptyset$ & $0$ & $\Z/12 \times \Z/6 \times \Z/2$ & $6$ & $2$ & $3$  & $6$ & $6$ \\
		  \hline
		  $1264$ & $\emptyset$ & $\emptyset$ & $\emptyset$ & $0$ & $\Z/6 \times \Z/3$ & $2$ & $2$ & $3$  & $2$ & $6$   \\
		  \hline
			$1339$ & $\emptyset$ & $\emptyset$ & $\emptyset$ & $0$ & $\Z/12 \times \Z/6 \times \Z/2$ & $4$ & $2$ & $3$  & $4$ & $6$   \\
			\hline
			$1342$ & $\emptyset$ & $\emptyset$ & $\emptyset$ & $0$ & $\Z/12 \times \Z/6 \times \Z/2$ & $6$ & $2$ & $3$  & $6$ & $6$  \\
			\hline
			$1384$ & $\emptyset$ & $\emptyset$ & $\emptyset$ & $0$ & $\Z/12 \times \Z/3$ & $4$ & $2$ & $3$  & $4$ & $6$   \\
			\hline
			$1436$ & $\emptyset$ & $\emptyset$ & $\emptyset$ & $0$ & $\Z/12 \times \Z/3$ & $4$ & $2$ & $3$  & $4$ & $6$   \\
			\hline
			$1714$ & $\emptyset$ & $\emptyset$ & $\emptyset$ & $0$ & $\Z/24 \times \Z/3 \times \Z/3$ & $2$ & $3$ & $4$ & $3$ & $8$  \\
		\hline
		$2230$ & $\emptyset$ & $\emptyset$ & $\emptyset$ & $0$ & $\Z/12 \times \Z/6 \times \Z/3$ & $0$ & $3$ & $4$ & $3$ & $8$  \\
		\hline
		$2263$ & $\emptyset$ & $\emptyset$ & $\emptyset$ & $0$ & $\Z/6 \times \Z/6 \times \Z/6$ & $2$ & $3$ & $4$ & $3$ & $8$ \\
		\hline
		$2659$ & $\emptyset$ & $\emptyset$ & $\emptyset$ & $0$ & $\Z/6 \times \Z/3 \times \Z/3$ & $2$ & $3$ & $4$ & $3$ & $8$  \\
		\hline
		$3023$ & $\emptyset$ & $\emptyset$ & $\emptyset$ & $0$ & $\Z/12 \times \Z/3 \times \Z/3$ & $2$ & $3$ & $4$ & $3$ & $8$ \\
		\hline
		$3391$ & $\emptyset$ & $\emptyset$ & $\emptyset$ & $0$ & $\Z/6 \times \Z/3 \times \Z/3$ & $8$ & $3$ & $4$ & $8$ & $8$  \\
		\hline
		$3667$ & $\emptyset$ & $\emptyset$ & $\emptyset$ & $0$ & $\Z/6 \times \Z/6 \times \Z/6$ & $2$ & $3$ & $4$ & $3$ & $8$  \\
		\hline
		$4094$ & $\emptyset$ & $\emptyset$ & $\emptyset$ & $0$ & $\Z/12 \times \Z/6 \times \Z/3$ & $6$ & $3$ & $4$ & $6$ & $8$ \\
		\hline
		$4151$ & $\emptyset$ & $\emptyset$ & $\emptyset$ & $0$ & $\Z/12 \times \Z/6 \times \Z/3$ & $2$ & $3$ & $4$ & $3$ & $8$ \\
		\hline
		$4279$ & $\emptyset$ & $\emptyset$ & $\emptyset$ & $0$ & $\Z/12 \times \Z/6 \times \Z/3$ & $6$ & $3$ & $4$ & $6$ & $8$  \\
		\hline
		$4778$ & $\emptyset$ & $\emptyset$ & $\emptyset$ & $0$ & $\Z/18 \times \Z/6 \times \Z/3$ & $2$ & $3$ & $4$ & $3$ & $8$  \\
		\hline
		$4910$ & $\emptyset$ & $\emptyset$ & $\emptyset$ & $0$ & $\Z/12 \times \Z/6 \times \Z/3$ & $8$ & $3$ & $4$ & $8$ & $8$  \\
		\hline
		$43063$ & $\emptyset$ & $\emptyset$ & $\emptyset$ & $0$ & $\Z/42 \times \Z/3 \times \Z/3 \times \Z/3$ & $4$ & $4$ & $5$ & $4$ & $10$  \\
		\hline
		$51694$ & $\emptyset$ & $\emptyset$ & $\emptyset$ & $0$ & $\Z/24 \times \Z/3 \times \Z/3 \times \Z/3$ & $0$ & $4$ & $5$ & $4$ & $10$  \\
		\hline
		$53507$ & $\emptyset$ & $\emptyset$ & $\emptyset$ & $0$ & $\Z/42 \times \Z/3 \times \Z/3 \times \Z/3$ & $2$ & $4$ & $5$ & $4$ & $10$  \\
		\hline
		$53678$ & $\emptyset$ & $\emptyset$ & $\emptyset$ & $0$ & $\Z/54 \times \Z/3 \times \Z/3 \times \Z/3$ & $2$ & $4$ & $5$ & $4$ & $10$  \\
		\hline
		$529987$ & $\emptyset$ & $\emptyset$ & $\emptyset$ & $0$ & $\Z/210 \times \Z/3 \times \Z/3 \times \Z/3$ & $2$ & $4$ & $5$ & $4$ & $10$ \\
		\hline
		\end{tabular}
\caption{Bounds on the $\phi$- and $3$-Selmer groups of $E_a/K$}
\label{tab:type1examples}
\end{table}
\end{scriptsize}


\clearpage
\begin{thebibliography}{9999999999}
\begin{scriptsize}
\bibitem[ABS-BS]{abs} L. Alp\"oge, M. Bhargava \& A. Shnidman, Integers expressible as the sum of two rational cubes, {\it https://arxiv.org/abs/2210.10730} (2022), 52 pages.
\bibitem[Ba]{ban2} A. Bandini, $3$-Selmer groups for curves $y^2 = x^3 + a$, {\it Czechoslovak Math. J.} {\bf 58}(2) (2008), 429-445.
\bibitem[BES]{bes} M. Bhargava, N. Elkies \& A. Shnidman, The average size of the $3$-isogeny Selmer groups of elliptic curves $y^2 = x^3 + k$, {\it J. London Math. Soc.} {\bf 101}(1) (2020), 299-327.
\bibitem[BG]{bg} Y. Bilu \& J. Gillibert, Chevalley-Weil theorem and subgroups of class groups, {\it Israel J. Math.} {\bf 226}(2) (2018), 927-956.
\bibitem[BKOS]{bkos} M. Bhargava, Z. Klagsbrun, R. Oliver \& A. Shnidman, Elements of given order in Tate–Shafarevich groups of abelian varieties in quadratic twist families, {\it Algebra Number Theory} {\bf 15}(3) (2021), 627-655.
\bibitem[By]{by} D. Byeon, Class numbers of quadratic fields $\Q(\sqrt{D})$ and
              $\Q(\sqrt{tD})$, {\it Proc. Amer. Math. Soc.} {\bf 132}(11) (2004), 3137--3140. 
\bibitem[Ca1]{cass3} J. W. S. Cassels, Arithmetic on curves of genus 1. VI. The Tate-{\v S}afarevi{\v c} group can be arbitrarily large, {\it J. Reine Angw. Math.} {\bf 214/215} (1964), 65-70.
\bibitem[Ca2]{cass2} J. W. S. Cassels, Arithmetic on curves of genus 1. VIII. On conjectures of Birch and Swinnerton-Dyer, {\it J. Reine Angw. Math.} {\bf 217} (1965), 180-199.
\bibitem[Ca3]{cass} J. W. S. Cassels, Lectures on elliptic curves, {\it LMS Student Texts} {\bf 24}, Cambridge Uni. Press (1991).
\bibitem[\v Ce]{ces} K. {\v C}esnavi{\v c}ius, Selmer groups and class groups, {\it Compositio Math.} {\bf 151} (2015), 416-434.
\bibitem[CST]{cst} L. Cai, J. Shu \& Y. Tian, Cube sum problem and an explicit Gross-Zagier formula, {\it Amer. J. Math.} {\bf 139}(3) (2017), 785-816.
\bibitem[DV]{dv} S. Dasgupta \& J. Voight, Heegner points and Sylvester’s conjecture, {\it Arithmetic geometry, Clay Math. Proc., Amer. Math. Soc., Providence} {\bf 8} (2009), 91–102.
\bibitem[Gr]{gras} G. Gras, Class Field Theory: From Theory to Practice, {\it Springer Monograph in Math.}, Springer-Verlag Berlin Heidelberg GmbH (2003).
\bibitem[He]{her} G. Hergoltz, {\" U}ber einen Dirichletschen Satz, {\it Math. Zeitschrift} {\bf 12} (1922), 255—261.
\bibitem[JMS]{jms} S. Jha, D. Majumdar \& B. Sury, Binary cubic forms and rational cube sum problem, {\it https://arxiv.org/abs/2301.06970} (2023).
\bibitem[KL]{kl} Y. Kezuka \& Y. Li, A classical family of elliptic curves having rank one and the $2$-primary part of their Tate–Shafarevich group non-trivial, {\it Documenta Math.} {\bf 25} (2020), 2115-2147. 
\bibitem[KS]{ks} R. Kloosterman \& E. Schaefer, Selmer groups of elliptic curves that can be arbitrarily large, {\it J. Number Theory} {\bf 99} (2003), 148-163.
\bibitem[Le]{lemm} F. Lemmermeyer, Reciprocity Laws: from Euler to Eisenstein, {\it Springer Monograph in Math.}, Springer-Verlag Berlin Heidelberg GmbH (2000).
\bibitem[Li]{li} C. Li, $2$-Selmer groups, $2$-class groups and rational points on elliptic curves, {\it Trans. of AMS} {\bf 371} (2019), 4631-4653.
\bibitem[Lie]{lie} D. Lieman, Nonvanishing of $L$-Series associated to cubic twists of elliptic curves, {\it Ann. of Math.} {\bf 140}(1) 1994, 81-108.
\bibitem[Liv]{liv} E. Liverance, A formula for the root number of a family of elliptic curves, {\it J. Number Theory} {\bf 51} (1995), 288-305.
\bibitem[MR]{mr} B. Mazur \& K. Rubin, Ranks of twists of elliptic curves and Hilbert’s tenth problem, {\it Invent. Math.} {\bf 181} (2010), 541–575.
\bibitem[MS]{ms} D. Majumdar \& B. Sury, Cyclic cubic extensions of $\Q$, {\it International J. Number Theory}, {\bf 18}(no.1), 2022, 1929-1955.
\bibitem[Ne]{nek} J. Nekov{\'a}{\v r}, Class numbers of quadratic fields and Shimura’s correspondence, {\it Math. Annalen} {\bf 287} (1990), 577–594.
\bibitem[OP]{op} K. Ono \& M. A. Papanikolas, Quadratic twists of modular forms and elliptic curves, {\it Number Theory for the Millennium III (Urbana, IL, 2000), A. K. Peters} (2002), 73-85.
\bibitem[PS]{ps} D. Prasad \& S. Shekhar, Relating the Tate–Shafarevich group of an elliptic curve with the class group, {\it Pacific Journal of Math.} {\bf 312} (no. 1), 2021, 203-218.
\bibitem[RZ]{rz} F. R. Villegas \& D. Zagier, Which primes are sums of two cubes? {\it Number Theory (Halifax, NS, 1994), CMS Conference Proceedings, Amer. Math. Soc., Providence} {\bf 15} (1995), 295–306.
\bibitem[Sa]{sa} P. Satg\'e, Un analogue du calcul de Heegner, {\it Invent. Math.} {\bf 87}(2) (1987), 425–439.
\bibitem[Sc]{sch} E. Schaefer, Class groups and Selmer groups, {\it J. Number Theory} {\bf 56}(1) (1996), 79-114.
\bibitem[Sz]{sc} A. Scholz, {\"U}ber die Beziehung der Klassenzahlen quadratischer K{\"o}rper zueinander, {\it J. Reine Angw. Math.} {\bf 166}, (1932), 201-203.
\bibitem[Sl]{sel} E. Selmer, The rational solutions of the diophantine equation $\eta^2=\xi^3-D$ for $|D| \le 100$, {\it Math. Scand.} {\bf 4} (1956), 281-286. 
\bibitem[S1]{sil1} J. Silverman, The arithmetic of elliptic curves, {\it GTM} {\bf 106}, Springer-Verlag, New York (1986).
\bibitem[S2]{sil2} J. Silverman, Advanced topics in the arithmetic of elliptic curves, {\it GTM} {\bf 151}, Springer-Verlag, 1994.
\bibitem[SPT]{spt} D. B. Salazar, A. Pacetti and G. Tornaría,	On 2-Selmer groups and quadratic twists of elliptic curves, {\it Math. Research Letters}, to appear.
\bibitem[SS]{ss} E. Schaefer \& M. Stoll, How to do a $p$-descent on an elliptic curve, {\it Trans. of AMS} {\bf 356}(3) (2003), 1209–1231.
\bibitem[Syl]{syl} J. J. Sylvester, On certain ternary cubic-form equations, {\it Amer. J. Math.} {\bf 2}(4) (1879), 357–393.
\bibitem[SY]{sy} J. Shu \& H. Yin, Cube sums of the forms $3p$ and $3p^2$ II, {\it Math. Annalen} {\bf 385} (2023), 1037–1060.
\bibitem[To]{top} J. Top, Descent by $3$-isogeny and $3$-rank of quadratic fields, {\it Advances in Number Theory, Proc. of Conf. in Kingston} (1991).
\bibitem[V{\'e}]{velu}  J. V{\'e}lu, Isog{\'e}nies entre courbes elliptiques, {\it C. R. Acad. Sci. Paris} {\bf 273} (1971), 238-241. 
\bibitem[Wa]{was} L. C. Washington, Introduction to cyclotomic fields, {\it GTM} {\bf 83} Springer-Verlag, New York (2012).
\bibitem[Yi]{yi} H. Yin, On the  case $8$ of the Sylvester conjecture, {\it Trans. of AMS} {\bf 375} (2022), 2705-2728.
\end{scriptsize}
\end{thebibliography}
\end{document}